\documentclass[usenames,dvipsnames,a4paper,reqno]{amsart}
\usepackage{amsfonts}
\usepackage{amsmath}
\usepackage{amssymb}
\usepackage{amsthm}
\usepackage{bbm}
\usepackage[tableposition=above]{caption}
\usepackage{mathrsfs}
\usepackage{mathtools}
\usepackage{enumerate}
\usepackage{enumitem}
\usepackage{graphicx}
\usepackage{color}
\usepackage{pst-node}
\usepackage{tikz-cd} 
\usepackage{stmaryrd}
\usepackage{longtable}
\usepackage{multirow}
\usetikzlibrary{patterns, decorations}
\usepackage{changepage}
\usepackage{subcaption}
\usepackage{import}
\usepackage{hyperref}
\usepackage[margin=1.2in]{geometry}
\usetikzlibrary{arrows}

\usepackage[labelfont=bf]{caption}
 
\usepackage[labelfont=bf,textfont=normalfont,singlelinecheck=on]{subcaption}

\newtheorem{thm}{Theorem}[section]
\newtheorem{cor}[thm]{Corollary}

\newtheorem{prop}[thm]{Proposition}
\newtheorem{lemma}[thm]{Lemma}

\theoremstyle{definition}
\newtheorem{definition}[thm]{Definition}
\newtheorem{eg}[thm]{Example}
\newtheorem{remark}[thm]{Remark}
\newtheorem{notation}[thm]{Notation}

\newcommand{\N}{\mathbb{N}}
\newcommand{\Z}{\mathbb{Z}}
\newcommand{\Q}{\mathbb{Q}}

\newcommand{\PP}{\mathbb{P}}
\newcommand{\R}{\mathbb{R}}

\newcommand{\Frob}{\mathrm{Frob}}

\newcommand{\Qur}{\Q_p^{\mathrm{ur}}}

\newcommand{\X}{\mathscr{X}}
\newcommand{\Y}{\mathscr{Y}}
\newcommand{\ZZ}{\mathscr{Z}}

\newcommand{\Ccal}{\mathcal{C}}

\newcommand{\s}{\mathfrak{s}}
\newcommand{\OO}{\mathcal{O}}
\newcommand{\OOO}{\OO_K}
\newcommand{\tcqi}{tame cyclic quotient invariants}

\newcommand{\GK}{\mathrm{G}_K}
\newcommand{\clusleadcoeff}{\nu_{\s}}
\newcommand{\depth}{d_{\s}}

\newcommand{\tfrak}{\mathfrak{t}}

\newcommand{\RR}{\mathcal{R}}

\newcommand{\Rcal}{\mathcal{R}}

\newcommand{\Kbar}{\overline{K}}

\newcommand{\Gal}{\textrm{Gal}}
\newcommand{\pointInOrbit}{Q}
\newcommand\restr[2]{{
  \left.\kern-\nulldelimiterspace 
  #1 
  \vphantom{\big|} 
  \right|_{#2} 
  }}
\newcommand{\oo}{\mathfrak{o}}

\newcommand{\singletonsofs}{\s_{\mathrm{sing}}}
\newcommand{\Cnew}{C^{\mathrm{new}}}
\newcommand{\snew}{\s^{\mathrm{new}}}

\newcommand{\childrenfixedofs}[1]{{\widehat{#1}}^{\mathrm{f}}}
\newcommand{\childrennotfixedofs}[1]{{\widehat{#1}}^{\mathrm{nf}}}
\newcommand{\Rcalnew}{\Rcal^{\mathrm{new}}}
\newcommand{\ssg}[1]{g_{\textrm{ss}}(#1)}
\newcommand{\ssgs}{\ssg{\s}}

\title{Models of Hyperelliptic Curves with Tame Potentially Semistable Reduction}
\author{Omri Faraggi}
\address{University College London, Gower Street, London, WC1E 6BT, UK}
\email{omri.faraggi.17@ucl.ac.uk}

\author{Sarah Nowell}
\address{University College London, Gower Street, London, WC1E 6BT, UK}
\email{sarah.nowell.17@ucl.ac.uk}

\begin{document}
	
\maketitle

\begin{abstract}
Let $C$ be a hyperelliptic curve $y^2 = f(x)$ over a discretely valued field $K$. The $p$-adic distances between the roots of $f(x)$ can be described by a completely combinatorial object known as the cluster picture. We show that the cluster picture of $C$, along with the leading coefficient of $f$ and the action of $\Gal(\overline{K}/K)$ on the roots of $f$, completely determines the combinatorics of the special fibre of the minimal strict normal crossings model of $C$. In particular, we give an explicit description of the special fibre in terms of this data.
\end{abstract}

\renewcommand{\contentsname}{Table of contents}
\thispagestyle{empty}\setcounter{tocdepth}{1}
\tableofcontents

\section{Introduction}
\label{sec::intro}

Models of curves are invaluable objects which can be used to deduce a 
large amount of arithmetic information about the curve 
more easily than would otherwise be possible.
In this paper we study 
hyperelliptic curves
, giving a description of their minimal strict normal crossings (SNC) models using 
\textit{cluster pictures}, a relatively new innovation which have already proved 
advantageous in studying the arithmetic of hyperelliptic curves. In particular, cluster pictures have been used to calculate semistable models, conductors, minimal discriminants and Galois representations in \cite{DDMM18}, Tamagawa numbers in \cite{Bet18}, root numbers in \cite{Bis19} and differentials in \cite{Kun19}. 


Let $K$ be a field complete with respect to a discrete valuation $v_K$, with algebraically closed residue field $k$ of characteristic $p > 2$. Let $C/K$ be a hyperelliptic curve given by Weierstrass equation $y^2 = f(x),$ with genus $g=g(C)$\footnote{Unless explicitly mentioned otherwise, we assume $g\geq 2$ throughout the paper.}. We write $\Rcal$ for the set of roots of $f(x)$ in the algebraic closure $\Kbar$ of $K$ and $c_f$ for the leading coefficient of $f$, so
$$f(x)=c_f\prod_{r\in\Rcal}(x-r),$$
and $|\Rcal|\in\{2g+1, 2g+2\}$. Following \cite{DDMM18} we associate to $C$ a cluster picture, defined by the combinatorics of the root configuration of $f$.

Using cluster pictures we will calculate a combinatorial description of the minimal SNC model $\X$ of $C/K$: a model whose singularities on the special fibre $\X_k$ are normal crossings (i.e. locally they look like the union of two axes), and where blowing down any exceptional component of $\X_k$ would result in a worse singularity. Such models can be used to calculate arithmetic invariants, to study the Galois representation, and to deduce the existence of $K$-rational points of $C$. For the case of elliptic curves, Tate's algorithm \cite{Sil94} is sufficient to calculate the minimal SNC model of a given curve. For hyperelliptic curves, \cite{DDMM18} the authors calculate the SNC model when $C$ has semistable reduction, and in \cite{Dok18} when $C$ has a particularly nice cluster picture.\footnote{In fact, the methods of \cite{Dok18} work for a much larger class of smooth projective curves, but we restrict our attention to its applications for hyperelliptic curves.} Similar work has also been done on models of different classes of curves and the applications of these models --- such as \cite{BW17} on stable models of superelliptic curves and \cite{LLLGR18} on non-hyperelliptic genus 3 curves. Other work on hyperelliptic invariants has also been done in \cite{OS19}, where the authors prove a conductor-discriminant inequality for hyperelliptic curves. 

We extend existing results about models of hyperelliptic curves to the more general case where $C$ has tame potentially semistable reduction over $K$ --- that is, there exists some finite extension $L/K$ such that $C$ has semistable reduction over $L$, and $[L:K]$ is coprime to $p$. It is important to note that our theorems do not apply in the case where a wild extension is required for semistability. However this condition is not too strong since for large enough $p$, every curve of genus $g$ has tame potentially semistable reduction. Most of the information required to deduce the special fibre of $\X$ is contained in the cluster picture of $C$. 
\begin{definition}
\label{def::clusterintro}
    A \textit{cluster} is a non-empty subset $\s\subseteq\mathcal{R}$ of the form $\s=D\cap\mathcal{R}$ for some disc $D = z + \pi_K^n\mathcal{O}_{\overline{K}}$, where $z\in\overline{K}$, $n\in\Q$ and $\pi_K$ is a uniformiser of $K$. If $\s$ is a cluster and $|\s| > 1$, we say that $\s$ is a \emph{proper cluster}. For a proper cluster $\s$ 
    we define its \emph{depth} $d_\s$ to be 
    $$d_\s=\min_{r,r'\in\s}v_K(r-r').$$ 
    We write $d_{\s} = \frac{a_{\s}}{b_{\s}}$ with $a_{\s},b_{\s}$ coprime. The \textit{cluster picture} $\Sigma_{C/K}$ of $C$ is the collection of all clusters of the roots of $f$. When there is no risk of confusion, we may simplify this to $\Sigma_{C}$.
\end{definition}

The cluster picture $\Sigma_{C/K}$ comes with a natural action of $\GK=\Gal(\overline{K}/K)$ and, along with the valuation of the leading coefficient $v_K(c_f)$, this action is all we need to calculate a combinatorial description of the minimal SNC model of $C$.

\begin{thm}
\label{thm::main1intro}
    Let $K$ be a complete discretely valued field with algebraically closed residue field of characteristic $p>2$. Let $C:y^2 = f(x)$ be a hyperelliptic curve over $K$ with tame potentially semistable reduction. Then the dual graph, with genus and multiplicity, of the special fibre of the minimal SNC model of $C/K$ is completely determined by $\Sigma_{C/K}$ (with depths), the valuation of the leading coefficient $v_K(c_f)$ of $f$, and the action of $\GK$.
\end{thm}

\begin{remark}
    If $K$ does not have an algebraically closed residue field, then the Frobenius action is determined by the data in Theorem \ref{thm::main1intro}, as well as the values of some invariants $\epsilon_X{\Frob}$ for all orbits of clusters $X$. See Definition \ref{def::clusterfunctions2} for a definition of $\epsilon$ and Theorem \ref{thm::frobactionintro} for a full description of the Frobenius action.
\end{remark}

\begin{remark}
    In \cite{Bis19} the author classifies the possible cluster pictures which can arise from hyperelliptic curves with tame potentially semistable reduction. He also shows that the inertia action is determined by the cluster picture (with depths). Given time and determination, this fact, along with Theorem \ref{thm::main1intro}, allows us to classify the minimal SNC models which can arise from such hyperelliptic curves of a given genus. We do so for elliptic curves in Example \ref{eg::KNclassification}.
\end{remark}

A maximal subcluster $\s'$ of a cluster $\s$ is called a \textit{child} of $\s$, denoted $\s' < \s$, and $\s$ is the \textit{parent} of $\s'$, denoted $P(\s')$. We say $\s$ is \textit{odd} (resp. \textit{even}) if $|\s|$ is odd (resp. even) Furthermore, $\s$ is a \emph{twin} if $|\s|=2$, and $\s$ is \emph{\"ubereven} if $\s$ has only even children. A cluster $\s\neq\Rcal$ is \emph{principal} if $|\s| \geq 3$. The cluster $\Rcal$ is \textit{not} principal if it has a child of size $2g(C)$, or if $\Rcal$ is even and has exactly two children; otherwise $\Rcal$ is principal. The remaining theorems given in the introduction assume that $\Rcal$ is principal. Full theorems including the case when $\Rcal$ is not principal are given in Section \ref{sec::mainthm}.

Every Galois orbit of principal clusters $X$ contributes components to the special fibre $\X_k$. More precisely: orbits of principal, \"ubereven clusters contribute either one or two components and orbits of principal, non-\"ubereven clusters contribute one component. We call these components \emph{central components}, and they are linked by either one or two chains of rational curves which we call \emph{linking chains}. The central components of two orbits $X$ and $X'$ are linked by a chain (or chains) of rational curves if and only if there exits some $\s\in X$ and $\s'\in X'$ such that $\s'<\s$. Orbits of twins gives rise to a chain of rational curves which intersects the component(s) arising from their parent's orbit. Some central components are also intersected by other chains of rational curves: \textit{loops}, \textit{tails} and \textit{crossed tails}. Loops are chains from a component to itself; tails are chains which intersect the rest of the special fibre in only one place; crossed tails are similar to tails but with two additional components, called \textit{crosses}, intersecting the final component of the chain. Figures \ref{fig::linkingchainsandloops} and \ref{fig::tailsandcrossed tails} give pictorial descriptions of the different types of chains of rational curves that can occur, where the dashed lines illustrate all the components of $\X_k$ which are intersected by the chain.
\begin{figure}[h]
\centering
\begin{subfigure}{.5\textwidth}
  \centering
  \begin{tikzpicture}
    \draw[dashed,gray] (0,-0.5) -- ++ (0,-1.5);
    \draw (-0.2,-1) -- 
    ++ (1,0);
    \draw (0.6,-0.8) -- 
    ++ (0,-1);
    \draw (0.4,-1.6) -- ++ (1,0);
    \path (1.4,-1.6) -- node[auto=false]{\ldots} (2.1,-1.6);
    \draw (2.1,-1.6) -- 
    ++ (1,0);
    \draw[dashed,gray] (2.9,-0.6)  -- ++ (0,-1.5);
    \end{tikzpicture}
    \caption{Linking chain}
    \label{fig::linkingchain}
\end{subfigure}%
\begin{subfigure}{.5\textwidth}
  \centering
  \begin{tikzpicture}
    \draw[dashed,gray] (0,0) -- ++ (3.7,0);
    \draw (0.75,0.2) -- 
    ++ (0,-1);
    \draw (2.95,0.2)-- 
    ++ (0,-1);
    \draw(0.6,-0.5)-- ++ (1,-0.8);
    \draw(2.1,-1.3)-- ++ (1,0.8);
    \path (1.6,-1.15) -- node[auto=false]{\ldots} (2.1,-1.15);
    \end{tikzpicture}
    \caption{Loop}
    \label{fig::loop}
\end{subfigure}
\caption{Pictorial description of linking chains and loops.}
\label{fig::linkingchainsandloops}
\end{figure}
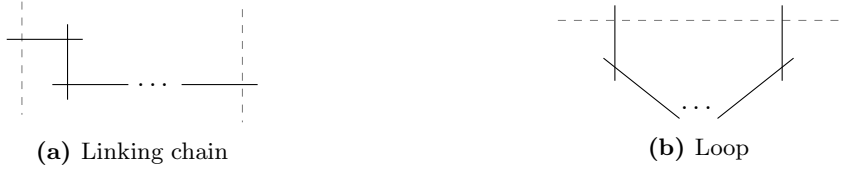
\vspace{-15px}
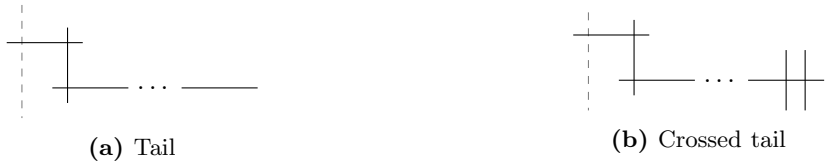
\begin{figure}[h]
\centering
\begin{subfigure}{.5\textwidth}
  \centering
  \begin{tikzpicture}
    \draw[dashed,gray] (0,-0.5) -- ++ (0,-1.5);
    \draw (-0.2,-1) -- 
    ++ (1,0);
    \draw (0.6,-0.8) -- 
    ++ (0,-1);
    \draw (0.4,-1.6) -- ++ (1,0);
    \path (1.4,-1.6) -- node[auto=false]{\ldots} (2.1,-1.6);
    \draw (2.1,-1.6) -- 
    ++ (1,0);
    \end{tikzpicture}
    \caption{Tail}
    \label{fig::tail}
\end{subfigure}%
\begin{subfigure}{.5\textwidth}
  \centering
  \begin{tikzpicture}
  \draw[dashed,gray] (0,-0.7) -- ++ (0,-1.3);
    \draw (-0.2,-1) -- 
    ++ (1,0);
    \draw (0.6,-0.8) -- 
    ++ (0,-1);
    \draw (0.4,-1.6) -- ++ (1,0);
    \path (1.4,-1.6) -- node[auto=false]{\ldots} (2.1,-1.6);
    \draw (2.1,-1.6) -- 
    ++ (1,0);
    \draw (2.6,-1.2) -- 
    ++ (0,-0.8);
    \draw (2.85,-1.2) -- 
    ++ (0,-0.8);
    \end{tikzpicture}
    \caption{Crossed tail} 
    \label{fig::crossed tail}
\end{subfigure}
\caption{Pictorial description of tails and crossed tails.}
\label{fig::tailsandcrossed tails}
\end{figure}

This paper explicitly describes the structure, multiplicities and genera of components of $\X_k$. Before we give a precise statement let us illustrate the main result of the paper via an example.

\begin{eg}
\label{eg::firstintroeg}
Let $K=\Qur$, and $C/K$ be the hyperelliptic curve given by $$C:y^2=((x^3-p)^3-p^{15})((x-1)^4-p^9).$$ The cluster picture of $C/K$ is shown in Figure \ref{fig::eg3clusterpic} and the special fibre $\X_k$ of the minimal SNC model of $C/K$ is shown in Figure \ref{fig::eg3model}. The principal clusters in $\Sigma_{C/K}$ are $\s_1,\s_2,\s_3,\s_4,\s_5$, and $\Rcal$, as labeled in Figure \ref{fig::eg3clusterpic}. Note that $\s_3$, $\s_4$ and $\s_5$ are permuted by $\GK$ and denote their orbit by $X$. None of the principal clusters in this example are \"ubereven, so by Theorem \ref{thm::structureofSNCmodelintro}, each orbit of principal clusters gives rise to one central component, shown in bold and labeled in Figure \ref{fig::eg3model}. Clusters $\s_1$ and $\s_2 $ are children of $\Rcal$, so there are one or two linking chains between $\Gamma_{\Rcal}$ and $\Gamma_{\s_i}$ for $i=1,2$, and between $\Gamma_{\s_2}$ and $\Gamma_{X}$. 
\begin{figure}[ht]
\centering
\begin{subfigure}{0.45\textwidth}
  \centering
  \begin{tikzpicture}
    \fill (0,0) circle (1.5pt);
    \fill (0.25,0) circle (1.5pt);
    \fill (0.5,0) circle (1.5pt);
    \fill (1,0) circle (1.5pt);
    \fill (1.25,0) circle (1.5pt);
    \fill (1.5,0) circle (1.5pt);
    \fill (2,0) circle (1.5pt);
    \fill (2.25,0) circle (1.5pt);
    \fill (2.5,0) circle (1.5pt);
    \fill (3.5,0) circle (1.5pt);
    \fill (3.75,0) circle (1.5pt);
    \fill (4,0) circle (1.5pt);
    \fill (4.25,0) circle (1.5pt);

    \draw (0.25,0)
    ellipse (0.45cm and 0.25cm) node[below, yshift = -0.1cm, xshift = 0.4cm, font=\small]{$\s_3$}node[above, yshift = 0cm, xshift = 0.5cm, font=\small]{$\frac{13}{3}$};
    
    \draw (1.25,0)
    ellipse (0.45cm and 0.25cm) node[below, yshift = -0.1cm, xshift = 0.4cm, font=\small]{$\s_4$};
    
    \draw (2.25,0)
    ellipse (0.45cm and 0.25cm) node[below, yshift = -0.1cm, xshift = 0.4cm, font=\small]{$\s_5$ };

    \draw (1.25,0)
    ellipse (1.8cm and 0.9cm) node[below, yshift = -0.4cm, xshift = 1.7cm,font=\small]{$\s_2$} node[above, yshift = 0.3cm, xshift = 1.7cm,font=\small]{$\frac{1}{3}$};
    
    \draw (3.875,0)
    ellipse (0.6cm and 0.25cm) node[below, yshift = -0.1cm, xshift = 0.7cm,font=\small]{$\s_1$} node[above, yshift = 0cm, xshift = 0.7cm,font=\small]{$\frac{9}{4}$};
    
    \draw (2.125,0)
    ellipse (2.9cm and 1.2cm) node[below, yshift = -0.5cm, xshift = 2.7cm,font=\small]{$\Rcal$} node[above, yshift = 0.5cm, xshift = 2.7cm,font=\small]{$0$};
    \end{tikzpicture}
  \caption{Cluster picture $\Sigma_{C/K}$.}
  \label{fig::eg3clusterpic}
\end{subfigure}
\begin{subfigure}{.5\textwidth}
  \centering
  \begin{tikzpicture}
  \draw(0,0)[line width = 0.5mm
  ] -- node[below,font=\small, xshift=-0.2cm] {1}  node[right, font=\small, xshift=1.4cm] {$\Gamma_{\Rcal}$} ++ (3,0);
  
  \draw(2,0.2) -- node[right,font=\small] {1} ++ (0,-1);
  \draw(2.75,0.2) -- node[right,font=\small] {1} ++ (0,-1);
  \draw(1.85,-0.45)
  -- node[below,font=\small, xshift=-0.1cm,yshift=0.1cm] {1} ++ (0.9,-0.9);
  \draw(2.6,-0.45)
  -- node[below,font=\small,xshift=-0.1cm,yshift=0.1cm] {1} ++ (0.9,-0.9);
  \draw(2.4,-1.2)[line width = 0.5mm
  ] -- node[above,font=\small,xshift=0.3cm] {4}  node[right, font=\small, xshift=0.9cm] {$\Gamma_{\s_1}$}++ (2,0);
  \draw(4.15,-1)
  -- node[left,font=\small,yshift=-0.1cm] {2} ++ (0,-0.8);
  
  \draw(0.25,0.2)[line width = 0.5mm
  ]
  -- node[left,font=\small,yshift=-0.4cm] {6}  node[below, font=\small, yshift=-0.75cm] {$\Gamma_{\s_2}$} ++ (0,-1.6);
  \draw(0.05,-0.6)
  -- node[above,font=\small] {2} ++ (0.8,0);
  \draw(0.05,-1.2)
  -- node[above,font=\small] {3} ++ (1,0);
  \draw(0.85,-1) -- node[left,font=\small] {3} ++ (0,-1);
  \draw(0.65,-1.8) -- node[above,font=\small] {3} ++ (1,0);
  \draw(1.45,-1.6) -- node[left,font=\small] {3} ++ (0,-1);
  \draw(1.25,-2.4) -- node[above,font=\small] {3} ++ (1,0);
  \draw(2.05,-2.2)
  -- node[left,font=\small] {3} ++ (0,-1);
  \draw(1.85,-3)[line width = 0.5mm
  ]
  -- node[above,font=\small] {6}  node[right, font=\small, xshift=0.9cm] {$\Gamma_{X}$} ++ (2,0);
  \draw(3.6,-2.8)
  -- node[left,font=\small,yshift=-0.1cm] {3} ++ (0,-0.8);
  \draw(3.1,-2.8)
  -- node[left,font=\small,yshift=-0.1cm] {3} ++ (0,-0.8);
  \draw(2.6,-2.8)
  -- node[left,font=\small,yshift=-0.1cm] {3} ++ (0,-0.8);
  
    \end{tikzpicture}
  \caption{Special fibre of the minimal SNC model of C/K.}
  \label{fig::eg3model}
\end{subfigure}
\caption{$C:y^2=((x^3-p)^3-p^{15})((x-1)^4-p^9)$ over $K=\Qur$.}
\label{fig::eg3}
\end{figure}
Each of $\Gamma_{s_1},\Gamma_{s_2},$ and $\Gamma_{X}$ are also intersected by tails. How one determines the number and length of the linking chains and tails is discussed in Theorem \ref{thm::intromainthm2}.

In this example, we can compare the chains intersecting the central components in $\X_k$ to tails appearing in the minimal SNC models of related elliptic curves, see Table \ref{tab::KNclassification}. 
The chains intersecting $\Gamma_\Rcal$, along with $\Gamma_{\Rcal}$ itself, look much like a type $\mathrm{I}_0$ elliptic curve. Similarly type $\mathrm{III}$ for $\s_1$, type $\mathrm{II}$ for $\s_2$, and type $\mathrm{I}_0^*$ for $X$ (but with multiplicities multiplied by $|X|=3$).
\end{eg}


Here we give an abridged version of the description of the structure of the special fibre, given in full in Theorem \ref{thm::structureofSNCmodel}. In stating this theorem we use a subtle invariant of even clusters denoted $\epsilon_X$. This is defined fully in Definition \ref{def::clusterfunctions2}, however in practice for $X$ with $\s \in X$ even, and $r_{\s}$ any root of $\s$, $\epsilon_X$ is given by $\epsilon_X = (-1)^{|X|\left(v_K(c_f) + \sum_{r\not\in\s} v_K(r_{\s} - r)\right)}.$


\begin{thm}[Structure of SNC model]
    \label{thm::structureofSNCmodelintro}
    Let $K$ be a complete discretely valued field with algebraically closed residue field of characteristic $p>2$. Let $C/K$ be a hyperelliptic curve with tame potentially semistable reduction. Then the special fibre of its minimal SNC model is structured as follows. Every principal Galois orbit of clusters $X$ contributes one component $\Gamma_X$, unless $X$ is \"ubereven with $\epsilon_X = 1$, in which case $X$ contributes two components $\Gamma_X^+$ and $\Gamma_X^-$.
    
    These components are linked by chains of rational curves in the following cases (where, for any orbit $Y$, we write $\Gamma_Y^+ = \Gamma_Y^- = \Gamma_Y$ if $Y$ contributes only one central component):
    \begin{center}
    \begin{tabular}{|c|c|c|c|}
    \hline
         \small{Name} & \small{From} & \small{To} & \small{Condition} \\ \hline
         $L_{X,X'}$ & $\Gamma_X$ & $\Gamma_{X'}$  & $X' < X$ both principal, $X'$ odd \\ \hline
         $L_{X,X'}^+$ & $\Gamma_X^+$ & $\Gamma_{X'}^+$ & $X' < X$ both principal, $X'$ even with $\epsilon_{X'} = 1$ \\ \hline
         $L_{X,X'}^-$ & $\Gamma_X^-$ & $\Gamma_{X'}^-$ & $X' < X$ both principal, $X'$ even with $\epsilon_{X'} = 1$ \\ \hline
         $L_{X,X'}$ & $\Gamma_X$ & $\Gamma_{X'}$ & $X' < X$ both principal, $X'$ even with $\epsilon_{X'} = -1$ \\ \hline
         $L_{X'}$ & $\Gamma_X^-$ & $\Gamma_X^+$ & $X$ principal, $X' < X$ orbit of twins, $\epsilon_{X'} = 1$ \\ \hline
         $T_{X'}$ & $\Gamma_X$ & - & $X$ principal, $X' \leq X$ orbit of twins, $\epsilon_{X'} = -1$ \\ \hline
    \end{tabular}
    \end{center}
    Chains where the ``To'' column has been left blank are crossed tails. Some central components $\Gamma_X$ are also intersected transversally by tails. These are explicitly described in Theorem \ref{thm::intromainthm2}.
\end{thm}

The case when $\Rcal$ is not principal is described in Theorem \ref{thm::structureofSNCmodel}. We do not given explicit equations for the components in the special fibre. However, these can be calculated using the method laid out in this paper if desired (see Remark \ref{rem::explicitequations}). 

The linking chains, tails, and the multiplicities and genera of the components in the special fibre are given explicitly in Theorem \ref{thm::intromainthm2} below. In order to describe the chains of rational curves in detail, we introduce the notion of sloped chains of rational curves. We will also need a few other numerical invariants associated to clusters.

\begin{definition}
    Let $t_1,t_2 \in \Q$ and $\mu, \lambda \in \N$ with $\lambda$ minimal be such that 
    \begin{equation*}
        \mu t_1=\frac{m_0}{d_0}>\frac{m_1}{d_1}>\dots>\frac{m_{\lambda}}{d_{\lambda}}>\frac{m_{\lambda+1}}{d_{\lambda+1}}= \mu t_2,\;\textrm{and}\;\left|\begin{tabular}{cc}
    $m_i$ & $m_{i+1}$\\
    $d_i$ & $d_{i+1}$
    \end{tabular}\right|=1.
    \end{equation*}
    Suppose $\Ccal = \bigcup_{i=1}^{\lambda} E_i$ is a chain of rational curves where $E_i$ has multiplicity $\mu d_i$. Then $\Ccal$ is a \textit{sloped chain of rational curves} with parameters $(t_2,t_1,\mu)$. If $\Ccal$ is a tail, then $\Ccal$ is a sloped chain with parameters $(\lfloor t_1 - 1 \rfloor, t_1, \mu)$, so we usually just write $(t_1,\mu)$ for its parameters.
\end{definition}


\begin{notation}
Write $\tilde{\s}$ for the set of odd children of $\s$, and $\singletonsofs$ for the set of size $1$ children of $\s$.
\end{notation}
\begin{definition}
Let $\s$ be a cluster, then the \emph{semistable genus of $\s$} is given by 
$$|\tilde{\s}| = 2\ssgs + 1 \textrm{ or } 2\ssgs + 2,$$
or $\ssgs=0$ if $\s$ is \"ubereven. If $X$ is an orbit with $\s\in X$ the \emph{semistable genus of $X$} is defined by $g_{\textrm{ss}}(X)=\ssgs$. From this we define the \emph{genus} of an orbit $X$. If $X=\{\s\}$ is a trivial orbit with $d_s=\frac{a_s}{b_s}$, where $(a_{\s},b_{\s})=1$, and $\ssgs > 0$ then $g(\s)$ is given by
\begin{align*} 
g(X)=g(\s)= \begin{cases} \lfloor \frac{\ssgs}{b_{\s}} \rfloor & \lambda_{\s} \in \Z, \\
\lfloor \frac{\ssgs}{b_{\s}} + \frac{1}{2} \rfloor & \lambda_{\s} \not \in \Z, b_{\s} \textrm{ even}, \\
0 & \lambda_{\s} \not \in \Z, b_{\s} \textrm{ odd}.\end{cases} 
\end{align*}
Otherwise, $g(X)=g(\s) = 0$ if $\ssgs = 0$. For a general orbit $X$, define $g(X) = g(\s)$ for $\s \in X$, where $\s$ is considered as a cluster in $\Sigma_{C/K_X}$, and $K_X$ is the unique extension of $K$ of degree $|X|$.
\end{definition}

\begin{definition}
Let $X$ be an orbit of clusters with $\s\in X$, and $r_{\s}$ any root of $\s$. Define $e_X$ to be the minimal degree of extension required to make the clusters in $X$ satisfy the conditions of the Semistability Criterion \cite[Theorem~1.8]{DDMM18}. $X$ also has the following invariants: $$d_X = d_{\s}, \quad b_X = b_{\s}, \quad \delta_X =  d_{\s} - d_{P(\s)}, \textrm{ and }\lambda_X = \frac{v_K(c_f)}{2} + \frac{|\tilde{\s}|d_{\s}}{2} + \frac{1}{2}\sum_{r \not \in \s} v_K(r_{\s} - r).$$
\end{definition}

\begin{definition}
A child $\s'<\s$ is \textit{stable} if it has the same stabiliser as $\s$, and an orbit is \textit{stable} if all (equivalently any) of its children are stable.
\end{definition}

\begin{thm}
\label{thm::intromainthm2}
Let $K$ and $C/K$ be as in Theorem \ref{thm::structureofSNCmodelintro}. Let $X$ be a principal orbit of clusters in the cluster picture of a hyperelliptic curve $C$ with tame potentially semistable reduction and with $\Rcal$ principal. Then $\Gamma_X^{\pm}$ has genus $g(X)$. Furthermore, it has multiplicity $|X|e_X$ if $X$ is non-\"ubereven, or if $X$ is \"ubereven with $\epsilon_X = 1$; otherwise $\Gamma_X^{\pm}$ has multiplicity $2|X|e_X$ if $X$ is \"ubereven with $\epsilon_X = -1$. Suppose further that $e_X > 1$, and choose some $\s \in X$. Then the central component(s) associated to $X$ are intersected transversely by the following sloped tails with parameters $(t_1,\mu)$ (writing $\Gamma_X = \Gamma_X^+ = \Gamma_X^-$ if $X$ contributes only one central component):
    \begin{center}
    \begin{tabular}{|c|c|c|c|c|p{6cm}|}
    \hline
         \small{Name} & \small{From} & \small{Number} & $t_1$ & $\mu$ & \small{Condition}  \\ \hline
         $T_{\infty}$ & $\Gamma_X$ & $1$ & \small{$ (g + 1)d_{\Rcal}-\lambda_{\Rcal}$} & $1$ & $X=\{\Rcal\}$, $\Rcal$ odd \\ \hline
         $T_{\infty}^{\pm}$ & $\Gamma_X^{\pm}$ & $2$ & $d_{\Rcal}$ &$1$ & $X = \{\Rcal\}$, $\Rcal$ even, $\epsilon_{\Rcal} = 1$ \\ \hline
         $T_{\infty}$ & $\Gamma_X$ & $1$ & $d_{\Rcal}$ & $2$ & $X = \{\Rcal\}$, $\Rcal$ even, $e_{\Rcal} > 2$, $\epsilon_{\Rcal}=-1$ \\ \hline
         $T_{y=0}$ & $\Gamma_X$ & $\frac{|\singletonsofs||X|}{b_X}$ & $-\lambda_X$ & $|X|b_X$ &  $|\s_{\mathrm{sing}}|\geq2$, and $e_X>b_X/|X|$  \\\hline
         $T_{x=0}$ & $\Gamma_X$ & $1$ & $-d_X$ & $2|X|$ & $X$ has no stable child, $\lambda_X \not \in \Z$, $e_X>2$ and either $\ssg{X} > 0$ or $X$ is \"ubereven \\ \hline 
         $T_{x=0}^{\pm}$ & $\Gamma_X^{\pm}$ & $2$ & $-d_X$ & $|X|$ & $X$ has no stable child, $\lambda_X \in \Z$, and either $\ssg{X} > 0$ or $X$ is \"ubereven \\ \hline
         $T_{(0,0)}$ & $\Gamma_X$ & $1$ & $-\lambda_X$ & $|X|$ & $X$ has a stable singleton or $\ssg{X} = 0$, $X$ is not \"ubereven and $X$ has no proper stable odd child\\ \hline
    \end{tabular}
    \end{center} 
    The central components are intersected by the following sloped chains of rational curves with parameters $(t_2, t_2 + \delta, \mu)$:
    \begin{center}
    \begin{tabular}{|c|c|c|c|c|}
    \hline
         \small{Name} & $t_2$ & $\delta$ & $\mu$ & \small{Condition} \\ \hline
         $L_{X,X'}$ & $-\lambda_X$ & $\delta_{X'}/2$& $|X|$ & $X' \leq X$ both principal, $X'$ odd \\ \hline
         $L_{X,X'}^+$ & $-d_X$ & $\delta_{X'}$ & $|X|$ & $X' \leq X$ both principal, $X'$ even with $\epsilon_{X'} = 1$ \\ \hline
         $L_{X,X'}^-$ & $-d_X$ & $\delta_{X'}$ & $|X|$ & $X' \leq X$ both principal, $X'$ even with $\epsilon_{X'} = 1$ \\ \hline
         $L_{X,X'}$ & $-d_X$ & $\delta_{X'}$ & $2|X|$ & $X' \leq X$ both principal, $X'$ even with $\epsilon_{X'} = -1$ \\ \hline
         $L_{X'}$ & $-d_X$ & $2\delta_{X'}$ & $|X|$ & $X$ principal, $X' \leq X$ orbit of twins, $\epsilon_{X'} = 1$ \\ \hline
         $T_{X'}$ & $-d_X$ & $\delta_{X'}+\frac{1}{\mu}$ & $2|X|$ & $X$ principal, $X' \leq X$ orbit of twins, $\epsilon_{X'} = -1$ \\ \hline
    \end{tabular}
    \end{center}
    Note that here the names indicate the components which each chains intersect, as explicitly written in the second table of Theorem \ref{thm::structureofSNCmodelintro}. Finally, the crosses of any crossed tail have multiplicity $\frac{\mu}{2}$.
\end{thm}

In practice, when $\ssg{X}>0$ there is a one-to-one correspondence between the chains intersecting a central component $\Gamma_X$ and the tails of the unique central component $\Gamma_{\widetilde{X}}$ of the minimal SNC model $\X_{\widetilde{X}}$ of a related curve $C_{\widetilde{X}}$. Choose some $\s\in X$. Then the curve $C_{\widetilde{X}}$ is a hyperelliptic curve over $K_X$ (an extension of $K$ of degree $|X|$) which has as its roots a centre for each odd child of $\s$ (with a correction to the leading coefficient to account for the rest of the cluster picture). This preserves the multiplicities in the corresponding chains (up to some small corrections), and the genus of $\Gamma_X$ is equal to the genus of $\Gamma_{\widetilde{X}}$ (or $0$ if $\ssg{X} = 0$). This idea has been briefly explored in Example \ref{eg::firstintroeg}, comparing parts of the special fibre to minimal SNC models of certain elliptic curves. We have a closer look at this idea in Example \ref{eg::intro1} below. Since $C_{\widetilde{X}}$ will often have a lower genus then $C$, this allows us to construct the minimal SNC model of $C$ in terms of simpler models. We now give some more examples, the first of which completely summarises the case for elliptic curves, and the second provides the motivation behind Theorem \ref{thm::intromainthm2}.

\begin{eg}
\label{eg::KNclassification}
This table shows $\X_k$, of the minimal SNC model $\X$ for the different Kodaira-N\'eron types of elliptic curves with tame potentially semistable reduction (for which it is sufficient to take $p\geq 5$). Our table differs from the table found in \cite[p~365]{Sil94}, where instead the special fibers of the \textit{minimal regular models} for the different types of elliptic curves are shown. This makes a difference for type II, III or IV elliptic curves, whereas for all the other types the minimal regular model is SNC. These special fibres can be read off straight from Theorems \ref{thm::structureofSNCmodelintro} and \ref{thm::intromainthm2}: one does not need to follow any laborious algorithm.

\pagebreak
\begin{center}
\renewcommand{\arraystretch}{1}
\begin{longtable}[h]{|c|c|c|p{0.5cm}|c|c|c|}
         \cline{1-3} \cline{5-7}
         Type & Cluster Picture & $\X_k$ & & Type & Cluster Picture & $\X_k$\\ \cline{1-3} \cline{5-7}
         
    $\mathrm{I}_0$ & \begin{tikzpicture}
    \fill (3.5,0) circle (1.5pt);
    \fill (3.75,0) circle (1.5pt);
    \fill (4,0) circle (1.5pt);
    \draw (3.75,0) ellipse (0.5cm and 0.25cm) node[above, yshift = 0cm, xshift = 0.6cm,font=\small]{$0$};
    \end{tikzpicture} & 
    \begin{tikzpicture}
    \draw (0,0) -- node[above,font=\small,xshift=0.7cm]{$1\;g1$} ++ (2,0);
    \end{tikzpicture} 
    & &
    $\mathrm{I}_0^*$ & \begin{tikzpicture}
    \fill (3.5,0) circle (1.5pt);
    \fill (3.75,0) circle (1.5pt);
    \fill (4,0) circle (1.5pt);
    \draw (3.75,0) ellipse (0.5cm and 0.25cm) node[above, yshift = 0cm, xshift = 0.6cm,font=\small]{$1$};
    \end{tikzpicture} & \begin{tikzpicture}
    \draw (0,0) -- node[above,font=\small,xshift=0.7cm] {$2$} ++ (2.75,0);
    \draw (0.25,-0.2) -- node[left,font=\small, yshift=0.2cm] {$1$} ++ (0,0.6);
    \draw (0.65,-0.2) -- node[left,font=\small, yshift=0.2cm] {$1$} ++ (0,0.6);
    \draw (1.05,-0.2) -- node[left,font=\small, yshift=0.2cm] {$1$} ++ (0,0.6);
    \draw (1.45,-0.2) -- node[left,font=\small, yshift=0.2cm] {$1$} ++ (0,0.6);
    \end{tikzpicture}\\ \cline{1-3} \cline{5-7}

    $\mathrm{I}_n$ & \begin{tikzpicture}
    \fill (0,0) circle (1.5pt);
    \fill (0.25,0) circle (1.5pt);
    \fill (1,0) circle (1.5pt);
    \draw (0.125,0)
    ellipse (0.4cm and 0.25cm) node[above, yshift = -0.1cm, xshift = 0.5cm, font=\small]{$\frac{n}{2}$};
    \draw (0.5,0)
    ellipse (0.9cm and 0.5cm) node[above, yshift = 0cm, xshift = 1cm, font=\small]{$0$};
    \end{tikzpicture} & \begin{tikzpicture}
    \draw (-0.1,-0.1) -- node[left,font=\small, yshift = 0.15cm, xshift = 0.1cm]{$1$} ++ (0.48,0.64);
    \draw (-0.1,0.1) -- node[left,font=\small, yshift = -0.15cm, xshift = 0.1cm]{$1$} ++ (0.48,-0.64);
    \path (0.45,0.54) -- node[auto=false]{\ldots} node[below, yshift = -0.34cm, font=\small]{$n$-gon}(1.1,0.54);
    \path (0.45,-0.54) -- node[auto=false]{\ldots} (1.1,-0.54);
    \draw (1.6,-0.1) -- node[right,font=\small, yshift = 0.15cm, xshift = -0.1cm]{$1$} ++ (-0.48,0.64);
    \draw (1.6,0.1) -- node[right,font=\small, yshift = -0.15cm, xshift = -0.1cm]{$1$} ++ (-0.48,-0.64);
    \end{tikzpicture} 
    &  &
    $\mathrm{I}_n^*$ & \begin{tikzpicture}
    \fill (0,0) circle (1.5pt);
    \fill (0.25,0) circle (1.5pt);
    \fill (1,0) circle (1.5pt);
    \draw (0.125,0)
    ellipse (0.4cm and 0.25cm) node[above, yshift = -0.1cm, xshift = 0.5cm, font=\small]{$\frac{n+1}{2}$};
    \draw (0.5,0)
    ellipse (0.9cm and 0.5cm) node[above, yshift = 0cm, xshift = 1cm, font=\small]{$1$};
    \end{tikzpicture} &
    \begin{tikzpicture}
    \draw (-0.2,0.4) -- node[above,font=\small, xshift=0.2cm]{$2$} ++ (1.4,0);
    \draw (1,-0.2) -- node[left,font=\small,yshift=-0.2cm, xshift=0.05cm]{$2$} ++ (0,0.8);
    \path (1.05,-0.2) -- node[auto=false]{\ldots}  node[above, font=\small,yshift=-0.05cm]{$n$}(1.7,-0.2);
    
    \draw (1.7,-0.2) -- node[right,font=\small,yshift=-0.2cm,xshift=-0.05cm]{$2$} ++ (0,0.8);
    \draw (1.5,0.4) -- node[above,font=\small, xshift=-0.2cm]{$2$} ++ (1.4,0);
    
    \draw (0,-0.2) -- node[left,font=\small,yshift=-0.2cm,xshift=0.05cm]{$1$} ++ (0,0.8);
    \draw (0.4,-0.2) -- node[left,font=\small,yshift=-0.2cm,xshift=0.05cm]{$1$} ++ (0,0.8);
    \draw (2.7,-0.2) -- node[right,font=\small,yshift=-0.2cm,xshift=-0.05cm]{$1$} ++ (0,0.8);
    \draw (2.3,-0.2) -- node[right,font=\small,yshift=-0.2cm,xshift=-0.05cm]{$1$} ++ (0,0.8);
    \end{tikzpicture}\\ \cline{1-3} \cline{5-7}
    
    II &  \begin{tikzpicture}
    \fill (3.5,0) circle (1.5pt);
    \fill (3.75,0) circle (1.5pt);
    \fill (4,0) circle (1.5pt);
    \draw (3.75,0) ellipse (0.5cm and 0.25cm) node[above, yshift = 0cm, xshift = 0.6cm,font=\small]{$\frac{1}{3}$};
    \end{tikzpicture} & \begin{tikzpicture}
    \draw (0,0) -- node[above,font=\small,xshift=0.7cm] {$6$} ++ (2,0);
    \draw (0.25,0.2) -- node[left,font=\small, yshift=-0.2cm] {$3$} ++ (0,-0.8);
    \draw (0.75,0.2) -- node[left,font=\small, yshift=-0.2cm] {$2$} ++ (0,-0.8);
    \draw (1.25,0.2) -- node[left,font=\small, yshift=-0.2cm] {$1$} ++ (0,-0.8);
    \end{tikzpicture}
    & &
   $\mathrm{IV}^*$ & \begin{tikzpicture}
    \fill (3.5,0) circle (1.5pt);
    \fill (3.75,0) circle (1.5pt);
    \fill (4,0) circle (1.5pt);
    \draw (3.75,0) ellipse (0.5cm and 0.25cm) node[above, yshift = 0cm, xshift = 0.6cm,font=\small]{$\frac{4}{3}$};
    \end{tikzpicture} & \begin{tikzpicture}
    \draw (0.05,0) -- node[right,font=\small,xshift=1.3cm] {$3$} ++ (2.78,0);
    \draw (0.25,0.2) -- node[left,font=\small,xshift=0.05cm] {$2$} ++ (0,-0.9);
    \draw (0.05,-0.5) -- node[above,font=\small,xshift=0.1cm,yshift=-0.05cm] {$1$} ++ (0.6,0);
    \draw (1.25,0.2) -- node[left,font=\small,xshift=0.05cm] {$2$} ++ (0,-0.9);
    \draw (1.05,-0.5) -- node[above,font=\small,xshift=0.1cm,yshift=-0.05cm] {$1$} ++ (0.6,0);
    \draw (2.25,0.2) -- node[left,font=\small,xshift=0.05cm] {$2$} ++ (0,-0.9);
    \draw (2.05,-0.5) -- node[above,font=\small,xshift=0.1cm,yshift=-0.05cm] {$1$} ++ (0.6,0);
   
    \end{tikzpicture} \\ \cline{1-3} \cline{5-7}

    III & \begin{tikzpicture}
    \fill (3.5,0) circle (1.5pt);
    \fill (3.75,0) circle (1.5pt);
    \fill (4,0) circle (1.5pt);
    \draw (3.75,0) ellipse (0.5cm and 0.25cm) node[above, yshift = 0cm, xshift = 0.6cm,font=\small]{$\frac{1}{2}$};
    \end{tikzpicture} & \begin{tikzpicture}
    \draw (0,0) -- node[above,font=\small,xshift=0.7cm] {$4$} ++ (2,0);
    \draw (0.25,0.2) -- node[left,font=\small, yshift=-0.2cm] {$1$} ++ (0,-0.8);
    \draw (0.75,0.2) -- node[left,font=\small, yshift=-0.2cm] {$2$} ++ (0,-0.8);
    \draw (1.25,0.2) -- node[left,font=\small, yshift=-0.2cm] {$1$} ++ (0,-0.8);
    \end{tikzpicture} 
    & &
   $\mathrm{III}^*$ & \begin{tikzpicture}
    \fill (3.5,0) circle (1.5pt);
    \fill (3.75,0) circle (1.5pt);
    \fill (4,0) circle (1.5pt);
    \draw (3.75,0) ellipse (0.5cm and 0.25cm) node[above, yshift = 0cm, xshift = 0.6cm,font=\small]{$\frac{3}{2}$};
    \end{tikzpicture} & \begin{tikzpicture}
    \draw (0,0) -- node[right,font=\small,xshift=1.3cm, yshift =0.1cm] {$4$} ++ (2.75,0);
    \draw (0.25,0.2) -- node[left,font=\small, yshift=-0.2cm, xshift=0.05cm] {$2$} ++ (0,-0.8);
    \draw (1,0.2) -- node[left,font=\small,xshift=0.05cm] {$3$} ++ (0,-0.9);
    \draw (0.8,-0.5) -- node[above,font=\small, yshift=-0.05cm] {$2$} ++ (0.9,0);
    \draw (1.5,-0.3) -- node[left,font=\small,yshift=-0.15cm, xshift=0.05cm] {$1$} ++ (0,-0.6);
    \draw (2.25,0.2) -- node[left,font=\small, xshift=0.05cm] {$3$} ++ (0,-0.9);
    \draw (2.05,-0.5) -- node[above,font=\small, yshift=-0.05cm] {$2$} ++ (0.9,0);
    \draw (2.75,-0.3) -- node[left,font=\small,yshift=-0.15cm, xshift=0.05cm] {$1$} ++ (0,-0.6);
    \end{tikzpicture} \\ \cline{1-3} \cline{5-7}
    
    IV & \begin{tikzpicture}
    \fill (3.5,0) circle (1.5pt);
    \fill (3.75,0) circle (1.5pt);
    \fill (4,0) circle (1.5pt);
    \draw (3.75,0) ellipse (0.5cm and 0.25cm) node[above, yshift = 0cm, xshift = 0.6cm,font=\small]{$\frac{2}{3}$};
    \end{tikzpicture} & \begin{tikzpicture}
    \draw (0,0) -- node[above,font=\small,xshift=0.7cm] {$3$} ++ (2,0);
    \draw (0.25,0.2) -- node[left,font=\small, yshift=-0.2cm] {$1$} ++ (0,-0.8);
    \draw (0.75,0.2) -- node[left,font=\small, yshift=-0.2cm] {$1$} ++ (0,-0.8);
    \draw (1.25,0.2) -- node[left,font=\small, yshift=-0.2cm] {$1$} ++ (0,-0.8);
    \end{tikzpicture} 
    & &
    $\mathrm{II}^*$ & \begin{tikzpicture}
    \fill (3.5,0) circle (1.5pt);
    \fill (3.75,0) circle (1.5pt);
    \fill (4,0) circle (1.5pt);
    \draw (3.75,0) ellipse (0.5cm and 0.25cm) node[above, yshift = 0cm, xshift = 0.6cm,font=\small]{$\frac{5}{3}$};
    \end{tikzpicture} & \begin{tikzpicture}
    \draw (0,0) -- node[right,font=\small,xshift=1.3cm, yshift =0.1cm] {$6$} ++ (2.75,0);
    \draw (0.25,0.2) -- node[left,font=\small, yshift=-0.2cm] {$3$} ++ (0,-0.8);
    \draw (1,0.2) -- node[left,font=\small,xshift=0.05cm] {$4$} ++ (0,-0.9);
    \draw (0.8,-0.5) -- node[above,font=\small, xshift=0.2cm,yshift=-0.05cm] {$2$} ++ (0.6,0);
    
    \draw (2.25,0.2) -- node[left,font=\small, xshift=0.05cm] {$5$} ++ (0,-0.9);
    \draw (2.05,-0.5) -- node[above,font=\small, yshift=-0.05cm] {$4$} ++ (0.9,0);
    \draw (2.75,-0.3) -- node[right,font=\small,xshift=-0.05cm] {$3$} ++ (0,-0.9);
    \draw (2.05,-1) -- node[above,font=\small,yshift=-0.05cm] {$2$} ++ (0.9,0);
    \draw (2.25,-0.8) -- node[left,font=\small,yshift=-0.15cm,xshift=0.05cm] {$1$} ++ (0,-0.6);
    \end{tikzpicture}\\
    \cline{1-3} \cline{5-7}
    \caption{Kodaira-N\'eron types of elliptic curves with $p\geq 5$. 
    }
    \label{tab::KNclassification}
    \end{longtable}\end{center}
\end{eg}

\vspace{-30px}
\begin{eg}
\label{eg::intro1}
Let $C$ over $K=\Qur$ be the hyperelliptic curve given by Weierstrass equation $$C:y^2=f(x)=(x^3-p^2)(x^4-p^{11}).$$ The cluster picture of $C/K$ consists of two proper clusters $\Rcal$ and $\s$, shown in Figure \ref{fig::introeg3cluster}. The special fibre $\X_k$ of the minimal SNC model $\X$ of $C/K$  is shown in Figure \ref{fig::introeg3model}.
\begin{figure}[h]
\centering
\begin{subfigure}{.5\textwidth}
  \centering
  \begin{tikzpicture}
    \fill (2.5,0) circle (1.5pt);
    \fill (2.75,0) circle (1.5pt);
    \fill (3,0) circle (1.5pt);
    \fill (3.5,0) circle (1.5pt);
    \fill (3.75,0) circle (1.5pt);
    \fill (4,0) circle (1.5pt);
    \fill (4.25,0) circle (1.5pt);
    \draw (3.875,0) ellipse (0.65cm and 0.25cm) node[below, yshift = 0cm, xshift = 0.7cm,font=\small]{$\s$} node[above, yshift = 0cm, xshift = 0.7cm,font=\small]{$\frac{11}{4}$};
    \draw (3.55,0) ellipse (1.55cm and 0.8cm)node[below, yshift = -0.2cm, xshift = 1.55cm,font=\small]{$\Rcal$} node[above, yshift = 0.2cm, xshift = 1.55cm,font=\small]{$\frac{2}{3}$};
    \end{tikzpicture}
    \caption{Cluster picture $\Sigma_{C/K}$.}
    \label{fig::introeg3cluster}
\end{subfigure}%
\begin{subfigure}{.5\textwidth}
  \centering
  \begin{tikzpicture}
    \draw (0,0) [line width = 0.5mm] -- node[above,font=\small,xshift=-0.7cm] {$4$} ++ (2.75,0);
    \draw (0.25,-0.2) -- node[left,font=\small, yshift=0.2cm] {$2$} ++ (0,0.8);
    \draw (1,-0.2) -- node[right,font=\small] {$3$} ++ (0,1);
    \draw (0.8,0.6) -- node[above,font=\small] {$2$} ++ (1,0);
    \draw (1.6,0.4) -- node[right,font=\small] {$1$} ++ (0,1);
    \draw (1.75,1.05) -- node[left,font=\small] {$1$} ++ (-0.8,0.8);
    
    \draw (2.25,-0.2) -- node[right,font=\small] {$3$} ++ (0,1);
    \draw (2.05,0.6) -- node[above,font=\small] {$2$} ++ (1,0);
    \draw (2.85,0.4) -- node[right,font=\small] {$1$} ++ (0,1);
    \draw (3,1.05) -- node[left,font=\small] {$1$} ++ (-0.8,0.8);
    
    \draw (0.45,1.7)[line width = 0.5mm] -- node[left,font=\small,xshift=-1.65cm] {$3$} ++ (3.4,0);
    \draw (3.5,1.9) -- node[right,font=\small] {$1$} ++ (0,-0.8);
    \end{tikzpicture}
    \caption{Special fibre of the minimal SNC model of $C/K$.}
    \label{fig::introeg3model}
\end{subfigure}
\caption{$C:y^2=x^3-p^{2}$ over $K=\Qur$.}
\label{fig::introeg3}
\end{figure}

Define elliptic curves $C_1$ and $C_2$ over $K$ by $C_1:y^2=f_1(x)=x^3-p^2$ and $C_2:y^2=p^2 f_2(x)=p^2(x^4-p^{11})$ respectively. Note that $f(x)=f_1(x)\cdot f_2(x)$. The roots of $f_1(x)$ contribute the roots in $\Rcal\setminus\s$, and the roots of $f_2(x)$ contribute the roots in $\s$. The coefficient in the defining equation of $C_2$ is chosen to somehow ``see'' the roots of $f_1$. It is interesting to compare the minimal SNC models of $C_i$ to that of $C$ for $i=1,2$. Note that $C_1$ and $C_2$ are type IV and type III$^*$ elliptic curves respectively, as shown in Table \ref{tab::KNclassification}. 
It appears that the roots of $f_1$ and $f_2$ are making their own contributions to $\X_k$, as both the special fibres of the minimal SNC models of $C_i$ can be seen as ``submodels'' of $\X_k$ for $i=1,2$. This shows how $\Rcal$ and $\s$ each make their own contribution to $\X_k$. Since $\s$ is an even child of $\Rcal$, and $\epsilon_{\s}=1$, there are two linking chains between their contributions in $\X_k$ .
\end{eg}

\begin{eg}
Let $K=\Qur$, and $C/K$ be the hyperelliptic curve given by $$C:y^2=(x^3-p^4)((x-1)^3-p^{17})((x-2)^3-p^{13}).$$
\vspace{-10px}
\begin{figure}[ht]
\centering
\begin{subfigure}{0.45\textwidth}
  \centering
  \begin{tikzpicture}
    \fill (0,0) circle (1.5pt);
    \fill (0.25,0) circle (1.5pt);
    \fill (0.5,0) circle (1.5pt);
    \fill (1.25,0) circle (1.5pt);
    \fill (1.5,0) circle (1.5pt);
    \fill (1.75,0) circle (1.5pt);
    \fill (2.5,0) circle (1.5pt);
    \fill (2.75,0) circle (1.5pt);
    \fill (3,0) circle (1.5pt);

    \draw (0.25,0)
    ellipse (0.45cm and 0.25cm) node[below, yshift = -0.1cm, xshift = 0.4cm, font=\small]{$\s_1$}node[above, yshift = 0cm, xshift = 0.5cm, font=\small]{$\frac{4}{3}$};
    
    \draw (1.5,0)
    ellipse (0.45cm and 0.25cm) node[below, yshift = -0.1cm, xshift = 0.4cm, font=\small]{$\s_2$}node[above, yshift = 0cm, xshift = 0.5cm, font=\small]{$\frac{17}{3}$};
    
    \draw (2.75,0)
    ellipse (0.45cm and 0.25cm) node[below, yshift = -0.1cm, xshift = 0.4cm, font=\small]{$\s_3$ }node[above, yshift = 0cm, xshift = 0.5cm, font=\small]{$\frac{13}{3}$};

    \draw (1.55,0) ellipse (2.2cm and 1cm) node[below, yshift = -0.5cm, xshift = 1.9cm,font=\small]{$\Rcal$} node[above, yshift = 0.5cm, xshift = 1.9cm,font=\small]{$0$};
    \end{tikzpicture}
  \caption{Cluster picture $\Sigma_{C/K}$.}
  \label{fig::eg2clusterpic}
\end{subfigure}
\begin{subfigure}{.5\textwidth}
  \centering
  \begin{tikzpicture}
  \draw(0,0)[line width = 0.5mm] -- node[above,font=\small,yshift=-0.05cm]{$1$ g$1$}  node[right, font=\small, xshift=1.4cm] {$\Gamma_{\Rcal}$} ++ (3,0);
  
  \draw(0.25,0.2)
  -- node[right,font=\small] {2} ++ (0,-1);
  \draw(0.45,-0.6)[line width = 0.5mm
  ] -- node[above,font=\small] {3}  node[left, font=\small, xshift=-0.9cm] {$\Gamma_{\s_1}$} ++ (-2,0);
  \draw(-0.25,-0.4)
  -- node[right,font=\small] {2} ++ (0,-1);
  \draw(-0.45,-1.2)
  -- node[below,font=\small,xshift=0.1cm] {1} ++ (0.8,0);
  \draw(-1.25,-0.4)
  -- node[right,font=\small] {2} ++ (0,-1);
  \draw(-1.45,-1.2)
  -- node[below,font=\small,xshift=0.1cm] {1} ++ (0.8,0);
  
  \draw(1,0.2)
  -- node[left,font=\small] {1} ++ (0,-1);
  \draw(0.8,-0.6)
  -- node[above,font=\small] {1} ++ (1,0);
  \draw(1.6,-0.4)
  -- node[left,font=\small] {2} ++ (0,-1);
  \draw(1.4,-1.2)
  -- node[above,font=\small] {3} ++ (1,0);
  \draw(2.2,-1)
  -- node[left,font=\small] {4} ++ (0,-1);
  \draw(2.35,-1.65)
  -- node[left,font=\small, xshift=0.1cm, yshift=0.1cm] {5} ++ (-0.9,-0.9);
  \draw(2.4,-2.4)[line width = 0.5mm
  ] -- node[above,font=\small] {6}  node[right, font=\small, xshift=0.9cm] {$\Gamma_{\s_2}$} ++ (-2,0);
  \draw(2.1,-2.2)
  -- node[right,font=\small] {3} ++ (0,-0.8);
  \draw(0.6,-2.2)
  -- node[left,font=\small] {4} ++ (0,-1);
  \draw(0.4,-3)
  -- node[above,font=\small, xshift=0.1cm] {2} ++ (0.8,0);
  
  \draw(2.5,0.2)
  -- node[right,font=\small] {1} ++ (0,-1);
  \draw(2.35,-0.45)
  -- node[left,font=\small,xshift=0.1cm,yshift=-0.1cm] {1} ++ (0.9,-0.9);
  \draw(2.9,-1.2)[line width = 0.5mm
  ] -- node[above,font=\small, xshift=0.2cm]{6}  node[right, font=\small, xshift=0.65cm] {$\Gamma_{\s_3}$} ++ (1.5,0);
  \draw(3.65,-1)
  -- node[left,font=\small] {3} ++ (0,-1);
  \draw(4.15,-1)
  -- node[left,font=\small] {2} ++ (0,-1);
  
    \end{tikzpicture}
  \caption{Special fibre of the minimal SNC model of C/K.}
  \label{fig::eg2model}
\end{subfigure}
\caption{$C:y^2=(x^3-p^4)((x-1)^3-p^{17})((x-2)^3-p^{13})$ over $K=\Qur$.}
\label{fig::eg2}
\end{figure}

The central components of the minimal SNC model of $C$ (Figure \ref{fig::eg2model}), which arise from clusters in $\Sigma_{C/K}$ (\ref{fig::eg2clusterpic}), are labeled. Note that $\Rcal$ contributes components to the model which look like those appearing the the minimal SNC model of a type $\mathrm{I}_0$ elliptic curve; $\s_1$ those of a type $\mathrm{IV}^*$ elliptic curve; $\s_2$ those of a type $\mathrm{II}^*$ elliptic curve; and $\s_3$ those of a type $\mathrm{II}$ elliptic curve. The special fibers of the minimal SNC models of these Kodaira types are all shown in Table \ref{tab::KNclassification} in Example \ref{eg::KNclassification}. This reflects the general phenomenon discussed above that the chains intersecting a central component arising from a cluster $\s$ ``correspond'' to the tails of a hyperelliptic curve constructed from $\s$.
\end{eg}

\begin{eg}
Let $K=\Qur$, and $C/K$ be the hyperelliptic curve given by $$C:y^2=(x^3-p)((x^3-p^4)^2-p^9).$$ The cluster picture of $C/K$ is shown in Figure \ref{fig::eg4clusterpic} and the special fibre of the minimal SNC model of $C/K$ is shown in Figure \ref{fig::eg4model}. The clusters $\tfrak_1$, $\tfrak_2$ and $\tfrak_3$ are swapped by $\GK$ and denote their orbit by $X$.  The central components of the model, which arise from clusters in $\Sigma_{C/K}$, are labeled, as is the crossed tail $T_X$ arising from the orbit of twins $X$.

The component $\Gamma_{\Rcal}$ and its chains look like a type II elliptic curve. Since $\s$ is \"ubereven, we cannot construct a curve $C_{\widetilde{\s}}$ to compare the contributions of $\s$ to. However, observe that $\s$ and its children contribute a divisor which looks like the minimal SNC model of a Namikawa-Ueno type $\mathrm{III}_2^*$ curve. In our final proof we will use induction on the number of proper clusters and this is a useful example to look back to when we do so.
\vspace{-10px}
\begin{figure}[ht]
\centering
\begin{subfigure}{0.45\textwidth}
  \centering
  \begin{tikzpicture}
    \fill (0,0) circle (1.5pt);
    \fill (0.25,0) circle (1.5pt);
    \fill (0.9,0) circle (1.5pt);
    \fill (1.15,0) circle (1.5pt);
    \fill (1.8,0) circle (1.5pt);
    \fill (2.05,0) circle (1.5pt);
    \fill (3.05,0) circle (1.5pt);
    \fill (3.3,0) circle (1.5pt);
    \fill (3.55,0) circle (1.5pt);

    \draw (0.125,0)
    ellipse (0.35cm and 0.2cm) node[below, yshift = -0.1cm, xshift = 0.3cm, font=\small]{$\tfrak_1$}node[above, yshift = 0cm, xshift = 0.4cm, font=\small]{$\frac{11}{6}$};
    
    \draw (1.025,0)
    ellipse (0.35cm and 0.2cm) node[below, yshift = -0.1cm, xshift = 0.3cm, font=\small]{$\tfrak_2$};
    
    \draw (1.925,0)
    ellipse (0.35cm and 0.2cm) node[below, yshift = -0.1cm, xshift = 0.3cm, font=\small]{$\tfrak_3$ };

    \draw (1.025,0)
    ellipse (1.6cm and 0.9cm) node[below, yshift = -0.4cm, xshift = 1.5cm,font=\small]{$\s$} node[above, yshift = 0.3cm, xshift = 1.5cm,font=\small]{$\frac{4}{3}$};
    
    \draw (1.5,0)
    ellipse (2.4cm and 1.2cm) node[below, yshift = -0.5cm, xshift = 2.3cm,font=\small]{$\Rcal$} node[above, yshift = 0.5cm, xshift = 2.3cm,font=\small]{$\frac{1}{3}$};
    \end{tikzpicture}
  \caption{Cluster picture $\Sigma_{C/K}$.}
  \label{fig::eg4clusterpic}
\end{subfigure}
\begin{subfigure}{.5\textwidth}
  \centering
  \begin{tikzpicture}
  \draw(0,0)[line width = 0.5mm
  ] -- node[above,font=\small] {6}  node[right, font=\small, xshift=1.4cm] {$\Gamma_{\Rcal}$} ++ (3,0);
  
  \draw(2,0.2)
  -- node[right,font=\small,yshift=-0.1cm] {3} ++ (0,-0.8);
  \draw(2.75,0.2)
  -- node[right,font=\small,yshift=-0.1cm] {1} ++ (0,-0.8);
  
  \draw(0.25,0.2)
  -- node[left,font=\small] {2} ++ (0,-1);
  \draw(0.05,-0.6) -- node[above,font=\small] {2} ++ (1,0);
  \draw(0.85,-0.4)
  -- node[left,font=\small] {4} ++ (0,-1);
  \draw(0.65,-1.2)[line width = 0.5mm] -- node[above,font=\small] {6} node[right, font=\small, xshift=1.4cm] {$\Gamma_{\s}$} ++ (3,0);
  
  \draw(1.65,-1) -- node[left,font=\small,yshift=-0.1cm] {2} ++ (0,-0.8);
  
  \draw(2.65,-1) -- node[left,font=\small] {6} ++ (0,-1);
  \draw(2.45,-1.8) -- node[above,font=\small] {6} ++ (1,0);
  \draw(3.25,-1.6) -- node[left,font=\small, yshift=0.2cm] {6} ++ (0,-1.5);
  
    \draw (3.05,-2.6) -- node[right,font=\small, xshift=0.35cm, yshift=0.05cm] {3}++ (0.8,0);
    \draw (3.05,-2.85) -- node[right,font=\small,xshift=0.35cm, yshift=-0.05cm] {$3$} ++ (0.8,0);
    \node[font=\small] at (2.7,-2.675) {$T_X$};
    \end{tikzpicture}
  \caption{Special fibre of the minimal SNC model of C/K.}
  \label{fig::eg4model}
\end{subfigure}
\caption{$C:y^2=(x^3-p)((x^3-p^4)^2-p^9)$ over $K=\Qur$.}
\label{fig::eg4}
\end{figure}
\end{eg}

In the case where $K$ does \textit{not} have algebraically closed residue field, the following theorem tells us precisely how the Frobenius automorphism acts on the components of the minimal SNC model.

\begin{thm}[Frobenius Action]
    \label{thm::frobactionintro}
    Let $K$ be a field, not necessarily with algebraically closed residue field, and let $C/K$ be a curve 
    with tame potentially semistable reduction and minimal SNC model $\X$ over $K^{\textrm{ur}}$. 
    Then the Frobenius automorphism, $\Frob$, acts on the components of $\X$ as:
    
    \begin{enumerate}
        \item $\Frob(\Gamma_X^{\pm}) = \Gamma_{\Frob(X)}^{\pm\epsilon_X(\Frob)}$,
        \item $\Frob(L_{X,\; X'}^{\pm}) = L_{\Frob(X),\; \Frob(X')}^{\pm\epsilon_{X'}(\Frob)}$,
        \item a loop $L_X$ is sent to $\epsilon_X(\Frob)L_{\Frob(X)}$, a crossed tail $T_X$ to $\epsilon_X(\Frob)T_{\Frob(X)}$,\footnote{$-L_X$ is same loop but with reversed orientation. $-T_X$ is the same crossed tail but with crosses swapped.}
        \item  and tails are permuted as $\Frob(T_{\infty}^{\pm}) = T_{\infty}^{\pm\epsilon_X(\Frob)}$, $\Frob(T_{x=0}^{\pm}) = T_{x=0}^{\pm1^{v(c_X)}}$, and $(y=0)$-tails are permuted as the corresponding roots of the cluster pictures are.
    \end{enumerate}
\end{thm}

The following example shows an application of this theorem, checking whether a curve has a $K$-rational point.

\begin{eg}
    Let $K = \Q_7$ and let $C/K$ be the hyperelliptic curve given by $$C:y^2=(x-i)\left((x-i)^2 - p\right)(x+i)\left((x+i)^2 - p\right)$$
    
    We require the full power of Theorem \ref{thm::structureofSNCmodel} to deduce the minimal SNC model of $C$, but Theorem \ref{thm::frobactionintro} still tells us how the components are permuted. In particular, $\Gamma_{\s_1}$ and $\Gamma_{\s_2}$ and their respective tails are swapped. In particular, there are no smooth points which are fixed by Frobenius, since the only multiplicity 1 components are swapped by Frobenius. Therefore, $C$ has no $\Q_7$-rational point.
    
    \begin{figure}[ht]
\centering
\begin{subfigure}{0.5\textwidth}
  \centering
    \begin{tikzpicture}
    \fill (0,0) circle (1.5pt);
    \fill (0.25,0) circle (1.5pt);
    \fill (0.5,0) circle (1.5pt);
    \fill (1.5,0) circle (1.5pt);
    \fill (1.75,0) circle (1.5pt);
    \fill (2,0) circle (1.5pt);

    \draw (0.25,0)
    ellipse (0.6cm and 0.25cm) node[below, yshift = -0.1cm, xshift = 0.65cm, font=\small]{$\s_1$}node[above, yshift = 0cm, xshift = 0.65cm, font=\small]{$\frac{1}{2}$};
    
    \draw (1.75,0)
    ellipse (0.6cm and 0.25cm) node[below, yshift = -0.1cm, xshift = 0.65cm, font=\small]{$\s_2$}node[above, yshift = 0cm, xshift = 0.65cm, font=\small]{$\frac{1}{2}$};

    \draw (1.1,0) ellipse (1.8cm and 0.9cm) node[below, yshift = -0.4cm, xshift = 1.75cm,font=\small]{$\Rcal$} node[above, yshift = 0.4cm, xshift = 1.75cm,font=\small]{$0$};
    \end{tikzpicture}
  \caption{Cluster picture $\Sigma_{C/K}$.}
  \label{fig::egfrobintro}
\end{subfigure}%
\begin{subfigure}{.5\textwidth}
  \centering
  \begin{tikzpicture}
    \draw (0,0)[line width = 0.5mm] -- node[above, font=\small] {$4$} node[right, font=\small, xshift=1.2cm] {$\Gamma_{\s_1}$}  ++ (2.5,0);
    \draw (0.25,-0.2) -- node[left, font=\small] {$1$} ++ (0,1);
    \draw (0.75,-0.2) -- node[left, font=\small] {$2$} ++ (0,1);
    \draw (2.25,0.2) -- node[left, font=\small] {$2$} ++ (0,-1);
    \draw (0,-0.6)[line width = 0.5mm] -- node[below, font=\small] {$4$}  node[right, font=\small, xshift=1.2cm] {$\Gamma_{\s_2}$}  ++ (2.5,0);
    \draw (0.25,-0.4) -- node[left, font=\small] {$1$} ++ (0,-1);
    \draw (0.75,-0.4) -- node[left, font=\small] {$1$} ++ (0,-1);
    \end{tikzpicture}
  \caption{Special fibre of the minimal SNC model of $C/K$.}
  \label{fig::eg5model}
\end{subfigure}
\caption{$C:y^2=(x-i)\left((x-i)^2 - p\right)(x+i)\left((x+i)^2 - p\right)$ over $K=\Q_7$.}
\label{fig::eg5frob}
\end{figure}
\end{eg}

The paper is structured as follows: in Sections \ref{sec::clusters} - \ref{sec::modsusingnewtpolys}, we restate key definitions and theorems from literature, which we will make use of in the remainder of the paper. We start with a brief introduction to cluster pictures in Section \ref{sec::clusters}, before moving onto look at models and the methods used to calculate them in Section \ref{sec::models}. Results from \cite{Dok18} concerning the use of Newton polytopes will be discussed in Section \ref{sec::modsusingnewtpolys}. In Sections \ref{sec::tpgr} and \ref{sec::upc}, we calculate the minimal SNC model for two special cases. The first of these special cases, Section \ref{sec::tpgr}, is where $C$ has \textit{tame potentially good reduction} - that is, it has a smooth model over a tame extension of $K$. This will act as a base case for our eventual proof by induction. The second of these cases, Section \ref{sec::upc}, examines curves $C$ with a cluster picture which consists of exactly two proper clusters $\s < \Rcal$. Curves with such cluster pictures are used to deduce the linking chains between central components in the main theorems. These main theorems are stated and proved in Section \ref{sec::mainthm}. 

\subsection{Notation}
For the convenience of the reader, the following two tables collate the general notation and terminology which we make use of throughout the paper. Table \ref{tab::generalnotation} lists the general notation associated to fields, hyperelliptic curves, and models. Table \ref{tab::notation} lists the notation and terminology associated to cluster pictures and Newton polytopes. Whenever a component in a figure is drawn in bold it is a central component. In any figure describing the special fiber of a model numbers indicate multiplicities, except those preceded by $g$, which indicate the genus of a component. So $2$ indicates a rational curve of multiplicity $2$ and $2g1$ indicates a genus $1$ curve of multiplicity $2$.

\begin{center}
\renewcommand{\arraystretch}{1}
\begin{longtable}[h]{c p{5.6cm} p{0.1cm} c p{5.6cm}}
    
    $K$ & non-archimedean field
    & & 
    $v_K$ & discrete valuation\\
    
    $\OO_K$ & ring of integers
    & &
    $\pi_K$ & uniformiser of $K$ \\
    
    $k$ & algebraically closed residue field of $K$ 
    & &
    $\overline{K}$ & algebraic closure of $K$\\
    
    $C$ & hyperelliptic curve over $K$ given by $y^2=f(x)$
    & &
    $L$ &  field extension of $K$ over which $C_L$ is semistable\\
    
    $g(C)$ & genus of $C$, sometimes denoted $g$
    & &
    $\Rcal$ & set of roots of $f(x)$ in $\overline{K}$
    \\
    
    $e$ & degree of $L/K$ for such $L$
    & & 
    $X$ & Galois orbit of clusters  \\
    
    $\X$ & minimal SNC model of $C/K$ 
    & & 
    $\X_k$ & special fibre of $\X$ \\
    
    $\Gamma^{\pm}_{\s,K}$ &  component(s) of $\X_k$ associated to $\s$ 
    & &
    $\Gamma^{\pm}_{X,K}$ & component(s) of $\X_k$ associated to $X$\\
    
    $\Y$ & minimal SNC model of $C/L$ 
    & &
    $\Y_k$ & special fibre of $\Y$ \\
    
    $\Gamma^{\pm}_{\s,L}$ &  component(s) of $\Y_k$ associated to $\s$ & & & \\
     
\caption{General notation associated to fields, hyperelliptic curves, and models}
\label{tab::generalnotation}
\end{longtable}
\end{center}

\vspace{10px}

\begin{center}
\begin{longtable}[h]{c p{2.5cm} c p{2.5cm} c l}
    $\Sigma_{C/K}$ & (\ref{def::clusterintro}) & $\delta(\s,\s')$ & (\ref{def::dist}) & $\mathrm{red}_{\s}$ & (\ref{def::red})\\
    
    $\s$ & (\ref{def::clusterintro}) & principal & (\ref{def::principal}) & $\Delta(C)$ & (\ref{def:newtonpolys})\\
    
    $d_\s$ & (\ref{def::clusterintro}) & $\s^*$ & (\ref{def::sstar}) & $\Delta_v(C)$ & (\ref{def:newtonpolys})\\
    
    $a_\s, b_{\s}$ & (\ref{def::clusterintro}) & $\ssgs$ & (\ref{def::clustergenus}) & $v_{\Delta}$ & (\ref{def:newtonpolys})\\
    
    odd cluster & (\ref{def::oddeventwin}) & singleton & (\ref{def:sing}) & $L,F$ & (\ref{def::vedgesandfaces})\\
    
    even cluster & (\ref{def::oddeventwin}) & $\s_{\textrm{sing}}$ & (\ref{def:sing}) & $\Delta(\Z),L(\Z),F(\Z)$ & (\ref{not::L(Z)etc})\\
    
    twin & (\ref{def::oddeventwin}) & $\nu_{\s}$ & (\ref{def::clusterfunctions1}) & $\overline{\Delta}(\Z)\overline{L}(\Z),\overline{F}(\Z)$ & (\ref{not::L(Z)etc})\\
    
    $\s'<\s$ & (\ref{def::parentchild}) & $\chi$ & (\ref{def::clusterfunctions2}) & $\delta_{\lambda}$ & (\ref{def:denominatorofvedgeorface})\\
    
    $P(\s)$ & (\ref{def::parentchild}) & $\lambda_{\s} $ & (\ref{def::clusterfunctions2}) & $s_1^L, s_2^L$ & (\ref{def:slope})\\
    
    $\widehat{\s}$, $\widetilde{\s}$ & (\ref{def::parentchild}) & $\alpha_{\s}$ & (\ref{def::clusterfunctions2}) & $g(\s)$ & (\ref{def::esgs})\\
    
    cotwin & (\ref{def:cotwin})  & $\beta_{\s}$ & (\ref{def::clusterfunctions2}) & principal orbit & (\ref{def::principalorbit})\\
    
    \"ubereven & (\ref{def:ubereven})& $\gamma_\s$ & (\ref{def::clusterfunctions2}) & $\lambda_X$ & (\ref{def::orbinvars})\\
    
    $z_{\s}$ & (\ref{def::centre}) & $\theta_\s$ & (\ref{def::clusterfunctions2}) & $K_X$ & (\ref{def::K_X})\\
    
    $\s\wedge\s'$ & (\ref{def::wedge}) & $\epsilon_\s$ & (\ref{def::clusterfunctions2}) & $e_X$ & (\ref{def::e_X})\\
    
    $\delta_{\s}$ & (\ref{def::reldepth}) & $c_\s$ & (\ref{def::red}) & $g(X)$ & (\ref{def::e_X}) \\

\caption{Notation associated to cluster pictures and Newton polytopes}
\label{tab::notation}
\end{longtable}
\end{center}
\vspace{-20px}

{\bf{Acknowledgements.}} 
We would like to thank Vladimir Dokchitser for suggesting this problem to us and for extensive support and guidance. We would also like to thank Tim Dokchitser and Adam Morgan for helpful conversations. This work was supported by the Engineering and Physical Sciences Research Council [EP/L015234/1], the EPSRC Centre for Doctoral Training in Geometry and Number Theory (The London School of Geometry and Number Theory), University College London, and King's College London. 

\section{Background - Cluster Pictures}
\label{sec::clusters}

Let $C/K$ be a hyperelliptic curve given by Weierstrass equation $y^2=f(x)$, with genus $g(C)\geq 1$. The $p$-adic distances between roots of $f(x)$ contain a large amount of useful information. To visualise these $p$-adic distances we use \textit{cluster pictures}, as described in \cite{DDMM18}. In this section we will outline the key definitions we require for this paper concerning cluster pictures; the interested reader can find more in \cite{DDMM18}. 


Recall the definitions of clusters and cluster pictures given in Definition \ref{def::clusterintro}. The cluster picture $\Sigma_C$ is a way of visualising which roots of $f$ are $p$-adically close. In a non-archimedean algebra, two discs either have a non empty intersection or one is contained in the other. So Definition \ref{def::clusterintro} gives us that any two clusters are either disjoint or is one contained in the other. Moreover $d_{\s'}>d_\s$ if $\s'\subsetneq\s$. Every root is a cluster - that is, $\{r\}\in\Sigma_C$ for every $r\in\Rcal$ - and $\Rcal\in\Sigma_C$. In order to work with clusters we need a significant amount of terminology from \cite{DDMM18} which we describe here.

\begin{definition}
\label{def::oddeventwin}
    A cluster $\s$ is \emph{even} (resp. \emph{odd}) if $|\s|$ is even (resp. odd). Furthermore $\s$ is a \emph{twin} if $|\s|=2$. 
\end{definition}

\begin{definition}
\label{def::parentchild}
\label{def:ubereven}
\label{def:cotwin}
    Let $\s $ be a cluster. If $\s' \subsetneq \s$ is a maximal subcluster of $\s$ then $\s'$ is a \textit{child} of $\s$ and $\s$ is a \textit{parent} of $\s'$. We write $\s' < \s$, and $P(\s')=\s$. Denote by $\widehat{\s}$ the set of all children of $\s$, and by $\widetilde{\s}$ the set of all odd children. A cluster is \emph{\"ubereven} if it only has even children. A cluster $\s$ is a \textit{cotwin} if it has a child of size $2g$ whose complement is not a twin.
\end{definition}



\begin{definition}
\label{def::centre}
    A \emph{centre} $z_\s$ of a proper cluster $\s$ is any element $z_\s\in \overline{K}$ such that $v_K(z_\s-r)\geq d_\s$ for all $r\in\s$. Equivalently, $z_\s$ is a centre of $\s$ if $\s$ can be written as $D\cap\mathcal{R}$, where $D = z_\s + \pi^{d_\s}\mathcal{O}_{\overline{K}}$. Note that any root $r\in\s$ can be chosen as a centre, and if $\s=\{r\}$ then the only centre is $z_\s=r$.
\end{definition}

\begin{definition}
\label{def::wedge}
    For clusters $\s$ and $\s'$, write $\s\wedge\s'$ for the smallest cluster containing $\s$ and $\s'$.
\end{definition}

\begin{definition}
\label{def::dist}
\label{def::reldepth}
    If $\s$ and $\s'$ are two clusters then the \textit{distance} between them is $\delta(\s,\s')=d_\s+d_{\s'}-2d_{\s\wedge\s'}$. For a proper cluster $\s\neq \Rcal$ define the \textit{relative depth} to be $\delta_\s=\delta(\s,P(\s)) = d_\s-d_{P(\s)}$.
\end{definition}

\begin{definition}
\label{def::principal}
    A cluster $\s$ is \textit{principal} if $|\s|\geq 3$ except if either $\s=\mathcal{R}$ is even and has exactly two children, or if $\s$ has a child of size $2g$.
\end{definition}

We will see later that principal clusters form an important class of clusters. Roughly, if $C/K$ is a hyperelliptic curve, every orbit of principal clusters in $\Sigma_{C/K}$ makes a contribution to the minimal SNC model of $C$ over $K$.

\begin{definition}
\label{def::sstar}
    For a cluster $\s$ that is not a cotwin we write $\s^*$ for the smallest cluster containing $\s$, whose parent is not \"ubereven. If no such cluster exists we write $\s^*=\Rcal$. If $\s$ is a cotwin, we write $\s^*$ for its child of size $2g$.
\end{definition}

\begin{definition}
\label{def::clustergenus}
    For a proper cluster $\s$ we write $\ssgs$ for the \emph{semistable genus of $\s$}. If $\s$ is \"ubereven, we set $\ssgs=0$. Otherwise, if $\s$ is not \"ubereven the genus is determined by
    $$|\tilde{\s}|= 2\ssgs+1,\;\textrm{or}\; 2\ssgs+2.$$
\end{definition}

It is important to note that $\ssg{\Rcal}$ is not necessarily the same as $g(C)$; in fact, they will only be the same when $\Rcal$ has no proper children. If $C$ has semistable reduction over $L$ and $\s\in\Sigma_{C/K}$ is principal, the semistable genus of $\s$ represents the genus of the contribution of $\s$ to the special fibre of the minimal semistable model of $C$ over $L$. 

We also need some new terminology, and the reminder of the definitions in this section are not given in \cite{DDMM18}. 

\begin{definition}
    A cluster picture $\Sigma$ is \textit{nested} if for all proper clusters $\s,\s'\in\Sigma$ either $\s\subseteq\s'$, or $\s'\subseteq\s$. If $C$ is a hyperelliptic curve, we say $C$ is \emph{nested} if $\Sigma_C$ is nested.
\end{definition}

Since the elements of $\RR$ lie in $\overline{K}$, there is a natural action of $\GK$ on $\RR$, hence also on $\Sigma_C$. Since $K$ has algebraically closed residue field, $\GK = I_K$ where $I_K$ is the inertia subgroup of $\GK$. It will be important later to know exactly how $\GK$ acts on the clusters of $\Sigma_C$. The following lemma is useful for this purpose.

\begin{lemma}
    \label{lem::orbitssizeb}
    Let $\Sigma_C$ be such that $K(\Rcal)/K$ is a tame extension, and let $\s \in \Sigma_C$ be a proper cluster fixed by $\GK$.
    \begin{enumerate}
        \item There exists a centre $z_{\s}$ of $\s$ such that $z_{\s} \in K$.
        \item Any child $\s' < \s$ is in an orbit of size $b_{\s}$, except possibly for one child $\s_f$, where we can choose $z_{\s_f}$ such that $v_K(z_{\s_f} - z_{\s}) > d_{\s}$, which is fixed by $\GK$.
    \end{enumerate}
\end{lemma}

\begin{proof}
(i) See \cite[Lemma~B.1]{DDMM18}.

(ii) See \cite[Theorem~1.3]{Bis19}.
\end{proof}

\begin{definition}
    Let $\s'<\s$ be clusters in $\Sigma_C$. Then $\s'$ is a \emph{stable child} of $\s$ if the stabiliser of $\s$ also stabilises $\s'$. Otherwise $\s'$ is an \emph{unstable child} of $\s$.
\end{definition}

\begin{remark}
    Let $\s\in\Sigma_C$ be fixed by $\GK$. If $\s$ has depth $d_\s$ with denominator $>1$ then, by Lemma \ref{lem::orbitssizeb} ii), $\s$ has at most once stable child. 
\end{remark}

\begin{definition}
\label{def:sing}
If $r\in\s$ is a root which is not contained in a proper child of $\s$ then we call $r$ a \emph{singleton} of $\s$. Define $\singletonsofs$ to be the set of all singletons of $\s$. In other words $\singletonsofs$ is the set of all children of size 1 of $\s$.
\end{definition}

\section{Background - Models}
\label{sec::models}

Let $C$ be a hyperelliptic curve over $K$. A \textit{model} $\X$ of $C$ is a flat scheme over $\OOO$ which has generic fibre $\X_K$ isomorphic to $C/K$. We will insist that all of our models are proper over $\OOO$. \textit{Strict normal crossing (SNC) models} are models with are regular as schemes and whose special fibre $\X_k$ is an SNC divisor - that is, a curve over $k$ whose worst singularities are normal crossings. Note that we do not insist that the irreducible components themselves are smooth. For a given curve, there is a unique SNC model $\X^{\min}$ which is minimal in the sense that any map of SNC models $\X^{\min} \rightarrow \X$ is an isomorphism (\cite[Proposition~9.3.36]{Liu02}).

Another class of models that are of particular interest to us are semistable models. These are SNC models which have a reduced special fibre. Curves which have a semistable model are said to have \textit{semistable reduction}. The minimal SNC models of such curves can be calculated explicitly from the cluster picture, this is done in \cite{DDMM18}.

In this section we collate some facts about models from \cite{Lor90}, \cite{CES03}, and \cite{DDMM18} for the convenience of the reader. Similar techniques concerning quotients of models are also used in \cite{Hal10}.

\subsection{Chains of Rational Curves}

Chains of rational curves are ubiquitous in our descriptions of SNC models. The following definition explains what we mean by a chain of rational curves and defines the three main types of chains that we are interested in: \textit{tails}, \textit{linking chains} and \textit{crossed tails}. 

\begin{definition}
    \label{def::chainrationalcurves}
    Let $\X$ be a SNC model of a hyperelliptic curve defined over $K$. Suppose $E_1,\ldots,E_{\lambda}$ are smooth irreducible rational components of $\X_k$. A divisor $\mathcal{C} = \bigcup_{i=1}^{\lambda} E_i$ is a \textit{chain of rational curves} if
    \begin{enumerate}
        \item $(E_i\cdot E_{i+1}) = 1$ for all $1\leq i < \lambda$ and $(E_i \cdot E_j) = 0$ for $j \neq i+1$,
        \item $(E_1\cdot \overline{\X_k \setminus \mathcal{C}}) = 1$,
        \item $(E_i\cdot \overline{\X_k\setminus \mathcal{C}}) = 0$ for $i \neq 1,\lambda $,
    \end{enumerate}
    where $(E\cdot F)$ is the usual intersection pairing defined on regular models. If $(E_{\lambda} \cdot \overline{\X_k\setminus \mathcal{C}}) = 0$ then $\mathcal{C}$ is a \textit{tail}. If $(E_{\lambda} \cdot \overline{\X_k\setminus \mathcal{C}}) = 1$ then $\mathcal{C}$ is a \textit{linking chain}. 
    
    We say a chain of rational curves $\mathcal{C}=\bigcup_{i=1}^{\lambda}E_i$ is a \emph{loop} if $\mathcal{C}$ is a linking chain such that $E_1$ and $E_{\lambda}$ both intersect the same component of $\overline{\X_k\setminus \mathcal{C}}$.
    
    Furthermore, if $(E_{\lambda} \cdot \overline{\X_k\setminus \mathcal{C}}) = 2$ then $\mathcal{C}$ is a \textit{crossed tail} if $E_{\lambda}$ intersects two rational components of $\X_k \setminus \mathcal{C}$, say $E_{\lambda+1}^+$ and $E_{\lambda+1}^-$, such that $(E_{\lambda+1}^{\pm}\cdot E_{\lambda}) = 1$ and $(E_{\lambda+1}^{\pm} \cdot \overline{\X_k \setminus E_{\lambda}}) = 0$. We call the components $E_{\lambda+1}^{\pm}$ the \textit{crosses}.
\end{definition}

Illustrations of the definitions of tails, linking chains, loops, and crossed tails are shown in Figures \ref{fig::linkingchainsandloops} and \ref{fig::tailsandcrossed tails} in Section \ref{sec::intro}.

Blowing down a component $E$ results in a regular model if and only if it is rational and has self intersection $-1$ (Castelnuovo's Criterion, \cite[Theorem~9.3.8]{Liu02}). However, blowing down a general rational component of $\X_k$ of self intersection $-1$ will not necessarily produce an SNC model. For example blowing down the component of multiplicity $3$ in the minimal SNC model of an elliptic curve of Kodaira type IV (shown in Figure \ref{fig::typeIV} below) 
\vspace{-10px}
\begin{figure}[h]
\centering
\begin{subfigure}{.5\textwidth}
  \centering
  \begin{tikzpicture}
    \draw (0,0) -- node[above,font=\small,xshift=0.7cm] {$3$} ++ (2.75,0);
    \draw (0.25,-0.2) -- node[left,font=\small, yshift=0.2cm] {$1$} ++ (0,1);
    \draw (0.75,-0.2) -- node[left,font=\small, yshift=0.2cm] {$1$} ++ (0,1);
    \draw (1.25,-0.2) -- node[left,font=\small, yshift=0.2cm] {$1$} ++ (0,1);
    \end{tikzpicture}
    \caption{Special fibre of the minimal\\ SNC model.}
    \label{fig::typeIVminSNC}
\end{subfigure}%
\begin{subfigure}{.5\textwidth}
  \centering
  \begin{tikzpicture}
  
  \draw (0.2,0) -- node[below,font=\small,xshift=0.8cm] {$1$} ++ (1.4,0);
    \draw (0.9,0) -- node[right,font=\small, xshift=0.1cm,yshift=0.1cm] {$1$} ++ (0.4,0.6);
    \draw (0.9,0) -- ++ (-0.4,-0.6);
    \draw (0.9,0) -- ++ (0.4,-0.6);
    \draw (0.9,0) -- node[above,font=\small] {$1$} ++ (-0.4,0.6);

    \end{tikzpicture}
    \caption{Special fibre of the minimal\\ regular model.}
    \label{fig::typeIVminreg}
\end{subfigure}
\caption{Elliptic curve of Kodaira Type IV.}
\label{fig::typeIV}
\end{figure}
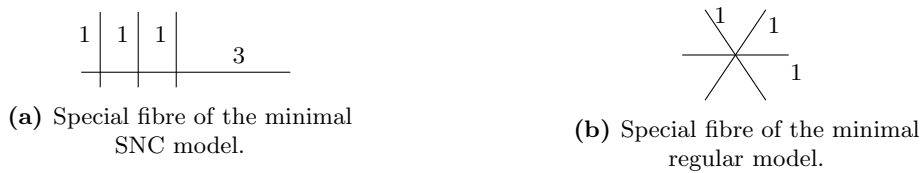
\vspace{-10px}
is no longer an SNC model. After blowing down a component of a chain of rational curves of self-intersection $-1$, the special fibre is still an SNC divisor. Therefore, we will be interested in blowing down all such components. If a chain of rational curves cannot be blown down any further, we call it \textit{minimal}.

\begin{definition}
    A chain of rational curves $\mathcal{C} = \bigcup _{i=1}^{\lambda}E_i$ is \textit{minimal} if $(E_i\cdot E_i) \leq -2$ for every $i$. 
\end{definition}

\subsection{Quotients of Models}
\label{subsec::QuotientsofModels}

This section collates several results from \cite{Lor90} and \cite{CES03} concerning taking quotient of models. Let $C$ be a hyperelliptic curve over $K$ and let $L/K$ be a tame field extension of degree $e$ such that $C_L = C \times_K L$ is semistable over $L$. Note that the cluster picture of $C_L/L$ is the same as the cluster picture of $C/K$, except all the depths have been multiplied by $e$. Since $k$ is algebraically closed, the extension $L/K$ is totally tamely ramified, hence $L/K$ is Galois with $\Gal(L/K)$ cyclic.

Let $\Y$ be the minimal semistable model of $C_L/\OO_L$, so $\Y_k$ is a reduced, SNC divisor of $\Y$. Any $\sigma \in \Gal(L/K)$ induces a unique automorphism of $\Y$ of the same degree which makes the following diagram commute \cite[p.~136]{Lor90}:

\begin{center}
\begin{tikzcd}
    \Y \arrow[r, "\sigma"] \arrow[d] & \Y \arrow[d] \\
    \OO_L \arrow[r, "\sigma"] & \OO_L
\end{tikzcd}    
\end{center}

Although a slight abuse of notation, we will also refer to this automorphism on $\Y$ as $\sigma$, and define $G=\langle \sigma:\Y\to\Y \rangle$ where $\sigma$ generates $\Gal(L/K)$. The model $\Y$, as well as the automorphism induced on the special fibre, will be given more explicitly in Section \ref{subsec::SemistableModels}. 

Since $\Y$ is projective, the quotient $\mathscr{Z} = \Y/G$ given by $q : \Y \rightarrow \mathscr{Z}$ can be constructed by glueing together the rings of invariants of $G$-invariant affine open sets of $\Y$. The resulting scheme $\ZZ/\OOO$ is a model of $C/K$. Furthermore, since $\ZZ$ is a normal scheme, its singularities are closed points lying on the special fibre $\ZZ_k$. The following proposition, from \cite[p.~137]{Lor90}, gives the multiplicities of the components of $\ZZ_k$.

\begin{prop}
    \label{prop::componentsmult}
    Let $Y \subseteq \Y_k$ be an irreducible component of $\Y_k$. Then $Z = q(Y)$ is a component of $\ZZ_k$ of multiplicity $e/|\textrm{Stab}(Y)|$, where $\textrm{Stab}(Y)$ is the pointwise stabiliser of $Y$. 
\end{prop}

Blowing up a singularity on $\ZZ_k$ results in a chain of rational curves, as in Definition \ref{def::chainrationalcurves}. It is well known (e.g. \cite[Fact~V]{Lor90}, \cite{Lip78}) that blowing up the singularities on $\ZZ_k$ and blowing down all rational components in chains with self intersection -1 results in a minimal SNC model.


	

The singularities on $\ZZ_k$ are \textit{tame cyclic quotient singularities}, and there is a precise description of the chain of rational curves that arises after resolving them. We will prove in Proposition \ref{prop::singPointsAretsgs} that singularities $z\in\ZZ_k$ which lie on precisely one irreducible component of $\ZZ_k$ are tame cyclic quotient singularities. The definition is as follows:

\begin{definition}
    \label{def::tcqs}
    Let $S$ be a scheme over $\OO_K$ and let $s\in S$ be a closed point. The point $s$ is a \textit{tame cyclic quotient singularity} if there exists
    \begin{itemize}
        \item[--] a positive integer $m>1$ which is invertible in $k$,
        \item[--] a unit $r \in (\Z/m\Z)^{\times}$,
        \item[--] integers $m_1 > 0$ and $m_2 \geq 0$ satisfying $m_1 \equiv -rm_2 \mod m$
    \end{itemize}
     such that $\widehat{\OO_{S,s}}$ is isomorphic to the subalgebra of $\mu_m$-invariants in $k\llbracket t_1,t_2 \rrbracket /(t_1^{m_1}t_2^{m_2} - \pi_K)$ under the action $t_1 \mapsto \zeta_m, t_2 \mapsto \zeta_m^r$. We will call the pair $(m,r)$ the \textit{\tcqi} of $s$.
\end{definition}

The following theorem, \cite[Theorem~2.4.1]{CES03}, tells us how to resolve tame cyclic quotient singularities.

\begin{thm}
    \label{thm::resolvingtcqs}
    Let $S$ be a flat, proper, normal curve over $\OO_K$ with smooth generic fibre. Suppose $s\in S_k$ is a tame cyclic quotient singularity with \tcqi\; $(m,r)$, as in Definition \ref{def::tcqs} above.
    
    Consider the Hirzebruch-Jung continued fraction expansion of $\frac{m}{r}$ given by
    \begin{align*}
        \frac{m}{r}=b_{\lambda}-\frac{1}{b_{\lambda-1}-\frac{1}{\cdots\;-\frac{1}{b_1}}},
    \end{align*}
    where $b_i \geq 2$ for all $1\leq i\leq \lambda$. Then the minimal regular resolution of $s$ is a chain of rational curves $\bigcup_{i=1}^{\lambda} E_i$ such that $E_i$ has self intersection $-b_i$.
\end{thm}

\begin{remark}
    Note that in \cite{CES03} the $E_i$ are labeled in the opposite order. Instead we use the same labeling of the components in our chain as in both \cite{Dok18}, and \cite{Lor90}.
\end{remark}

\subsection{Semistable Models}
\label{subsec::SemistableModels}


A critical step in the proof of the main theorem in this paper will be extending the field so that $C$ has semistable reduction. The following theorem, a criterion for $C$ to have semistable reduction in terms of the cluster picture of $C$, allows us to do just that. First we need the following definition:
\begin{definition}
\label{def::clusterfunctions1}
    For a proper cluster $\s\in\Sigma_{C}$ define
       $ \nu_{\s} = v_K(c_f) + \sum_{r \in \mathcal{R}} d_{\s \wedge r}.$
\end{definition}

\begin{thm}[The Semistability Criterion]
    \label{thm::sscriterion}
Let $C:y^2=f(x)$ be a hyperelliptic curve, and let $\Rcal$ be the set of roots of $f(x)$ in $\overline{K}$. Then $C$ has semistable reduction over $L$ if and only if 
	\begin{enumerate}
		\item the extension $L(\mathcal{R})/L$ has ramification degree at most $2$,
		\item every proper cluster of $\Sigma_{C/L}$ is $I_L$ invariant,
		\item every principal cluster $\s\in\Sigma_{C/L}$ has $\depth \in \Z$ and $\clusleadcoeff \in 2\Z$.
    \end{enumerate}
\end{thm}

Once the field has been extended so that $C$ has semistable reduction, there is a very explicit description of the special fibre of $\Y$ in terms of the cluster picture of $C$ in \cite[Theorem~8.5]{DDMM18}. For this we need some definitions. To simplify some invariants, we assume that all clusters $\s \in \Sigma_{C/K}$ have $e\delta_{\s} > \frac{1}{2}$, since a cluster $\s$ with $e\delta_{\s} = 1/2$ introduces singular irreducible components. This will be sufficient for our purposes since these invariants are used to describe the explicit automorphism on $\Y_k$ and we can always extend our field so that the minimal semistable model has no singular components. Note that the valuation on $\overline{K}$ is normalised with respect to $K$, such that the valuation of a uniformiser $\pi_L$ of $L$ is $v_K(\pi_L)=\frac{1}{e}$.

\begin{definition}
\label{def::clusterfunctions2}
    For $\sigma\in\GK$ let
    \begin{align*}
        \chi(\sigma) = \frac{\sigma(\pi_L)}{\pi_L} \mod \mathfrak{m}.
    \end{align*}
    For a proper cluster $\s\in\Sigma_{C}$ define
    \begin{align*}
        \lambda_{\s} &= \frac{\nu_{\s}}{2} - d_{\s} \sum_{\s'<\s} \bigg\lfloor \frac{|\s'|}{2} \bigg\rfloor.
    \end{align*}
    Define $\theta_\s=\sqrt{c_f\prod_{r\notin \s}(z_\s-r)}.$ If $\s$ is either even or a cotwin, we define $\epsilon_\s: \GK\to\{\pm1\}$ by $$\epsilon_\s(\sigma)\equiv\frac{\sigma(\theta_{\s^*})}{\theta_{(\sigma\s)^*}} \mod \mathfrak{m}.$$
    For all other clusters $\s$ set $\epsilon_\s(\sigma)=0$. We write $\epsilon_{\s}$ without reference to any $\sigma\in \Gal(L/K)$ for $\epsilon_{\s}(\sigma)$, $\sigma \in \Gal(L/K)$ a generator. \cite[Definition~8.2]{DDMM18} 
\end{definition}

\begin{remark}
\label{rem::epsilon}
    The quantity $\epsilon_{\s}(\sigma) = -1$ if and only if $\sigma$ swaps the two points at infinity of $\Gamma_{\s,L}$. When $k=\overline{k}$ $\epsilon_{\s} = (-1)^{\nu_{\s^*} - |\s^*|d_{\s^*}}$ since $$\nu_{\s^*} = v_K\left(c_f\prod_{r\notin \s^*}(z_{\s^*}-r)\right) + \sum_{r \in \s^*} d_{\s^*}.$$
\end{remark}

\begin{remark}
    Our definition of $\lambda_\s$ differs slightly to that in \cite{DDMM18}. In \cite{DDMM18} $\lambda_\s$ is defined to be $\frac{\nu_\s}{2}-d_\s\sum_{\s'<\s,\delta_{\s'}>\frac{1}{2}}\lfloor \frac{|\s'|}{2}\rfloor$, and a second quantity $\tilde{\lambda}_\s$ is defined to be $\frac{\nu_\s}{2}-d_\s\sum_{\s'<\s}\lfloor \frac{|\s'|}{2}\rfloor$. This is to account for singular components of the special fibre. Given our assumption that every cluster in $\Sigma_{C/L}$ has relative depth $>\frac{1}{2}$, when we calculate these for $C/L$ we find that $\lambda_\s=\tilde{\lambda}_\s$, so for simplicity we do not write the tilde. 
\end{remark}

\begin{definition}
\label{def::red}
Let $\s\in\Sigma_{C/K}$ be a principal cluster. Define $c_{\s}\in k^{\times}$ by 
$$c_{\s} = \frac{c_f}{\pi_L^{v_K(c_f)}} \prod_{r \not \in \s} \frac{z_{\s} - r}{\pi_L^{v_K(z_{\s} - r)}} \mod \mathfrak{m}.$$
\end{definition}

These definitions are key for the description of the minimal SNC model of $C$. In the interest of brevity, we will not restate \cite[Theorem~8.5]{DDMM18} here, which is a simplification of Theorem \ref{thm::structureofSNCmodelintro} to the case of semistable reduction, and also gives the action of $\Gal(\overline{K}/K)$ on the minimal SNC model. However, we recommend that the reader familiarise themselves with this theorem as it is helpful for understanding the case when $C$ does not have semistable reduction. The main idea is that principal non-\"ubereven (resp. \"ubereven) clusters each have one (resp. two) components associated to them, and components of parents and odd (resp. even) children are linked by one (resp. two) chain(s) of rational curves. The Galois action on components is induced by the Galois action on clusters, and the two components (resp. two linking chains) of an \"ubereven cluster (resp. even child) $\s$ are swapped precisely when $\epsilon_{\s} = -1$.

\section{Background - Models of Curves via Newton polytopes}
\label{sec::modsusingnewtpolys}
 
In this section we describe a method from \cite{Dok18} for calculating a SNC model of a curve $C/K$ which is $\Delta_v$-regular. The notion of $\Delta_v$-regularity, given in \cite[Definition~3.9]{Dok18}, applies to more general smooth projective curves. Here we restrict to the case where $C$ has a nested cluster picture, and note that this condition implies $\Delta_v$-regularity. The results are applied in Section \ref{sec::upc}. 
 
\subsection{Newton polytopes}
\label{subsec::newtonpolys}

Here we briefly collate the key definitions regarding Newton polytopes necessary for this paper. We begin with the definition of a Newton polytope.

\begin{definition}
\label{def:newtonpolys}
Let $G(x,y) = y^2 - f(x) = \sum a_{ij}x^iy^j$ be the defining equation of a hyperelliptic curve $C$ over $K$. The \textit{Newton polytopes} of $C$ over $K$ and $\OOO$ respectively are:  
\begin{align*}
  \Delta(C)&=\textrm{convex hull }\{(i,j)\mid a_{ij}\neq 0\}\subseteq \R^2,\\
  \Delta_v(C)&=\textrm{lower convex hull }\{(i,j,v_K(a_{ij}))\mid a_{ij}\neq 0\}\subseteq \R^2\times \R.
\end{align*}
\end{definition}
Above every point $P\in\Delta$ there is exactly one point $(P,v_{K}(P))\in \Delta_v$. This defines a piecewise affine function $v_{\Delta(C)}:\Delta(C)\to \R$. When there is no risk of confusion, we may sometimes write $\Delta=\Delta(C)$, and $\Delta_v=\Delta_v(C)$ and the pair $(\Delta, v_{\Delta})$ determines $\Delta_v$. \cite[\S~3]{Dok18}

\begin{definition}
\label{def::vedgesandfaces}
Under the homeomorphic projection $\Delta_v\to \Delta$, the images of the 0, 1, and 2 dimensional open faces of $\Delta_v$ are called \textit{$v$-vertices, $v$-edges}, and \textit{$v$-faces} respectively. Note that, a $v$-vertex is a point in $\Z^2$, a $v$-edge (often denoted $L$) is homeomorphic to an open interval, and a $v$-face (often denoted $F$) is homeomorphic to an open disc.
\end{definition}

\begin{notation}
\label{not::L(Z)etc}
For a $v$-edge $L$ and a $v$-face $F$ we write 
\begin{align*}
    L(\Z)=L\cap\Z^2,\quad
    F(\Z)=F\cap\Z^2,\quad
    \Delta(\Z)=(\Delta^\mathrm{o})\cap\Z^2,
\end{align*}
and $\overline{L}(\Z)$, $\overline{F}(\Z)$, $\overline{\Delta}(\Z)$ to include points on the boundary. 
We use subscripts to restrict to the set of points $P$ with $v_K(P)$ in a given set, for instance
$F(\Z)_{\Z}=\{P\in F(\Z)\mid v_{\Delta}(P)\in\Z\}.$
\end{notation}

\begin{definition}
\label{def:denominatorofvedgeorface}
The \textit{denominator} $\delta_{\lambda}$, for every $v$-face or $v$-edge $\lambda$ is defined to be the common denominator of $v_{\Delta}(P)$ for $P\in\overline{\lambda}(\Z)$.
For two alternate, but equivalent, definitions see \cite[Notation~3.2]{Dok18}.
\end{definition}

\begin{remark}
We shall see that the denominator of a $v$-face or $v$-edge $\lambda$, in some sense, tells us the multiplicity of the component or chain of the SNC model arising from $\lambda$. Roughly, for a $v$-face $F$, $\delta_F$ is the multiplicity of the component $\Gamma_F$, and for a $v$-edge $L$, $\delta_L$ is the minimum multiplicity appearing in the chain of rational curves arising from $L$.
\end{remark}

We distinguish between $v$-edges which lie on precisely one or two $v$-faces of the Newton polytope, the former giving rise to tails and the latter to linking chains.

\begin{definition}
A $v$-edge $L$ is \textit{inner} if it is on the boundary of two
$v$-faces. Otherwise, if $L$ is only on the boundary of one $v$-face, $L$ is \textit{outer}.
\end{definition}

 
\subsection{Calculating a Model}
\label{subsec::timsmethod}

Before we begin, we give a few constants related to $v$-faces and $v$-edges which will be necessary for our description.

\begin{definition}
Let $L$ be a $v$-edge on the boundary of a $v$-face $F$. Write 
$$L^*=L_{(F)}^*=\textrm{ the unique affine function } \Z^2\twoheadrightarrow\Z \textrm{ with } L^*|_L=0, \textrm{ and } L^*|_F\geq0.$$
\end{definition}

\begin{definition}
\label{def:slope}
Let $L$ be a $v$-edge. If $L$ is outer it bounds two $v$-faces, say $F_1$ and $F_2$. If $L$ is inner it bounds one $v$-face, say $F_1$. Choose $P_0, P_1\in\Z^2$ with $L^*_{(F_1)}(P_0)=0$, and $L^*_{(F_1)}(P_1)=1$. The \textit{slopes} $[s_1^L,s_2^L]$ at $L$ are
\begin{align*}
    s_1^L=\delta_L(v_1(P_1)-v_1(P_0)),\quad\textrm{and}\quad &
    s_2^L= \left\{
    \begin{tabular}{cc}
    $\delta_L(v_2(P_1)-v_2(P_0))$ & for $L$ inner,\\
    $\lfloor s_1^L-1\rfloor$ & for $L$ outer,
    \end{tabular}\right.
\end{align*}
where $v_i$ is the unique affine function $\Z^2\to \Q$ that agrees with $v_{\Delta}$ on $F_i$.
\end{definition}

\begin{thm}
\label{thm::timsmainthm}
Suppose $C: y^2 = f(x)$ is a nested hyperelliptic curve over $K$. Then there exists a regular model $C _{\Delta}/\OO_K$ of $C/K$ with strict normal crossings. Its special fibre is as follows:
\begin{enumerate}
    \item Every $v$-face $F$ of $\Delta$ gives a complete smooth curve $\Gamma_F/k$ of multiplicity $\delta_F$ and genus $|F(\Z)_\Z|$.
    \item For every $v$-edge $L$ with slopes $[s_1^L,s_2^L]$ pick $\frac{m_i}{d_i}\in\Q$ such that 
    \begin{equation}
    \label{timseq}
        s_1^L=\frac{m_0}{d_0}>\frac{m_1}{d_1}>\dots>\frac{m_{\lambda}}{d_{\lambda}}>\frac{m_{\lambda+1}}{d_{\lambda+1}}=s_2^L, \textrm{ with } \left|\begin{tabular}{cc}
    $m_i$ & $m_{i+1}$\\
    $d_i$ & $d_{i+1}$
    \end{tabular}\right|=1.
    \end{equation}
Then $L$ gives 
$|\overline{L}(\Z)_{\Z}|-1$ chains of rational curves of length $\lambda$ from $\Gamma_{F_1}$ to $\Gamma_{F_2}$ (if $L$ is outer these chains are tails of $\Gamma_{F_1}$) with multiplicities $\delta_L d_1,\dots,\delta_L d_{\lambda}$. \cite[Theorem~3.13]{Dok18}
\end{enumerate}
\end{thm}

\begin{remark}
    In (\ref{timseq}), $\lambda = 0$ indicates that $\Gamma_{F_1}$ and $\Gamma_{F_2}$ intersect $|\overline{L}(\Z)_{\Z}|  - 1$ times in the inner case, and that $L$ contributes no tails in the outer case. 
\end{remark}

\begin{remark}
    An explicit equation for $\Gamma_F$ is given in \cite[Definition~3.7]{Dok18}, where it is denoted by $\overline{X}_F$. However this is more information than necessary for our situation so we do not give this description here. A description of a similar object $X_L$ is also given in \cite[Definition~3.7]{Dok18}, and in Theorem 3.13 of \cite{Dok18} the number of rational chains that a $v$-edge $L$ gives rise to is described in terms of $|X_L(\overline{k})|$. However, it is straightforward to show that in our case $|X_L(k)|=|L(\Z)_{\Z}|-1$, so we omit this description also.
\end{remark}

\begin{remark}
\label{rem::timsmainrem}
To see that the sequences in Theorem \ref{thm::timsmainthm} exist, take all numbers in $[s_2^L,s_1^L]\cap\Q$ of denominator $\leq \max\{\textrm{denom}(s_1^L),\textrm{denom}(s_2^L)\}$ in decreasing order. This is essentially a Farey series, so satisfies the determinant condition in (\ref{timseq}). One can then repeatedly remove, in any order, terms of the form 
$$\dots>\frac{a}{b}>\frac{a+c}{b+d}>\frac{c}{d}>\dots\mapsto\dots>\frac{a}{b}>\frac{c}{d}>\dots,$$
where $(a+c)$ and $(b+d)$ are coprime, until no longer possible. This corresponds to blowing down $\PP^1$s of self intersection -1 (see Remark 3.16 in \cite{Dok18}). The resulting minimal sequence is unique (else this would contradict uniqueness of minimal SNC model), and still satisfies the determinant condition. If $(s_2^L,s_1^L)\cap\Z=\{N,\dots,N+a\}\neq\emptyset$ this minimal sequence has the form 
\begin{equation}
\label{timseq2}
    s_1^L=\frac{m_0}{d_0}>\dots>\frac{m_h}{d_h}>N+a>\dots>N>\frac{m_l}{d_l}>\dots>\frac{m_{\lambda+1}}{d_{\lambda+1}}=s_2^L, 
\end{equation}
with $d_0,\ldots,d_h$ strictly decreasing and $d_l,\ldots,d_{\lambda+1}$ strictly increasing. If $(s_2^L,s_1^L)\cap\Z=\emptyset$ this minimal sequence has the form 
\begin{equation}
\label{timseq3}
    s_1^L=\frac{m_0}{d_0}>\dots>\frac{m_l}{d_l}>\frac{m_{l+1}}{d_{l+1}}>\dots>\frac{m_{\lambda+1}}{d_{\lambda+1}}=s_2^L, 
\end{equation}
with $d_0,\ldots,d_l$ strictly decreasing, $d_{l+1},\ldots,d_{\lambda+1}$ strictly increasing, and $d_i>1$ for all $1\leq i\leq \lambda$. Notice that shifting either $s_1^L$ or $s_2^L$ by an integer does not change the denominators $d_i$, that appear in this sequence. If $s_2>0$ (else shift by an integer), the numbers $N>\frac{m_l}{d_l}>\dots>\frac{m_{\lambda+1}}{d_{\lambda+1}}$ are the approximants of the Hirzebruch-Jung continued fraction expansion of $s_2^L$, similarly for $\frac{m_0}{d_0}>\dots>\frac{m_h}{d_h}>N+a$ consider the expansion of $1 - s_1^L$. \cite[Remark~3.15]{Dok18}
\end{remark}

\subsection{Sloped Chains}
\label{subsec::linkingchains}

The following definition allows us to talk about different parts of chains of rational curves arising from $v$-edges in the Newton polytope of $C$.

\begin{definition}
\label{def::linkingchainsections}
    Let $t_1,t_2 \in \Q$ and $\mu\in \N$. Pick $m_i, d_i$ as in Theorem \ref{thm::timsmainthm}; that is, such that 
    \begin{equation*}
        \mu t_1=\frac{m_0}{d_0}>\frac{m_1}{d_1}>\dots>\frac{m_{\lambda}}{d_{\lambda}}>\frac{m_{\lambda+1}}{d_{\lambda+1}}= \mu t_2,\;\textrm{and}\;\left|\begin{tabular}{cc}
    $m_i$ & $m_{i+1}$\\
    $d_i$ & $d_{i+1}$
    \end{tabular}\right|=1,
    \end{equation*}
with $d_0\geq \dots \geq d_l$ and $d_l\leq\dots\leq d_{\lambda+1}$, for some $0\leq l \leq \lambda+1$. 

Let $A=\{i\mid 1\leq i\leq\lambda\textrm{ and }d_i=1\}$. If $A$ is non-empty, let $a_0$ be the minimal element of $A$ and let $a_1$ the maximal element of $A$. Suppose $\mathcal{C} = \bigcup_{i=1}^{\lambda} E_i$ is a chain of rational curves where $E_i$ has multiplicity $\mu d_i$. Then $\mathcal{C}$ is a \textit{sloped chain of rational curves} with parameters $(t_2,t_1,\mu)$ and we split $\mathcal{C}$ into three sections. If $A\neq\emptyset$ we define the following:
\begin{enumerate}
    \item $E_1 \cup \cdots \cup E_{a_0-1}$, the \textit{downhill section},
    \item $E_{a_0} \cup \cdots \cup E_{a_1}$, the \textit{level section},
    \item $E_{a_1+1} \cup \cdots \cup E_{\lambda}$, the \textit{uphill section}.
\end{enumerate}
If instead $A=\emptyset$ we define:
\begin{enumerate}
    \item $E_1 \cup \cdots \cup E_{l}$, the \textit{downhill section},
    \item $E_{l+1} \cup \cdots \cup E_{\lambda}$, the \textit{uphill section},
\end{enumerate}
and there is no \emph{level section}. 

We define the length of each section to be the number of $E_i$ contained in it, and each section is allowed to have length 0. For instance, the level section has length 0 if and only if $A=\emptyset$, and the downhill section has length 0 if and only if $1\in A$. 
\end{definition}

\begin{remark}
\label{rem::linkingchainsections}
    A tail is a sloped chain with level section of length 1 and no uphill section. Therefore any tail can be given by just two parameters, namely $t_1$ and $\mu$ (since $t_2 = \frac{1}{\mu}\lfloor \mu t_1 - 1 \rfloor$). We will often refer to a tail as a tail with parameters $(t_1,\mu)$. It follows from Remark \ref{rem::timsmainrem} that a tail with parameters $(t_1,\mu)$ has the same multiplicities as the tail obtained by resolving a tame cyclic quotient singularity with \tcqi\; $\frac{1}{\mu t_1}$.
\end{remark}

\begin{remark}
    All of our chains of rational curves, be they tails, linking chains or crossed tails, are sloped chains. For example, a linking chain in a semistable model will consist of only level section. Both tails and crossed tails in a minimal SNC model will have no uphill section. 
\end{remark}

\section{Curves with Tame Potentially Good Reduction (The Base Case)}
\label{sec::tpgr}

In this section we calculate the minimal SNC model of a hyperelliptic curve $C/K$ with genus $g=g(C)\geq 1$ which has tame potentially good reduction. That is, there exists a field extension $L/K$ of degree $e$ such that $e$ and $p$ are coprime, and $C$ has a smooth model over $\OO_L$. In order to calculate this model, we assume that $L$ is the minimal such extension. The minimal SNC model of a hyperelliptic curve has a rather straightforward description: it consists of a central component with some tails (in the sense of Definition \ref{def::chainrationalcurves}) whose multiplicities can be explicitly described using the results of Section \ref{subsec::QuotientsofModels}. The size and depth of the unique proper cluster $\s$, as well as the valuation of the leading coefficient $c_f$ will be sufficient to calculate the (dual graph with multiplicity of the) minimal SNC model of $C$ over $K$:

\begin{thm}\label{thm::tpgr} 
    Let $C$ be hyperelliptic curve over $K$ with tame potentially good reduction. Then the special fibre $\X_k$ of the minimal SNC model $\X$ of $C/K$ consists of a component $\Gamma=\Gamma_{\s,K}$, the central component, of multiplicity $e$. Furthermore, if $e>1$ then the following tails intersect the central component $\Gamma$:
\begin{center}
\begin{tabular}{|c|c|c|}
    \hline
     Name & Number & Condition  \\ \hline
     $T_{\infty}$ & $1$ & $\s$ odd \\ \hline
     $T_{\infty}^{\pm}$ & $2$ & $\s$ even and $v_K(c_f)$ even \\ \hline
     $T_{\infty}$ & $1$ & $\s$ even, $e>2$ and $v_K(c_f)$ odd \\ \hline
     $T_{y=0}$ & $\frac{|\singletonsofs|}{b_{\s}}$ & $e = 2b_{\s}$ \\ \hline
     $T_{x=0}$ & $1$ & $b_{\s}\mid |\s|$, $\lambda_{\s} \not \in \Z $ and $e> 2$ \\ \hline
     $T_{x=0}^{\pm}$ & $2$ & $b_{\s} \mid |\s|$ and $\lambda_{\s}  \in \Z $ \\ \hline
     $T_{(0,0)}$ & $1$ & $b_{\s} \nmid |\s|$ \\ \hline
\end{tabular}
\end{center}
\end{thm}

\begin{remark}
    The genus of the central component can be calculated using Riemann Hurwitz, and we prove an explicit formula for it in Proposition \ref{prop::genusformula}.
\end{remark}

Since $C$ has tame potentially good reduction, by \cite[Theorem~1.8(3)]{DDMM18} we can assume (possibly after a M\"obius transform) that that the cluster picture of $C$ over $K$ consists of a single proper cluster $\s$. After an appropriate shift of the affine line we can assume that $\s$ is centered around $0$ and that $C$ is given by one of the following two equations:
$$y^2 = c_f\prod_{0\neq r \in \Rcal}(x - u_r \pi^{d_{\s}}),\quad\textrm{or} \quad y^2 = c_fx\prod_{0\neq r \in \Rcal}(x - u_r \pi^{d_{\s}}), $$
where the $u_r\in K$ are units.

We will proceed in the manner of Section \ref{subsec::QuotientsofModels}. Let $\Y$ be the smooth Weierstrass model of $C$ over $L$ (the hyperelliptic curve over $\OO_L$ given by $y^2 = f(x)$), and let $q : \Y \rightarrow \mathscr{Z}$ be the quotient map induced by the action of $\Gal(L/K)$. We will explicitly describe the singular points of $\mathscr{Z}$, show that they are tame cyclic quotient singularities in the sense of Definition \ref{def::tcqs}, and give their tame cyclic quotient invariants in Proposition \ref{prop::singPointsAretsgs}. Theorem \ref{thm::resolvingtcqs} then tells us the self intersection numbers of the rational curves in the tails obtained by resolving the tame cyclic quotient singularities. After using intersection theory, this allows us to describe the special fibre of the minimal SNC model $\X$ of $C/K$ in full. 

\subsection{The Automorphism and its Orbits}
\label{subsec::tpgrauto}

To describe the singularities on $\mathscr{Z}_k$, we must first explicitly describe the Galois automorphism on the unique component $\Gamma_{\s,L}=\Gamma \subseteq \mathscr{Y}_k$ of the special fibre of the smooth Weierstrass model of $C$ over $L$. The following fact from \cite[Fact~IV~p.~139]{Lor90} describes the singularities of $\mathscr{Z}_k$ in terms of the quotient $q : \Y \rightarrow \mathscr{Z}$.
\begin{prop}
    \label{prop::tpgrSings}
	Let $z_1,\ldots,z_d$ be the ramification points of the morphism $q:\Gamma\rightarrow \ZZ_k$. Then $\{z_1,\ldots,z_d\}$ is precisely the set of singular points of $\ZZ_k$.
\end{prop}
Furthermore, the ramification points of $q$ correspond to points whose preimage is an orbit of size strictly less than $e$. 

\begin{definition}
Let $X$ be an orbit of clusters. If $|X|<e$, we say that $X$ is a \emph{small orbit}.
\end{definition}

Hence describing the singular points of $\ZZ_k$ is equivalent to describing the small orbits of $\Gal(L/K)$. In order to list these orbits, we first simplify some cluster invariants from \ref{def::clusterfunctions1} and \ref{def::clusterfunctions2}.

\begin{lemma}
\label{lem::autintpgrcase}
Let $C/K$ be a hyperelliptic curve with tame potentially good reduction and unique proper cluster $\s$. Then:
$$\nu_\s=|\s|d_{\s}+v_K(c_f), \quad
    \lambda_\s = \frac{\nu_\s}{2}= \frac{|\s|d_{\s}+v_K(c_f)}{2}, \quad \epsilon_{\s}=(-1)^{v_K(c_f)},$$
    
    and any $\sigma \in \Gal(\overline{K}/K)$ induces
$$\sigma\vert_{\Gamma} : (x,y) \longmapsto (\chi(\sigma)^{ed_{\s}}x, \chi(\sigma)^{e\lambda_{\s}}y).$$
\end{lemma}
\begin{proof}
    Definitions \ref{def::clusterfunctions1} and \ref{def::clusterfunctions2} and \cite[Theorem~8.5]{DDMM18}.
\end{proof}

Since $\chi(\sigma)^{ed_{\s}}$ and $\chi(\sigma)^{e\lambda_{\s}}$ are non-zero and $k$ is algebraically closed, the only points which can lie in orbits of size strictly less than $e$ are points at infinity, or points where $x=0$ or $y=0$. This gives four cases which we will take care to distinguish between, as it will make it easier to describe the minimal SNC model for a general cluster picture. With this in mind we make the following definitions:

\begin{definition}
    \label{def::orbittypes}
    We split the small orbits that can occur into the following types.
    \begin{itemize}
        \item \emph{$\infty$-orbits}: orbits on the point(s) at infinity,
        \item \emph{$(y=0)$-orbits}: orbits on non-zero roots,
        \item \emph{$(x=0)$-orbits}: orbits on the points $(0,\pm\sqrt{c_f})$,
        \item \emph{$(0,0)$-orbits}: the orbit on the point $(0,0)$.
    \end{itemize}
\end{definition}

The following lemmas describe in which situations we see these small orbits. We will assume $e>1$ since no small orbits occur when $e=1$.

\begin{lemma}
\label{lem::inftyorbits}
    If $\deg(f)$ is odd then there is a single $\infty$-orbit consisting of a single point. If $\deg(f)$ is even and $v_K(c_f) \in 2\Z$ then there are two $\infty$-orbits each of size 1. If $\deg(f)$ is even, $v_K(c_f) \not \in 2\Z$ and $e>2$ then there is a single $\infty$-orbit of size 2.
\end{lemma}
\begin{proof}
    Let $u=1/x, v=y/x^{g+1}$ denote the coordinates at infinity. The curve $C$ has a single point at infinity $(u,v) = (0,0)$ if $\deg(f)$ is odd, and two points at infinity $(u,v) = (0,\pm \sqrt{c_f})$ if $\deg(f)$ is even. In the latter case, Lemma \ref{lem::autintpgrcase} gives the action at infinity $\sigma : (0,\sqrt{\overline{c_f}}) \mapsto (0, \chi(\sigma)^{e \lambda_{\s}}\sqrt{\overline{c_f}})$. Therefore, when $\deg(f)$ is even, the points at infinity are swapped if and only if $\chi(\sigma)^{e\lambda_{\s}} = -1$ for some $\sigma \in \Gal(L/K)$. This is the case if and only if $v_K({c_f})$ is odd. In this case, the orbit at infinity has size 2 and is only a small orbit if $e>2$. 
\end{proof}

\begin{lemma}
\label{lem::x0orbits}
    If $f(0) = 0$ then there is a single $(0,0)$-orbit consisting of a single point. Otherwise $f(0)\neq 0$, and if ${\lambda}_{\s}\in\Z$ then there are two $(x=0)$-orbits of size 1, else ${\lambda}_{\s}\not\in\Z$ and $e>2$ there is a single ($x=0$)-orbit of size 2.
\end{lemma}
\begin{proof}
    If $f(0) = 0$ then $\{(0,0)\} \in \Gamma$ is the unique $(0,0)$-orbit. If $f(0) \neq 0$ then $(0,\pm \sqrt{\overline{c}_f}) \in \Gamma$, and these points are swapped by some element of the Galois group (see Lemma (\ref{lem::autintpgrcase})) if and only if $\lambda_\s \not \in \Z$. If $\lambda_{\s} \not \in \Z$ then the orbit has size $2$ hence it is only a small orbit if $e > 2$.
\end{proof}

\begin{lemma}\label{lem::2bornot2b}
    Either $e=b_{\s}$ or $e=2b_{\s}$, where $b_{\s}$ is the denominator of $d_{\s}$. In particular $e=2b_{\s}$ if and only if $b_{\s}\nu_\s \not \in 2\Z$.
\end{lemma}
\begin{proof}
    By Theorem \ref{thm::sscriterion}, $e$ is the minimal integer such that $ed_{\s} \in \Z$ and $e\nu_\s \in 2\Z$. Since $ed_{\s} \in \Z$, we can deduce that $b_{\s} \mid e$. Since $2b_{\s}\nu_\s \in 2\Z$, $e = b_{\s}$ or $e=2b_{\s}$. We can check that the other conditions of Theorem \ref{thm::sscriterion} are satisfied over a field extension of degree $e$.
\end{proof}

\begin{lemma}
\label{lem::y0orbits}
    If $e>b_{\s}$ then there are $\frac{|\s|}{b_{\s}}$ $(y=0)$-orbits if $b_{\s} \mid |\s|$, or $\frac{|\s|-1}{b_{\s}}$ $(y=0)$-orbits if $b_{\s} \nmid |\s|$.
\end{lemma}
\begin{proof}
    The non-zero points with $y=0$ are of the form $(\zeta_{b_{\s}}^i,0)$ for $\zeta_{b_{\s}}$ a primitive $b_{\s}^{\textrm{th}}$ root of unity. The $(y=0)$-orbits have size $b_{\s}$ so if $e=b_{\s}$ then the $(y=0)$-orbits are not small orbits.
\end{proof}

These lemmas allow us to fully describe how many singularities $\ZZ_k$ has. The following proposition tells us that they are tame cyclic quotient singularities in the sense of Definition \ref{def::tcqs}. Theorem \ref{thm::resolvingtcqs} then allows us to resolve these singularities. 

\begin{prop}
    \label{prop::singPointsAretsgs} Let $z\in\ZZ_k$ be a singularity which is the image of a Galois orbit $Y\subseteq \Y_k$. Then $z$ is a tame cyclic quotient singularity. In addition, with notation as in Definition \ref{def::tcqs}, $\frac{m}{r} = \frac{e}{r}$ where $1 \leq r < e$ and $r\mod e$ is given in the following table:
    \begin{center}
        \begin{tabular}{|c|c|c|}
        \hline
             Orbit Type & $ r \mod e$ & Condition \\ \hline
             $\infty$ & $e\lambda_\s - e(g(C)+1)d_{\s}$ & $\s$ odd \\ \hline
             $\infty$ & $-ed_{\s}|Y|$ & $\s$ even \\ \hline
             $y=0$ & $e\lambda_\s|Y|$ & None \\ \hline
             $x=0$ & $ed_{\s}|Y|$ & None \\ \hline
             $(0,0)$ & $e\lambda_\s$ & None \\ \hline
        \end{tabular}
    \end{center}
\end{prop}
\begin{proof}
    Recall that for $z$ to be a tame cyclic quotient singularity, there must exist $m>1$ invertible in $k$, a unit $r \in (\Z/m\Z)^{\times}$ and integers $m_1>0$ and $m_2\geq 0$ such that $m_1 \equiv m_2\mod m$, and such that $\OO_{\ZZ,z}$ is equal to the subalgebra of $\mu_m$-invariants in $k\llbracket t_1,t_2 \rrbracket/(t_1^{m_1}t_2^{m_2} - \pi_K)$ under the action $t_1 \mapsto \zeta_m, t_2 \mapsto \zeta_m^r$. We will show that $m = \frac{e}{|Y|} = |\textrm{Stab}(Y)|,$  $m_1 = e$, $m_2 = 0$ and will explicitly calculate $r$.
    
    Let $Y \subseteq \Y_k$ be a small orbit and let $\pointInOrbit\in Y$. Then $\OO_{\ZZ,z}$ is the subalgebra of $\mu_m$-invariants of $\OO_{\Y,\pointInOrbit}$ under the action of $\textrm{Stab}(Y)$, where $m = |\textrm{Stab(Y)}|$. This follows from the definition of $\ZZ$ as the quotient of $\Y$ under the action of $\Gal(L/K)$, which for a generator  $\sigma\in\Gal(L/K)$ sends 
    $$\sigma: \pi_L \longmapsto \chi(\sigma) \pi_L,\quad
        \sigma: x \longmapsto \chi(\sigma)^{ed_{\s}} x,\quad
        \sigma: y \longmapsto \chi(\sigma)^{e\lambda_\s} y.$$
    To prove that $z$ is a tame cyclic quotient singularity we must calculate $\OO_{\Y,\pointInOrbit}$.
    
    First, suppose $Y$ is a $(y=0)$ or a $(0,0)$-orbit, and write $Q = (x_{\pointInOrbit},0)$. Then $\OO_{\Y,\pointInOrbit}$ is generated by $\pi_L, x-x_{\pointInOrbit}$ and $y$. However, since $x - x_{\pointInOrbit} = uy^2$ for a unit $u \in \OO_{\Y,\pointInOrbit}$, $\OO_{\Y,\pointInOrbit}$ is generated by $\pi_L$ and $y$. Therefore, $\OO_{\Y,\pointInOrbit}\cong k\llbracket \pi_L, y \rrbracket/(\pi_L^e - \pi_K)$, and $\OO_{\ZZ,z}$ is the subalgebra of $\mu_m$-invariants of this under the action $\pi_L \mapsto \zeta_m \pi_L, y \mapsto \zeta_m^{e\lambda_\s}y$ where $\zeta_m = \chi(\sigma)^{|Y|}$ generates $\textrm{Stab}(Y)$ (as $\Gal(L/K)$ is cyclic). Let $r$ be such that $0 < r < m$ and $r \equiv e\lambda_\s|Y|\mod m$. Then to prove $z$ is a tame cyclic quotient singularity all that is left to show is that $r$ is a unit in $(\Z/m\Z)^{\times}$ and that $e \equiv 0 \mod m$. The second is clear, and for the first note that since $\zeta_m^r$ also generates $\textrm{Stab}(Y)$, it must be a primitive $m^{\textrm{th}}$ root of unity hence $r$ must be a unit.
    
    If $Y$ is an $(x=0)$-orbit, then $\pointInOrbit = (0,\pm\sqrt(c_f))$. By a similar argument to above, $\OO_{\Y,\pointInOrbit} \cong k\llbracket \pi_L, x \rrbracket/(\pi_L^e - \pi_K)$ and $\OO_{\ZZ,z}$ is the subalgebra of $\mu_m$ invariants under the action $\pi_L \mapsto \zeta_m \pi_L, x \mapsto \zeta_m^r x$, where $m = \frac{e}{|Y|}$ and $r$ is such that $0 < r < m$ and $r \equiv ed_{\s}|Y| \mod m$.
    
    If $Y$ is an $\infty$ orbit, then we can calculate $m,r,m_1$ and $m_2$ explicitly by going to the chart at infinity.
\end{proof}


\begin{cor}
    If $Y$ is a $(y=0)$-orbit which gives rise to a tame cyclic quotient singularity $z\in\ZZ_k$, then the tame cyclic quotient invariants $(m,r)$ of $z$ are such that $\frac{m}{r} = 2$.
\end{cor}
\begin{proof}
    The orbit $Y$ is a $(y=0)$-orbit hence has size $b_{\s}$. Lemma \ref{lem::2bornot2b} tells us that, $|Y|<e$ if and only if $e = 2b_{\s}$. In this case $e\lambda_\s|Y| = 2b_{\s} \cdot \frac{\nu_\s}{2} \cdot b_{\s} = b_{\s}^2 \nu_\s$. Since $b_{\s} = \frac{e}{2}$ and $b_{\s}\nu_\s$ is an odd integer, this gives $e\lambda_\s|Y| \equiv \frac{e}{2} \mod e$, hence $\frac{m}{r}= 2$. 
\end{proof}


\subsection{Tails}
\label{subsec::tpgrtails}

Resolving singularities as in Section \ref{subsec::tpgrauto} results in tails. These are chains of rational curves intersecting the central component once and intersecting the rest of the special fibre nowhere else. It is useful to distinguish between tails based on the type of orbit they arise from. 
\begin{definition}
    \label{def::chainTypes} Define the following tails based on the type of singularity of $\ZZ_k$ they arise from:
    \begin{itemize}[leftmargin=*]
        \item \textit{$\infty$-tail}: arising from the blow up of a singularity of $\ZZ_k$ which arose from an $\infty$-orbit,
        \item \textit{$(y=0)$-tail}: arising from the blow up of a singularity of $\ZZ_k$ which arose from an orbit of non-zero roots,
        \item \textit{$(x=0)$-tail}: arising from the blow up of a singularity of $\ZZ_k$ which arose from an orbit on the points $(0,\pm \sqrt{c_f})$,
        \item \textit{$(0,0)$-tail}: arising from the blow up of a singularity of $\ZZ_k$ which arose from the point $(0,0)$.
    \end{itemize}
\end{definition}

\begin{remark}
    The tails defined in Definition \ref{def::chainTypes} are the only tails that can possibly occur in $\X_k$. This is because any tail must arise from a singularity of $\ZZ_k$ which lies on just one component, namely a singularity which arises from one of the small orbits discussed in Section \ref{subsec::tpgrauto}. 
\end{remark}


\begin{proof}[Proof of Theorem \ref{thm::tpgr}]
    The central component $\Gamma$ is the image of the unique component of $\Y_k$ under $q$. Since blowing up points on $\Gamma$ does not affect its multiplicity, this has multiplicity $e$, by Proposition \ref{prop::componentsmult}. The description of the tails follows from Lemmas \ref{lem::inftyorbits}, \ref{lem::x0orbits}, and \ref{lem::y0orbits}, since the tails are in a bijective correspondence with the orbits of points of $\Y_k$ of size strictly less than $e$. We must check that $\Gamma$ really appears in the minimal SNC model. Suppose $\Gamma$ is exceptional. Then $g(\Gamma) = 0$ and Riemann-Hurwitz says
    $$\sum_{z\in\Z}\left( \frac{e}{|q^{-1}(z)|} - 1 \right) \geq e.$$
    Therefore there must be at least three ramification points, so $\Gamma$ intersects at least three tails.
\end{proof}

\begin{remark}
    The method for calculating the multiplicities of the rational curves in these tails is described in Theorem \ref{thm::resolvingtcqs} using the tame cyclic quotient invariants given in Proposition \ref{prop::singPointsAretsgs}.
\end{remark}

\begin{remark}
    The central component $\Gamma$ is the only component of $\X_k$ which may have non-zero genus. Its genus, $g(\Gamma)$, can be calculated via the Riemann-Hurwitz formula. 
    An even more explicit calculation of $g(\Gamma)$ in terms of the Newton polytope is given in Proposition \ref{prop::genusformula}.
\end{remark}

\subsection{Relation to Newton polytopes}
\label{subsec::relationtonp}

Up to this point, this section has described the minimal SNC model of a hyperelliptic curve $C/K$ with tame potentially good reduction using the methods from Section \ref{subsec::QuotientsofModels}. However, such a hyperelliptic curve has a nested cluster picture so we can also calculate the minimal SNC model using Newton polytopes and the techniques described in Section \ref{sec::modsusingnewtpolys}. By the uniqueness of the minimal SNC model, these two methods will give the same result: for the reader's sanity, in this section we will show that this is indeed the case. Recall that without loss of generality we can assume that $C/K$ with tame potentially good reduction is given by one of the following two equations: 
\begin{align*}
    y^2 &= c_f\prod_{0\neq r \in \Rcal}(x - u_r \pi_K^{d_{\s}}), & \textrm{ if } b_\s\mid |\s|,\\
    y^2 &= c_fx\prod_{0\neq r \in \Rcal}(x - u_r \pi_K^{d_{\s}}),  & \textrm{ if } b_\s\nmid |\s|.
\end{align*}
The Newton polytope of $C$ is shown in Figure \ref{fig::tpgrnp1} if $b_{\s} \mid |\s|$, and in Figure \ref{fig::tpgrnp2} if $b_{\s} \nmid |\s|$. In each case there is exactly one $v$-face of $\Delta_v(C)$, which we shall label $F$. Therefore, by Theorem \ref{thm::timsmainthm}, the minimal SNC model consists of a central component $\Gamma_{\s} = \Gamma_F$, and possibly tails arising from the three outer $v$-edges of $F$.
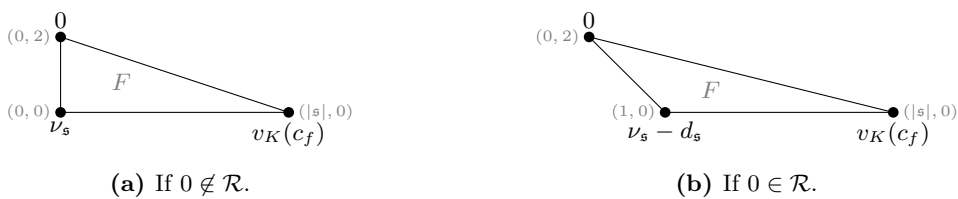
\begin{figure}[h]
\centering
\begin{subfigure}{.5\textwidth}
  \centering
  \begin{tikzpicture}
        \draw (0,0) -- (3,0);
        \draw (0,0) -- (0,1);
        \draw (0,1) -- (3,0);
        \fill (0,0) circle (2pt) node[below, font=\small]{$\nu_\s$} node[left, color=gray,font=\tiny]{$(0,0)$};
        \fill (3,0) circle (2pt) node[below, font=\small]{$v_K(c_{f})$} node[right, color=gray,font=\tiny]{$(|\s|,0)$};
        \fill (0,1) circle (2pt) node[above, font=\small]{$0$} node[left, color=gray,font=\tiny]{$(0,2)$};
        \node[font=\small, color=gray] at (0.8,0.4) {$F$};
    \end{tikzpicture}
    \caption{If $0 \not \in \mathcal{R}$.}
    \label{fig::tpgrnp1}
\end{subfigure}%
\begin{subfigure}{.5\textwidth}
  \centering
  \begin{tikzpicture}
    \draw (1,0) -- (4,0);
        \draw (1,0) -- (0,1);
        \draw (0,1) -- (4,0);
        \fill (1,0) circle (2pt) node[below, font=\small]{$\nu_\s - d_{\s}$} node[left, color=gray,font=\tiny]{$(1,0)$};
        \fill (4,0) circle (2pt) node[below, font=\small]{$v_K(c_{f})$} node[right, color=gray,font=\tiny]{$(|\s|,0)$};
        \fill (0,1) circle (2pt) node[above, font=\small]{$0$} node[left, color=gray,font=\tiny]{$(0,2)$};
        \node[font=\small, color=gray] at (1.6,0.3) {$F$};
    \end{tikzpicture}
    \caption{If $0 \in \mathcal{R}$.}
    \label{fig::tpgrnp2}
\end{subfigure}
\caption{
$\Delta_v(C)$ of a hyperelliptic curve $C$ with tame potential good reduction.}
\label{fig::newtonpolytpgr}
\end{figure}

\begin{lemma}
    The multiplicity of $\Gamma_{\s} = \Gamma_F$ is $\delta_F$; that is $\delta_F = e$.
\end{lemma}
\begin{proof}
    We will first show that $e \mid \delta_F$, and then that $\delta_F \mid e$. Note that, in both Newton polytopes in Figure \ref{fig::newtonpolytpgr}, the valuation map is given by the affine function $v_{\Delta}(x,y)$ = $\nu_\s - d_{\s}x - \frac{\nu_\s}{2}y$. Since $e$ is such that $ed_{\s} \in \Z$ and $e\nu_\s\in 2\Z$, we have $e v_{\Delta}(x,y) = e\nu_\s - ed_{\s}x - e\frac{\nu_\s}{2}y \in \Z$. As $\delta_F$ is the common denominator of all $v_{\Delta}(x,y)$ for $x,y \in \Delta$, this gives that $\delta_F \mid e$.
    
    Note that $\delta_F\left(v_{\Delta}(n-1,0) - v_{\Delta}(n,0)\right) = \delta_Fd_{\s} \in \Z$, and $\delta_F\left(v_{\Delta}(1,0) - v_{\Delta}(1,1)\right) = \delta_F \frac{\nu_\s}{2} \in \Z$. By minimality of $e$, this implies $e \mid \delta_F$.
\end{proof}

\begin{lemma}
\label{lem::inftytent}
    The $\infty$-tails arise from the outer $v$-edge of $\Delta_v(C)$ between $(0,2)$ and $(|\s|,0)$.
\end{lemma}
\begin{proof}
    We will first check that this $v$-edge gives the correct number of $\infty$-tails, and then calculate the slope to check that the multiplicities of the components are the same.
    
    Let us call this $v$-edge $L$. By Theorem \ref{thm::timsmainthm} then $L$ contributes $|\overline{L}(\Z)_{\Z}| - 1$ tails to the SNC model. Since the points $(0,2),(|\s|,0) \in L(\Z)_{\Z}$, it contributes two tail if and only if $P = (\frac{|\s|}{2},1) \in L(\Z)_{\Z}$. If $\s$ is odd then $P \not \in L \cap \Z^2$, hence $L$ contributes one tail. If $\s$ is even then $v_{\Delta}(P) = \frac{v_K(c_f)}{2}$, hence $P \in L(\Z)_{\Z}$ if and only if $v_K(c_f) \in 2\Z$. Therefore $L$ contributes one tail if $\s$ is even and $v_K(c_f)$ is odd, and two tails if $\s$ and $v_K(c_f)$ are even. This agrees with Theorem \ref{thm::tpgr}.
    
    A quick calculation tells us that $\delta_L = 2$ if and only if $\s$ is even and $v_K(c_f) \not \in 2\Z$, and that $\delta_L = 1$ otherwise. Therefore, $\delta_L = |Y|$, where $Y$ is the orbit at infinity. The unique surjective affine function which is zero on $L$ and non-negative on $F$ is $L^*_F(x,y) = 2|\s| - 2x - |\s|y$ if $\s$ is odd, and $L^*_F(x,y) = |\s| - x - \frac{1}{2}|\s|y$ if $\s$ is even. Therefore, $s_1^L = (g + 1)d_{\s} -\lambda_{\s}$ if $\s$ is odd, and $s_1^L=-d_{\s}|Y|$ if $\s$ is even. Since the multiplicities of the components of a tail are the Hirzebruch-Jung approximants of the slopes, we are done after comparing the slopes to the table in Proposition \ref{prop::tpgrSings}.
    
    If $e = 2$ (when $\s$ is even and $v_K(c_f)$ is odd) has $s^1_L \in \Z$, so the associated tail is empty, which agrees with the table in Theorem \ref{thm::tpgr}.
\end{proof}

\begin{lemma}
    In both cases, when $0\in\Rcal$ and when $0\notin\Rcal$, the $(y=0)$-tails arise from the outer $v$-edge of $\Delta_v(C)$ on the $x$-axis. Furthermore, if $b_{\s} \mid |\s|$ then the $(x=0)$-tails arise from the $v$-edge between $(0,0)$ and $(0,2)$. Else the $(0,0)$-tail arises from the $v$-edge between $(1,0)$ and $(0,2)$. 
    \label{lem::yiszero}
\end{lemma}
\begin{proof}
    This follows after a similar calcuation to Lemma \ref{lem::inftytent}.
    %
    %
\end{proof}


\subsection{The Curve $C_{\tilde{\s}}$}
\label{subsec::cso}

To conclude this section, we drop the requirement for $C/K$ to have tame potentially good reduction. We will describe a hyperelliptic curve with potentially good reduction which we associate to a principal cluster $\s\in\Sigma_C$ with $\ssgs > 0$. This new curve, which we will denote by $C_{\tilde{\s}}$, will be invaluable in describing the components of the minimal SNC model of $C/K$ which are associated to $\s\in\Sigma_{C}$. For $\s\in\Sigma_{C/K}$ with $\ssgs > 0$, the cluster picture $\Sigma_{\tilde{\s}}$ of $C_{\tilde{\s}}/K$ will be such that the singletons in $\Sigma_{\tilde{\s}}$ correspond to odd children of $\s$ and the even children of $\s$ are in effect discarded. The leading coefficient of $C_{\tilde{\s}}/K$ is chosen so that everything behaves well, and allows us to make the comparisons we wish between the minimal SNC model of $C/K$ and the minimal SNC model of $C_{\tilde{\s}}/K$. We describe this formally now.

\begin{definition}
    \label{def::cso}
    Let $C/K$ be a hyperelliptic curve, not necessarily with tame potentially good reduction. Let $\s\in\Sigma_{C/K}$ be a principal cluster with $\ssgs > 0$ such that $\s$ is fixed by $\GK$. Suppose furthermore that $\sigma(z_{\s'}) = z_{\sigma(\s')}$ for any $\sigma \in \GK$, $\s'\in\Sigma_{C/K}$. We define another hyperelliptic curve $C_{\tilde{\s}}/K$ by
    \begin{align*}
        C_{\tilde{\s}}:y^2 = c_{f_{\s}} \prod_{\mathfrak{o}\in\tilde{\s}} (x-z_{\mathfrak{o}}),\textrm{ where }
        c_{f_{\s}} = c_f \prod_{r \not\in \s} (z_{\s} - r).
    \end{align*}
    Write $\Sigma_{\tilde{\s}/K}=\Sigma_{\tilde{\s}}=\Sigma(C_{\tilde{\s}}/K)$ for the cluster picture of $C_{\tilde{\s}}/K$, and $\X_{\tilde{\s}}$ for the minimal SNC model of $C_{\tilde{\s}}/K$. The special fibre of the minimal SNC model of $C_{\tilde{\s}}$ is denoted $\X_{\tilde{\s},k}$, and the central component is denoted $\Gamma_{\tilde{\s}}$. We also write $\Rcal_{\tilde{\s}}$ for the set of all roots of $c_{f_{\s}} \prod_{\mathfrak{o}\in\tilde{\s}} (x-z_{\mathfrak{o}})$, and define $d_{\tilde{\s}}=d_{\Rcal_{\tilde{\s}}}$, $\nu_{\tilde{\s}}=\nu_{\Rcal_{\tilde{\s}}}$, and $\lambda_{\tilde{\s}}=\lambda_{\Rcal_{\tilde{\s}}}$.
\end{definition}


\begin{remark}
\label{rem::Cstilde}
    Let $\Y$ be the minimal semistable model of $C$ over $\OO_L$, for some $L/K$ such that $C/L$ is semistable. Let $\s$ be a principal cluster with $\ssgs > 0$. If we reduce $C_{\tilde{\s}} \mod \mathfrak{m}$, we obtain $\Gamma_{\s,L}$, the component of $\Y_k$ corresponding to $\s$ (see \cite[Theorem~8.5]{DDMM18} for the equation of $\Gamma_{\s,L}$). In addition, $c_{f_{\s}}$ has been carefully chosen so that $d_{\s}=d_{\tilde{s}}$, $\nu_{\s} = \nu_{\tilde{\s}}$ and $\lambda_{\s} = \lambda_{\tilde{\s}}$. In particular, the automorphisms induced by Galois on $\Gamma_{\s,L}$ and $\Gamma_{\tilde{s},L}$ are the same.
\end{remark}

\begin{definition}
    \label{def::esgs}
    For a principal, Galois-invariant cluster $\s$, define $e_{\s}$ to be the minimum integer such that $e_{\s}d_{\s} \in \Z$ and $e_{\s}\nu_{\s} \in 2\Z$. Furthermore, if $\ssgs > 0$ define $g(\s)$ to be the \textit{genus of $\Gamma_{\tilde{\s}}$} and if $\ssgs = 0$ define $g(\s) = 0$. We call $g(\s)$ the \emph{genus of $\s$}.
\end{definition}


\begin{remark}
    By the Semistability Criterion \cite[Theorem~1.8]{DDMM18}, if $\s$ is not \"ubereven then $e_{\s}$ is the minimum integer such that $C_{\tilde{\s}}$ has semistable reduction over a field extension $L/K$ of degree $e_{\s}$. In particular, the central component $\Gamma_{\tilde{\s}}$ of $\X_{\tilde{\s},k}$ has multiplicity $e_{\s}$ and genus $g(\s)$. If $e_{\s} = 1$ then $\ssgs = g(\s)$, but the converse is not necessarily true.

\end{remark}

\begin{prop}
    \label{prop::genusformula}
    If $\ssgs > 0$, the genus $g(\s)$ is given by 
    
    \begin{align*} g(\s) = \begin{cases} \lfloor\frac{\ssgs}{b_{\s}}\rfloor & \lambda_{\s}\in \Z,\\ 
    \lfloor \frac{\ssgs}{b_{\s}}+\frac{1}{2}\rfloor & \lambda_{\s}\not\in \Z, b_{\s}\textrm{ even},\\ 
    0 & \lambda_{\s} \not \in \Z, b_{\s} \textrm{ odd}.
    \end{cases} \end{align*}
\end{prop}
\begin{proof}
    By Theorem \ref{thm::timsmainthm}, we know $g(\s)$ is given by $|F(\Z)_{\Z}|$. This is the number of interior points with integer valuation of the unique face $F$ of the Newton polytope of $C_{\widetilde{\s}}$. By examining Figure \ref{fig::newtonpolytpgr}, we see that all interior points are of the form $(x,1)$ with $1 \leq x \leq \ssgs$. For such points, $v_{\Delta}(x,1) = \lambda_{\s} - d_{\s}x$. Therefore, $$ g(\s) = \left|\{x : 1 \leq x \leq \ssgs, \lambda_{\s} - xd_{\s} \in \Z\}\right|.$$
    When $\lambda_{\s} \in \Z$ this is therefore equal to 
    $$\left|\{x : 1 \leq x \leq \ssgs, b_{\s} \mid x \}\right|=\left\lfloor \frac{\ssgs}{b_{\s}}\right\rfloor.$$ 
    When $\lambda_{\s}\not\in\Z$, this is equal to 
    $$\left| \left\{ x : 1 \leq x \leq \ssgs, xd_{\s}\in\frac{1}{2}\Z\setminus\Z\right\} \right|.$$
    When $\lambda_{\s}\not\in\Z$ and $b_{\s}$ is odd this set is always empty, and when $\lambda_{\s}\not\in\Z$ and $b_{\s}$ is even it has size $\left\lfloor \frac{\ssgs}{b_{\s}} + \frac{1}{2} \right\rfloor$.
\end{proof}


\begin{lemma}
    \label{lem::degreeofaction}
    Let $C$ be a hyperelliptic curve and let $\s \in \Sigma_{C/K}$ be a principal cluster which is fixed by Galois. Let $L$ be an extension such that $C$ is semistable over $L$, and let $\sigma$ generate $\Gal(L/K)$. Then $\sigma\vert_{\Gamma_{\s,L}} : \Gamma_{\s,L} \rightarrow \Gamma_{\s,L}$ has degree $e_{\s}$.
\end{lemma}
\begin{proof}
    The map $\sigma\vert_{\Gamma_{\s,L}}$ is given by $(x,y) \mapsto (\chi(\sigma)^{ed_{\s}} x, \chi(\sigma)^{e\lambda_{\s}} y)$. The result follows as $e_{\s}$, by definition, is the minimal integer such that $e_{\s}d_{\s},e_{\s}\lambda_{\s} \in \Z$.
\end{proof}

\section{Calculating Linking Chains}
\label{sec::upc}

The minimal SNC model of a general hyperelliptic curve $C/K$ can roughly be described as follows. Each principal cluster of $\Sigma_C$ has one or two central components, and some tails associated to it. These central components are linked by chains of rational curves. Section \ref{sec::tpgr} will allow us to describe these central components and tails, while this section will be used to describe these linking chains. This includes describing any loops. We will also see the simplest example of the general philosophy that the components of the special fibre of the minimal SNC model of $C/K$ associated to a principal cluster $\s$ ``look like'' the special fibre of the minimal SNC model of $C_{\tilde{\s}}/K$.

Throughout the rest of this section we will take $C/K$ to be a hyperelliptic curve such that such that $\Sigma_{C/K}$ consists of exactly two proper clusters: a proper cluster $\s$ and a unique proper child $\s'<\s$. This is illustrated in Figure \ref{fig:uniqpropchild}. Note that $d_{\s'} > d_{\s}$ and $|\s| > |\s'|$. If $C$ is such that $\s$ is even and $|\s| = |\s'| + 1$ then $C/K$ has potentially good reduction, this case is covered in Section \ref{sec::tpgr}. To avoid this case we will assume that if $\s$ is even then $|\s| \geq |\s'| +  2$. Since hyperelliptic curves of this type are nested we can directly apply the methods from \cite{Dok18}. Before we apply Theorem \ref{thm::timsmainthm}, we need to understand the Newton polytope of $C/K$.

\begin{figure}[h]
	\centering
\begin{tikzpicture}
\fill (0,0) circle (1.5pt);
\path (0,0) -- node[auto=false]{\ldots} (1,0);
\fill (1,0) circle (1.5pt) ;

\fill (1.75,0) circle (1.5pt);
\path (1.75,0) -- node[auto=false]{\ldots} (2.75,0);
\fill (2.75,0) circle (1.5pt);

\draw (0.5,0) ellipse (0.8cm and 0.3cm) node[below,  xshift = 0.8cm]{$\s'$} node[above, yshift = 0.1cm, xshift = 0.8cm]{$d_{\s'}$};

\draw (1.3,0) ellipse (2cm and 0.7cm) node[below, yshift = -0.4cm, xshift = 2cm]{$\Rcal=\s$} node[above, yshift = 0.3cm, xshift = 1.8cm]{$d_{\s}$};

\end{tikzpicture}
	\caption{Cluster picture with a parent $\s$ and a unique proper child $\s'$ with no proper children of its own.}
	\label{fig:uniqpropchild}
\end{figure}
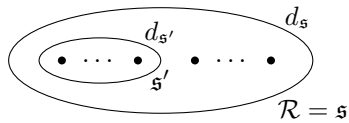

\subsection{The Newton polytope}
\label{subsec::upcnetwonpoly}

Without loss of generality, we can assume that the defining equation of $C/K$ will be either
\begin{align}
\label{eq::uniqupropchildcase1}
    y^2 = c_f \prod_{r \in \Rcal \setminus \s'} \left(x - u_r \pi_K^{d_{\s}}\right) \prod_{r \in \s'} \left(x - u_r \pi_K^{d_{\s'}}\right),
\end{align}
or 
\begin{align}
\label{eq::uniqupropchildcase2}
y^2 = c_fx\prod_{r \in \Rcal \setminus \s'} \left(x - u_r \pi_K^{d_{\s}}\right) \prod_{r \in \s'} \left(x - u_r \pi_K^{d_{\s'}}\right).
\end{align}
where the $u_r$ are units. If $C$ has defining equation (\ref{eq::uniqupropchildcase1}), then $\nu_{\s'}=v_K(c_f)+(|\s|-|\s'|)d_{\s}+|\s'|d_{\s'}$, and the Newton polytope $\Delta_v(C)$ of $C$ will be as shown in Figure \ref{fig::newtonpolyuniqpropchild1}.  If instead $C$ has defining equation (\ref{eq::uniqupropchildcase2}), the Newton polytope will be as shown in Figure \ref{fig::newtonpolyuniqpropchild2}. 
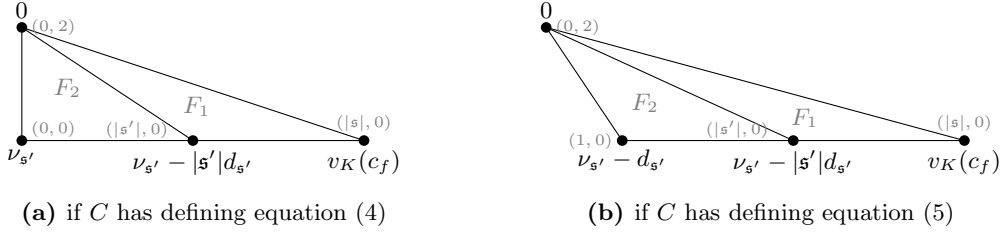
\begin{figure}[h]
\begin{adjustwidth}{-0.5cm}{}
\centering
\begin{subfigure}{0.5\textwidth}
  \centering
  \begin{tikzpicture}
        \draw (0,0) -- (4.5,0);
        \draw (0,0) -- (0,1.5);
        \draw (0,1.5) -- (4.5,0);
        \draw (0,1.5) -- (2.25,0);
        \fill (0,0) circle (2pt) node[below, font=\small]{$\nu_{\s'}$} node[right,yshift=0.15cm, color=gray,font=\tiny]{$(0,0)$};
        \fill (2.25,0) circle (2pt) node[below, font=\small]{$\nu_{\s'}-|\s'|d_{\s'}$} node[left, color=gray,font=\tiny, xshift=-0.2cm, yshift=0.15cm]{$(|\s'|,0)$};
        \fill (4.5,0) circle (2pt) node[below, font=\small]{$v_K(c_f)$} node[above, color=gray,font=\tiny]{$(|\s|,0)$};
        \fill (0,1.5) circle (2pt) node[above, font=\small]{$0$} node[right, color=gray,font=\tiny]{$(0,2)$};
        \node[font=\small, color=gray] at (0.6,0.7) {$F_2$};
        \node[font=\small, color=gray] at (2.3,0.45) {$F_1$};
    \end{tikzpicture}
  \caption{if $C$ has defining equation (\ref{eq::uniqupropchildcase1})}
  \label{fig::newtonpolyuniqpropchild1}
\end{subfigure}%
\begin{subfigure}{.5\textwidth}
  \centering
  \begin{tikzpicture}
        \draw (1,0) -- (5.5,0);
        \draw (1,0) -- (0,1.5);
        \draw (0,1.5) -- (5.5,0);
        \draw (0,1.5) -- (3.25,0);
        \fill (1,0) circle (2pt) node[below, font=\small]{$\nu_{\s'}-d_{\s'}$} node[left, color=gray,font=\tiny]{$(1,0)$};
        \fill (3.25,0) circle (2pt) node[below, font=\small]{$\nu_{\s'}-|\s'|d_{\s'}$} node[left, color=gray,font=\tiny, xshift=-0.2cm, yshift=0.15cm]{$(|\s'|,0)$};
        \fill (5.5,0) circle (2pt) node[below, font=\small]{$v_K(c_f)$} node[above, color=gray,font=\tiny]{$(|\s|,0)$};
        \fill (0,1.5) circle (2pt) node[above, font=\small]{$0$} node[right, color=gray,font=\tiny]{$(0,2)$};
        \node[font=\small, color=gray] at (1.3,0.5) {$F_2$};
        \node[font=\small, color=gray] at (3.4,0.3) {$F_1$};
    \end{tikzpicture}
  \caption{if $C$ has defining equation (\ref{eq::uniqupropchildcase2})}
  \label{fig::newtonpolyuniqpropchild2}
\end{subfigure}
\caption{Newton polytope $\Delta_v(C)$ of $C$.}
\label{fig::newtonpolyuniqpropchild}
\end{adjustwidth}
\end{figure}

\vspace{-10px}
\begin{lemma}
\label{lem::corofnewpolys1}
    Let $C$ have  Newton polytope as in Figure \ref{fig::newtonpolyuniqpropchild1}. Then there is an isomorphism $\psi:\overline{F_1}\to \Delta_v(C_{\tilde{\s}}),$ from the closure of the $v$-face marked $F_1$ to the Newton polytope of $C_{\tilde{\s}}$ (whose only $v$-face we label $F_{\tilde{\s}}$), shown in Figure \ref{fig::newtonpolyuniqpropchild3}. In particular $\psi$ preserves valuations and $\delta_{F_1}=\delta_{F_{\tilde{\s}}}$. In this sense we say that $F_1$ \emph{corresponds} to the cluster $\s$. Similarly the $v$-face $F_2$ in Figure \ref{fig::newtonpolyuniqpropchild1} corresponds to $\s'$.
\end{lemma}
\begin{proof}[Proof of Lemma \ref{lem::corofnewpolys1}]
Let us compare the $v$-face $F_1$ in Figure \ref{fig::newtonpolyuniqpropchild1} to the Newton polytope, $\Delta_v(C_{\tilde{\s}})$, of $C_{\tilde{\s}}$. This is given in Figure \ref{fig::newtonpolyuniqpropchild3a} if $\s'$ is even, and given in Figure \ref{fig::newtonpolyuniqpropchild3b} if $\s'$ is odd.
    
If $\s'$ is even we can define $$\psi:\overline{F_1}\to \Delta_v(C_{\tilde{\s}}):(x,y)\mapsto\left(x-\frac{|\s'|}{2}(2-y),y\right).$$ It is easy to show that this is an isomorphism, and that the valuations are preserved. Similarly if $\s'$ is odd we can define $$\psi:\overline{F_1}\to \Delta_v(C_{\tilde{\s}}):(x,y)\mapsto\left(x-\frac{(|\s'|+1)}{2}(2-y),y\right),$$ which is also an isomorphism that preserves the valuations. In particular, in both cases we have $\delta_{F_1} = \delta_{F_{\tilde{\s}}}$, and if $v_1$ is the unique affine function agreeing with $v_{\Delta(C)}$ on $F_1$, then $v_1(x,y)=v_{\Delta_{\tilde{\s}}}(\psi(x,y))$, where $v_{\Delta_{\tilde{\s}}}=v_{\Delta(C_{\tilde{\s}})}$.

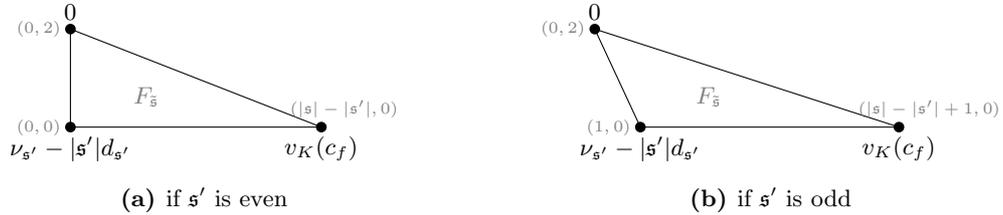
\begin{figure}[ht]
\begin{adjustwidth}{-0.5cm}{}
\centering
\begin{subfigure}{.5\textwidth}
  \centering
  \begin{tikzpicture}
        \draw (0,0) -- (3.3,0);
        \draw (0,0) -- (0,1.3);
        \draw (0,1.3) -- (3.3,0);
        \fill (0,0) circle (2pt) node[below, font=\small]{$\nu_{\s'}-|\s'|d_{\s'}$} node[left, color=gray,font=\tiny]{$(0,0)$};
        \fill (3.3,0) circle (2pt) node[below, font=\small]{$v_K(c_{f})$} node[above, color=gray,font=\tiny,xshift=0.3cm]{$(|\s|-|\s'|,0)$};
        \fill (0,1.3) circle (2pt) node[above, font=\small]{$0$} node[left, color=gray,font=\tiny]{$(0,2)$};
        \node[font=\small, color=gray] at (1,0.4) {$F_{\tilde{\s}}$};
    \end{tikzpicture}
  \caption{if $\s'$ is even}
  \label{fig::newtonpolyuniqpropchild3a}
\end{subfigure}%
\begin{subfigure}{.5\textwidth}
  \centering
  \begin{tikzpicture}
        \draw (0.6,0) -- (4,0);
        \draw (0.6,0) -- (0,1.3);
        \draw (0,1.3) -- (4,0);
        \fill (0.6,0) circle (2pt) node[below, font=\small]{$\nu_{\s'}-|\s'|d_{\s'}$} node[left, color=gray,font=\tiny]{$(1,0)$};
        \fill (4,0) circle (2pt) node[below, font=\small]{$v_K(c_{f})$} node[above, color=gray,font=\tiny,xshift=0.4cm]{$(|\s|-|\s'|+1,0)$};
        \fill (0,1.3) circle (2pt) node[above, font=\small]{$0$} node[left, color=gray,font=\tiny]{$(0,2)$};
        \node[font=\small, color=gray] at (1.5,0.4) {$F_{\tilde{\s}}$};
    \end{tikzpicture}
  \caption{if $\s'$ is odd}
  \label{fig::newtonpolyuniqpropchild3b}
\end{subfigure}
\caption{Newton polytope $\Delta_v(C_{\tilde{\s}})$ of $C_{\tilde{\s}}$,where $C$ is given by either defining equation (\ref{eq::uniqupropchildcase1}), or (\ref{eq::uniqupropchildcase2}).}
\label{fig::newtonpolyuniqpropchild3}
\end{adjustwidth}
\end{figure}

Similarly, we can see that the $v$-face $F_2$ in Figure \ref{fig::newtonpolyuniqpropchild1} corresponds to $\s'$ by considering the Newton polytope $\Delta_v(C_{\tilde{\s'}})$ of $C_{\tilde{\s'}}$. This is shown in Figure \ref{fig::newtonpolyuniqpropchild5}. We see that the map $$\overline{F_2}\to\Delta_v(C_{\tilde{\s'}}):(x,y)\mapsto(x,y)$$ is an isomorphism that preserves the valuations, that is $v_2(x,y) = v_{\Delta(C_{\tilde{\s'}})}(x,y)$, and $\delta_{F_2} = \delta_{F_{\tilde{\s'}}}$, where $v_2$ is the unique affine function agreeing with $v_{\Delta(C)}$ on $F_2$.

\begin{figure}[h]
    \centering
    \begin{tikzpicture}
        \draw (0,0) -- (5,0);
        \draw (0,0) -- (0,1.2);
        \draw (0,1.2) -- (5,0);
        \fill (0,0) circle (2pt) node[below, font=\small]{$\nu_{\s'}=v_K(c_f)+(|\s|-|\s'|)d_{\s}+|\s'|d_{\s'}$} node[left, color=gray,font=\tiny]{$(0,0)$};
        \fill (5,0) circle (2pt) node[below, font=\small]{$\nu_{\s'}-|\s'|d_{\s'}$} node[right, color=gray,font=\tiny]{$(|\s'|,0)$};
        \fill (0,1.2) circle (2pt) node[above, font=\small]{$0$} node[left, color=gray,font=\tiny]{$(0,2)$};
        \node[font=\small, color=gray] at (1.3,0.5) {$F_{\tilde{\s'}}$};
    \end{tikzpicture}
    \caption{Newton polytope $\Delta_v(C_{\tilde{\s'}})$ of $C_{\tilde{\s'}}$, where $C$ has defining equation (\ref{eq::uniqupropchildcase1}).}
    \label{fig::newtonpolyuniqpropchild5}
\end{figure}
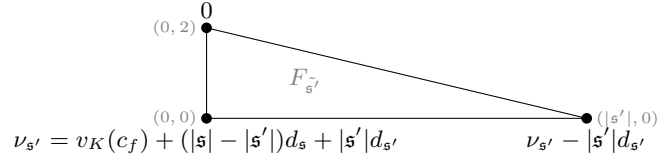
\end{proof}

We can make a similar comparison of the $v$-faces of the Newton polytope in Figure \ref{fig::newtonpolyuniqpropchild2}. 

\begin{lemma}
\label{lem:corofnetwonpolys2}
    Let $C$ have Newton polytope as in Figure \ref{fig::newtonpolyuniqpropchild2}. Then the $v$-face marked $F_1$ in Figure \ref{fig::newtonpolyuniqpropchild2} corresponds to the cluster $\s$. That is there is a valuation preserving isomorphism between $\overline{F_1}$ and $\Delta_v(C_{\tilde{\s}})$, and $\delta_{F_1}=\delta_{F_{\tilde{\s}}}$, where $F_{\tilde{\s}}$ is the unique $v$-face of $\Delta_v(C_{\tilde{\s}})$. Similarly the $v$-face marked $F_2$ on the Newton polytope in Figure \ref{fig::newtonpolyuniqpropchild2} corresponds to the cluster $\s'$.
\end{lemma}
\begin{proof}
    Follows by a similar argument to Lemma \ref{lem::corofnewpolys1}.
\end{proof}
    \vspace{-10px}
    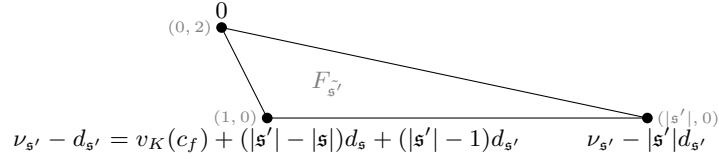
\begin{figure}[ht]
    \centering
    \begin{tikzpicture}
        \draw (1,0) -- (6,0);
        \draw (1,0) -- (0.4,1.2);
        \draw (0.4,1.2) -- (6,0);
        \fill (1,0) circle (2pt) node[below, font=\small]{$\nu_{\s'}-d_{\s'}=v_K(c_f)+(|\s'|-|\s|)d_{\s}+(|\s'|-1)d_{\s'}$} node[left, color=gray,font=\tiny]{$(1,0)$};
        \fill (6,0) circle (2pt) node[below, font=\small]{$\nu_{\s'}-|\s'|d_{\s'}$} node[right, color=gray,font=\tiny]{$(|\s'|,0)$};
        \fill (0.4,1.2) circle (2pt) node[above, font=\small]{$0$} node[left, color=gray,font=\tiny]{$(0,2)$};
        \node[font=\small, color=gray] at (1.8,0.45) {$F_{\tilde{\s'}}$};
    \end{tikzpicture}
    \caption{Newton polytope $\Delta_v(C_{\tilde{\s'}})$ of $C_{\tilde{\s'}}$, where $C$ has defining equation (\ref{eq::uniqupropchildcase2}).}
    \label{fig::newtonpolyuniqpropchild6}
\end{figure}

\subsection{Structure of the SNC Model}
\label{subsec::upcSNC}

The following theorem describes the structure of the special fibre of the minimal SNC model for hyperelliptic curves whose cluster picture looks like Figure \ref{fig:uniqpropchild}.

\begin{thm}
\label{thm:uniquepropchild}
    Let $C/K$ be a hyperelliptic curve with cluster picture as in Figure \ref{fig:uniqpropchild}. If $\s$ is principal, then the special fibre of the minimal SNC model has a component $\Gamma_{\s,K}$ arising from $\s$ with multiplicity $e_{\s}$ and genus $g(\s)$. If $\s'$ is principal then there is a component $\Gamma_{\s',K}$ arising from $\s'$ of multiplicity $e_{\s'}$ and genus $g(\s')$. These are linked by sloped chain(s) of rational curves with parameters $(t_2,t_2 + \delta, \mu)$, which are described in the following table:
    \begin{center}
        \begin{tabular}{|c|c|c|c|c|c|c|}
        \hline
             Name & From & To & $t_2$ & $\delta$ & $\mu$ & Conditions\\ \hline
             $L_{\s,\s'}$ & $\Gamma_{\s}$ & $\Gamma_{\s'}$ & $-\lambda_{\s}$ & $\delta_{\s'}/2$ & $1$ & $\s$ principal, $\s'$ odd, principal \\ \hline
             $L_{\s,\s'}^+$ & $\Gamma_{\s}$ & $\Gamma_{\s'}$ & $-d_{\s}$ & $\delta_{\s'}$ & $1$ & $\s$ principal, $\s'$ even, principal, $\epsilon_{\s'} = 1$ \\ \hline
             $L_{\s,\s'}^-$ & $\Gamma_{\s}$ & $\Gamma_{\s'}$ & $-d_{\s}$ & $\delta_{\s'}$ & $1$ & $\s$ principal $\s'$ even, principal, $\epsilon_{\s'} = 1$ \\ \hline
             $L_{\s,\s'}$ & $\Gamma_{\s}$ & $\Gamma_{\s'}$ & $-d_{\s}$ & $\delta_{\s'}$ & $2$ & $\s$ principal, $\s'$ even, principal, $\epsilon_{\s'} = -1$ \\ \hline
             $L_{\s'}$ & $\Gamma_{\s}$ & $\Gamma_{\s}$ &$-d_{\s}$ & $2\delta_{\s'}$ & $1$ & $\s$ principal, $\s'$ twin, $\epsilon_{\s'} = 1$ \\ \hline
             $T_{\s'}$ & $\Gamma_{\s}$ & - & $-d_{\s}$ & $\delta_{\s'} + \frac{1}{2}$ & $2$ & $\s$ principal, $\s'$ twin, $\epsilon_{\s'} = -1$ \\ \hline
             $L_{\s'}$ & $\Gamma_{\s'}$ & $\Gamma_{\s'}$ & $-d_{\s}$ & $2\delta_{\s'}$ & $1$ & $\s$ cotwin, $v_K(c_f) \in 2\Z$ \\ \hline
             $T_{\s'}$ & $\Gamma_{\s'}$ & - & $-d_{\s}$ & $\delta_{\s'} + \frac{1}{2}$ & $2$ & $\s$ cotwin, $v_K(c_f) \not\in 2\Z$ \\ \hline
        \end{tabular}
    \end{center} 
    The chains where the ``To'' column has been left empty are crossed tails with crosses of multiplicity $1$. If $\s$ is principal and $e_{\s} > 1$ then $\Gamma_{\s}$ has the following tails with parameters $(t_1,\mu)$:
    \begin{center}
    \begin{tabular}{|c|c|c|c|c|}
    \hline 
        Name & Number & $t_1$ & $\mu$ & Condition \\ \hline
        $T_{\infty}$ & $1$ & $(g(\s) + 1)d_{\s}-\lambda_{\s}$ & $1$ & $\s$ odd \\\hline
        $T_{\infty}^{\pm}$ & $2$ & $-d_{\s}$ & $1$ &  $\s$ even and $\epsilon_{\s} = 1$ \\ \hline
        $T_{\infty}$ & $1$ & $-d_{\s}$ & $2$ &  $\s$ even, $\epsilon_{\s}=-1$ and $e_{\s} > 2$ \\\hline
        $T_{y=0}$ & $|\singletonsofs|/b_{\s}$ & $-\lambda_{\s}$ & $b_{\s}$ & $e_{\s} = 2b_{\s}$ \\ \hline
    \end{tabular}
    \end{center}
    
    If $\s'$ is principal and $e_{\s'} > 1$ then $\Gamma_{\s'}$ has the following tails with parameters $(t_1,\mu)$:
    \begin{center}
    \begin{tabular}{|c|c|c|c|c|}
         \hline
         Name & Number & $t_1$ & $\mu$ & Condition \\ \hline
         $T_{y=0}$ & $|\singletonsofs'|/b_{\s'}$ & $-\lambda_{\s'}$ & $b_{\s'}$ &  $e_{\s'} = 2b_{\s'}$  \\ \hline
         $T_{x=0}$ & $1$ & $-d_{\s'}$ & $2$ & $b_{\s'} \mid |\s'|$, $\lambda_{\s'} \not \in \Z$ and $e_{\s'} > 2$\\ \hline
         $T_{x=0}^{\pm}$ & $2$ & $-d_{\s'}$ & $1$ & $b_{\s'} \mid |\s'|$, $\lambda_{\s'} \in \Z$ \\ \hline
         $T_{(0,0)}$ & $1$ & $-\lambda_{\s'}$ & $1$ & $b_{\s'} \nmid |\s'|$ \\ \hline
    \end{tabular}
    \end{center} 
\end{thm}

\begin{remark}
    For this particular type of hyperelliptic curve, $\s$ will be principal unless it is a cotwin (i.e. if $|\s'| = 2g(C)$), and $\s'$ will be principal unless it is a twin. Since we have assumed that $g \geq 2$, these cases cannot coincide. Note that neither $\s$ nor $\s'$ can be \"ubereven in this case.
\end{remark}

\begin{remark}
    \label{rem::upccorrespondence}
    Suppose $\s$ is principal. In $\X_k$ we can see most of the components of $\X_{\widetilde{\s},k}$. The central component $\Gamma_{\s}$ will have the same multiplicity and genus as $\Gamma_{\widetilde{\s}}$, and will have almost the same tails. The only difference being that one or two of the tails (the $(0,0)$-tail in the case $\s'$ is odd and the $(x=0)$-tail(s) otherwise) will instead form either part of a linking chain between $\Gamma_{\s}$ and $\Gamma_{\s'}$ (in the case $\s'$ principal); or a loop or a crossed tail associated to $\s'$ (in the case where $\s'$ is a twin). We will say that the downhill section of the linking chain \textit{corresponds} to this tail. If the linking chain, loop or crossed tail in $\X_k$ has a non-trivial level section, then all the components of the tails in $\X_{\s,k}$ appear in the linking chain(s) in $\X_k$. If the level section has length zero then some of the lower multiplicity components do not appear - we expand on this in Section \ref{subsec:upcsmalldistance}.
    
    Similarly, if $\s'$ is principal, we see most of the components of $\X_{\widetilde{\s'},k}$ in $\X_k$. In this case, $\Gamma_{\s'}$ has the same tails as $\Gamma_{\widetilde{\s'}}$ except that the infinity tail(s) of the latter are absorbed into the linking chain(s) $L_{\s,\s'}$ (or the loop or crossed tail arising from $\s$ if it is a cotwin). In this case, we say that the uphill section of the linking chain \textit{corresponds} to the infinity tail in $\X_{\widetilde{\s'},k}$. We shall see that this is a phenomenon which generalises to the main theorems in Section \ref{sec::mainthm}.
\end{remark}

\begin{remark}
    The length of the level section of a linking chain, loop or crossed tail $\Ccal \subseteq \X_k$ (that is, the number of $\PP^1$s with multiplicity $\mu$) is equal to $\left|\left(\mu t_2, \mu(t_2 + \delta) \right) \cap \Z\right|$. Let $\Y$ be the minimal regular model of $C$ over $L$, $q:\Y \rightarrow \ZZ$ be the quotient by $\Gal(L/K)$ and $\phi : \X\to\ZZ$ the resolution of singularities. Then any irreducible component $E$ in the level section of $\Ccal$ is \textit{not} an exceptional divisor - that is to say, it is the image of $\mu$ components of $\Y_k$ which are permuted by $\Gal(L/K)$. This can be seen by looking at the explicit automorphisms on the components of $\Y$ given in \cite[Theorem~6.2]{DDMM18}.
\end{remark} 

\begin{eg}
\label{eg::compareegs}
Consider the hyperelliptic curve $C:y^2=(x^2-p)(x^3-p^5)$ over $K=\Qur$. The special fibre of the minimal SNC model of $C/K$ can be seen in Figure \ref{fig:uniqpropchild1}. The central components $\Gamma_{\s}$, and $\Gamma_{\s'}$ are labeled and shown in bold.
    \begin{figure}[h]
	\centering
	\begin{tikzpicture}
	\fill (-4,-0.3) circle (1.5pt);
    \fill (-3.75,-0.3) circle (1.5pt);
    \fill (-3.5,-0.3) circle (1.5pt);
    \fill (-3,-0.3) circle (1.5pt);
    \fill (-2.75,-0.3) circle (1.5pt);
    
    \draw (-3.75,-0.3) ellipse (0.5cm and 0.3cm) node[below, yshift = -0.1cm, xshift = 0.5cm]{$\s'$} node[above, yshift=0.1cm, xshift = 0.5cm]{$\frac{5}{3}$};
    \draw (-3.375,-0.3) ellipse (1.25cm and 0.8cm)node[below, yshift = -0.3cm, xshift = 1.2cm]{$\s$} node[above, yshift=0.3cm, xshift = 1.2cm]{$\frac{1}{2}$};	
	
	\draw (0,0)[line width = 0.5mm] -- node[above, font=\small] {$4$}  node[right, font=\small, xshift=1.3cm] {$\Gamma_{\s}$} ++ (2.75,0);
    \draw (0.25,-0.2) -- node[left, font=\small, yshift=0.1cm] {$1$} ++ (0,0.8);
    \draw (0.75,-0.2) -- node[left, font=\small, yshift=0.1cm] {$2$} ++ (0,0.8);
    \draw (2.5,0.2) -- node[left, font=\small] {$1$} ++ (0,-1);
    \draw (0,-0.6)[line width = 0.5mm] -- node[below, font=\small] {$3$}  node[right, font=\small, xshift=1.3cm] {$\Gamma_{\s'}$}++ (2.75,0);
    \draw (0.25,-0.4) -- node[left, font=\small, yshift=-0.1cm] {$1$} ++ (0,-0.8);
    \draw (0.75,-0.4) -- node[left, font=\small, yshift=-0.1cm] {$1$} ++ (0,-0.8);
    \end{tikzpicture}
    \caption{Cluster picture and special fibre of the minimal SNC model of $C:y^2=(x^2-p)(x^3-p^5)$.}
\label{fig:uniqpropchild1}
\end{figure}
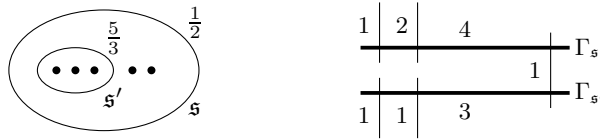

If we consider the curves $C_{\tilde{\s}}$ and $C_{\tilde{\s'}}$ and the special fibres of their minimal SNC models we find that they are as pictured in Figure \ref{fig::corsmallmodeg} below.
\begin{figure}[h]
\begin{adjustwidth}{-0.5cm}{}
\centering
\begin{subfigure}{0.5\textwidth}
  \centering
  \begin{tikzpicture}
    \draw (0,0)[line width = 0.5mm] -- node[above, xshift=0.6cm, font=\small] {$4$} ++ (2.75,0);
    \draw (0.25,-0.2) -- node[left, font=\small, yshift=0.1cm] {$1$} ++ (0,0.8);
    \draw (0.75,-0.2) -- node[left, font=\small, yshift=0.1cm] {$2$} ++ (0,0.8);
    \draw (1.25,-0.2) -- node[left, font=\small, yshift=0.1cm] {$1$} ++ (0,0.8);
    \end{tikzpicture}
  \caption{where $C_{\tilde{\s}}:y^2=x(x^2-p)$.}
  \label{fig::corsmallmodeg1}
\end{subfigure}
\begin{subfigure}{.5\textwidth}
  \centering
  \begin{tikzpicture}
    \draw (0,0)[line width = 0.5mm] -- node[above, xshift=0.6cm, font=\small] {$3$} ++ (2.75,0);
    \draw (0.25,-0.2) -- node[left, font=\small, yshift=0.1cm] {$1$} ++ (0,0.8);
    \draw (0.75,-0.2) -- node[left, font=\small, yshift=0.1cm] {$1$} ++ (0,0.8);
    \draw (1.25,-0.2) -- node[left, font=\small, yshift=0.1cm] {$1$} ++ (0,0.8);
    \end{tikzpicture}
  \caption{where $C_{\tilde{\s'}}:y^2=p(x^3-p^5)$.}
  \label{fig::corsmallmodeg2}
\end{subfigure}
\caption{The special fibres of the minimal SNC models of $C_{\tilde{\s}}$ and $C_{\tilde{\s'}}$}
\label{fig::corsmallmodeg}
\end{adjustwidth}
\end{figure}
We can see that all the components in both Figures \ref{fig::corsmallmodeg1} and \ref{fig::corsmallmodeg2} also appear in the special fibre of the minimal SNC model of $C$. They are glued together along one of their multiplicity one components which forms the linking chain in Figure \ref{fig:uniqpropchild1}. This provides a visualisation of what we mean when we say the tails of $\Gamma_{\s}$ correspond to those of $\Gamma_{\tilde{\s}}$, the tails of $\Gamma_{\s'}$ correspond to those of $\Gamma_{\tilde{\s'}}$, and that some of these tails form part of the linking chains of the special fibre of the minimal SNC model of $C$.
\end{eg}

The proof of Theorem \ref{thm:uniquepropchild} will be presented as a series of lemmas. 

\begin{lemma}
    \label{lem::upccomponents}
    If $\s$ is principal then the special fibre has an irreducible component $\Gamma_{\s} = \Gamma_{F_1}$ of multiplicity $e_{\s}$ and genus $g(\s)$. If $\s'$ is principal then there is a component $\Gamma_{\s'} = \Gamma_{F_2}$ of multiplicity $e_{\s'}$ and genus $g(\s')$.
\end{lemma}
\begin{proof}
    Follows from Lemmas \ref{lem::corofnewpolys1} and \ref{lem:corofnetwonpolys2}.
\end{proof}

\begin{remark}
    Lemma \ref{lem::upccomponents} further proves that $\delta_{F_1} = e_{\s}$ and $\delta_{F_2} = e_{\s'}$ since, by Theorem \ref{thm::timsmainthm}, $\Gamma_{F_i}$ has multiplicity $\delta_{F_i}$.
\end{remark}



\begin{lemma}
    \label{lem::gammastail}
    If $\s$ is principal and $e_{\s} > 1$, the following tails of $\Gamma_{\s}$ arise from outer $v$-edges of the $v$-face $F_1$ in Figure \ref{fig::newtonpolyuniqpropchild}, with conditions as in Theorem \ref{thm:uniquepropchild}:
    \begin{enumerate}
        \item $\infty$-tail(s) arising from the $v$-edge connecting $(0,2)$ and $(|\s|,0)$,
        \item $(y=0)$-tail(s) arising from the $v$-edge connecting $(|\s'|,0)$ and $(|\s|,0)$.
    \end{enumerate}
\end{lemma}
\begin{proof}
    This is a consequence of our discussion above, relating $F_1$ to the Newton polytope of $C_{\widetilde{\s}}$. The conditions in Theorem \ref{thm:uniquepropchild} for the tails to occur follow since $\epsilon_{\s} = (-1)^{v_K(c_f)}$.
\end{proof}

\begin{lemma}
    \label{lem::gammasprimetails}
    If $\s'$ is principal and $e_{\s'} > 1$, the following tails of $\Gamma_{\s'}$ arise from outer $v$-edges of the $v$-face $F_2$ in Figure \ref{fig::newtonpolyuniqpropchild}, with conditions as in Theorem \ref{thm:uniquepropchild}:
    \begin{enumerate}
        \item if $b_{\s'} \mid |\s'|$, $(x=0)$-tail(s) arise from the $v$-edge connecting $(0,0)$ and $(0,2)$,
        \item if $b_{\s'} \nmid |\s'|$, a $(0,0)$-tail arises from the $v$-edge connecting $(1,0)$ and $(0,2)$,
        \item in both cases, $(y=0)$-tail(s) arise from the $v$-edge intersecting the $x$-axis.
    \end{enumerate}
\end{lemma}
\begin{proof}
    This is a consequence of our discussion above, relating $F_2$ to the Newton polytope of $C_{\widetilde{\s}'}$. The conditions in Theorem \ref{thm:uniquepropchild} for these tails to occur follow since $\epsilon_{\s'} = (-1)^{\nu_{\s'} - |\s'|d_{\s'}}$.
\end{proof}

In order to find the lengths of the level sections of the linking chains, we must calculate the slopes of the unique inner $v$-edge $L$, adjacent to both $v$-faces $F_1$ and $F_2$ in Figure \ref{fig::newtonpolyuniqpropchild}.

\begin{lemma}
\label{lem::slopeslemma}
    If $\s'$ is odd $s_1^L=-\lambda_{\s}$, and $s_2^L=-\lambda_{\s}-\frac{\delta_{\s'}}{2}$. Else $s_1^L = - \delta_Ld_{\s}$, and $s_2^L=-\delta_L d_{\s'}$.
\end{lemma}
\begin{proof}
Suppose $s'$ is odd. Then the only points in $\overline{L}(\Z)$ are the endpoints $(0,2)$ and $(|\s'|,0)$, so $\delta_L=1$. The unique function $L^*_{F_1} : \Z^2 \rightarrow \Z$ such that $\restr{L^*_{F_1}}{L} = 0$ and $\restr{L^*_{F_1}}{F_1} \geq 0$ is given by
$$L^*_{F_1}(x,y) = 2x + |\s'|y - 2|\s'|.$$
To calculate $s_1^L$ and $s_2^L$ we need $P_0$ and $P_1$ such that $L^*_{F_1}(P_1) = 1$ and $L^*_{F_1}(P_0) = 0$. We will take $P_0=(|\s'|,0)$ and $P_1=(\frac{|\s'|+1}{2},1)$. The unique affine function which agrees with $v_{\Delta}$ on $F_1$ is defined by
$v_1(x,y) = \nu_{\s} - d_{\s}x - \frac{\nu_{\s}}{2}y.$
Therefore,
$$ s_1^L = \delta_L(v_1(P_1) - v_1(P_0))
     = \nu_{\s} - d_{\s}\frac{|\s'|+1}{2} - \frac{\nu_{\s}}{2} - \nu_{\s} + d_{\s}|\s'|
    =-\left(\frac{\nu_{\s}}{2}-d_{\s}\frac{|\s'|-1}{2}\right)
    = -\lambda_{\s}.$$

The calculations for $s_2^L$ and $\s'$ even are similar.
%
\end{proof}

\begin{proof}[Proof of Theorem \ref{thm:uniquepropchild}]
     Recall that $e_{\s}$ is the minimum integer such that $e_{\s}d_{\s}\in\Z$, and $e_{\s}\nu_{\s}\in2\Z$. If $e_{\s} = 1$ then  $d_{\s},\lambda_{\s} \in \Z$, hence the slopes of the outer $v$-edges of $F_1$ are integers and $\Gamma_{\s}$ has no tails. If $e_{\s}>1$ then Lemma \ref{lem::gammastail} describes the tails of $\Gamma_{\s}$. Similarly if $e_{\s'} = 1$ then $\Gamma_{\s'}$ has no tails and if $e_{\s'}>1$ then Lemma \ref{lem::gammasprimetails} describes the tails of $\Gamma_{\s'}$. The statement on the parameters of the tails and the linking chain follows from Remark \ref{rem::timsmainrem} and the calculation of the slopes in Lemma \ref{lem::slopeslemma}. The multiplicity of the level section is $\delta_L$ where $L$ is the inner $v$-edge between $F_1$ and $F_2$.
    
    The two cases left to worry about are when $\s'$ is a twin or when $\s$ is a cotwin. We will only argue the case where $\s'$ is a twin, as the case where $\s$ is a cotwin is proved similarly. Recall from Remark \ref{rem::epsilon} that $\epsilon_{\s'} = (-1)^{\nu_{\s'} - |\s'|d_{\s'}}$. So, $\epsilon_{\s'} = 1$ if and only if $v_{\Delta}((|\s'|,0))
    =\nu_{\s'}-|\s'|d_{\s'}\in 2\Z$.
    
    Suppose that $\epsilon_{\s'} = 1$. Since $v_{\Delta}(0,2)=0\in 2\Z$ and $|\s'|=2$ we have that $(\frac{|\s'|}{2},1)=(1,1)\in\Z^2$, and $v_{\Delta}(1,1)\in\Z$. So, $|\overline{L}(\Z)_\Z|=3$ and by Theorem \ref{thm::timsmainthm} there are two linking chains from $\Gamma_{\s}$ to the component $\Gamma_{F_2}$ arising from the $v$-face $F_2$ of $\Delta_v(C)$ in Figure \ref{fig::newtonpolyuniqpropchild}. The component $\Gamma_{F_2}$ is exceptional by \cite[Proposition~5.2]{Dok18} and the linking chains between $\Gamma_{\s}$ and $\Gamma_{F_2}$ are minimal. After blowing down $\Gamma_{F_2}$, this results in a loop from $\Gamma_{\s}$ to itself. 
    
    Suppose instead that, $\epsilon_{\s'}= -1$. Then there is a single chain of rational curves from $\Gamma_{\s}$ to $\Gamma_{F_2}$, and $\Gamma_{F_2}$ has two rational other rational curves intersecting it transversely (which arise from the $v$-edge connecting $(0,0)$ and $(0,2)$). Therefore, $\Gamma_{F_2}$ is not exceptional and must appear in the minimal SNC model. This means, if we consider $\Gamma_{F_2}$ as a component of the level section, that this chain of rational curves is a crossed tail.
\end{proof}

\subsection{Small Distances}
\label{subsec:upcsmalldistance}
Let $\s_1$ and $\s_2$ be the  principal clusters such that there is a linking chain $\mathcal{C}\subseteq \X_k$ from $\Gamma_{\s_1}$ to $\Gamma_{\s_2}$. If $\mathcal{C}$ has level section of length greater than 0, it is straightforward to compare the multiplicities of $\mathcal{C}$ to those of the corresponding tails (see Remark \ref{rem::upccorrespondence}). All of the multiplicities of the corresponding tails appear in the uphill and downhill sections of $\mathcal{C}$. However, if the level section is empty and the downhill section of $\mathcal{C}$ corresponds to a tail, say $\mathcal{T}_1$, then not all of the multiplicites of $\mathcal{T}_1\subseteq \X_{\widetilde{\s_1},k}$ appear in the downhill section of $\Ccal$. Similarly if the uphill section corresponds to a tail, say $\mathcal{T}_2\subseteq \X_{\widetilde{\s_2},k}$. We shall show that in this case, $\mathcal{T}_1$ and $\mathcal{T}_2$ ``meet'' at a component of second least common multiplicity. In other words, if we consider a chain of rational curves $\Ccal'$ such that $\Ccal'$ has level section of length $1$, and whose downhill and uphill sections correspond to $\mathcal{T}_1$ and $\mathcal{T}_2$ respectively, then we ``cut out'' a section of $\Ccal'$ to obtain $\Ccal$. 

\begin{eg} Consider the hyperelliptic curves given by $y^2=(x^4 -p)(x^5-p^{2+10m})$ over $K=\Qur$ for $m\in\Z_{\geq 0}$, with cluster pictures shown in Figure \ref{fig:smalldistcuttingex}.
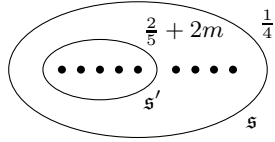
\begin{figure}[ht]
	\centering
\begin{tikzpicture}

\fill (0,0) circle (1.5pt);
\fill (0.25,0) circle (1.5pt);
\fill (0.5,0) circle (1.5pt);
\fill (0.75,0) circle (1.5pt);
\fill (1,0) circle (1.5pt);

\fill (1.5,0) circle (1.5pt);
\fill (1.755,0) circle (1.5pt);
\fill (2,0) circle (1.5pt);
\fill (2.25,0) circle (1.5pt);

\draw (0.5,0) ellipse (0.75cm and 0.4cm) node[below, yshift = -0.15cm, xshift = 0.7cm,font=\small]{$\s'$} node[above, yshift = 0.2cm, xshift = 1.1cm, font=\small]{$\frac{2}{5}+2m$};

\draw (1,0) ellipse (1.7cm and 0.9cm) node[below, yshift = -0.5cm, xshift = 1.5cm,font=\small]{$\s$} node[above, yshift = 0.3cm, xshift = 1.7cm,font=\small]{$\frac{1}{4}$};
\end{tikzpicture}
	\caption{Cluster picture $\Sigma_C$ of $C:y^2=(x^4-p)(x^5-p^{2+10m})$.}
	\label{fig:smalldistcuttingex}
\end{figure}
The level section of the linking chain between $\Gamma_{\s}$ and $\Gamma_{\s'}$ has length $m$. Figure \ref{fig::largetosmallpic} shows the special fibres of the minimal SNC models for both when $m=1$, and the small distance case (when $m=0$). Here we can see that when $m=1$ the uphill and downhill sections of the linking chain have a common multiplicity greater than 1, namely 3, and that to obtain the $m=0$ case we remove the dashed section of the linking chain and glue back along the multiplicity 3 components. 
\begin{figure}[h]
\centering
\begin{subfigure}{0.5\textwidth}
  \centering
  \begin{tikzpicture}
\draw (5,0)[line width = 0.5mm] -- node[above,font=\small] {10} ++ (3,0);
\draw (5,-1.7)[line width = 0.5mm] -- node[below,font=\small] {8} ++ (3,0);
\draw(5.5,0.2) -- node[right,font=\small] {3} ++ (-0.9,-0.9);
\draw[dashed] (4.75,-0.3) -- node[right,font=\small] {2} ++ (0,-1);
\draw[dashed] (4.55,-1.1) -- node[below,font=\small] {1} ++ (1,0);
\draw (5.35,-0.9) -- node[right,font=\small] {3} ++ (0,-1);
\draw (7,0.2) -- node[right,font=\small, yshift=-0.1cm] {5} ++ (0,-0.8);
\draw (7.5,0.2) -- node[right,font=\small, yshift=-0.1cm] {2} ++ (0,-0.8);
\draw (7,-1.1) -- node[right,font=\small, yshift=0.1cm] {4} ++ (0,-0.8);
\draw (7.5,-1.1) -- node[right,font=\small, yshift=0.1cm] {1} ++ (0,-0.8);

\node at (8.4,0) {$\Gamma_{\s'}$};
\node at (8.4,-1.7) {$\Gamma_{\s}$};
    \end{tikzpicture}
  \caption{$m=1$ case.}
  \label{fig::largetosmallpicm=1}
\end{subfigure}%
\begin{subfigure}{.5\textwidth}
  \centering
  \begin{tikzpicture}
  \draw (0,0)[line width = 0.5mm] -- node[above,font=\small] {10} ++ (3,0);
\draw (0,-1)[line width = 0.5mm] -- node[below,font=\small] {8} ++ (3,0);
\draw (0.5,0.2) -- node[right,font=\small] {3} ++ (0,-1.4);
\draw (2,-0.2) -- node[right,font=\small, yshift=0.1cm] {5} ++ (0,0.8);
\draw (2.5,-0.2) -- node[right,font=\small, yshift=0.1cm] {2} ++ (0,0.8);
\draw (2,-0.8) -- node[right,font=\small, yshift=-0.1cm] {4} ++ (0,-0.8);
\draw (2.5,-0.8) -- node[right,font=\small, yshift=-0.1cm] {1} ++ (0,-0.8);

\node at (-0.4,0) {$\Gamma_{\s'}$};
\node at (-0.4,-1) {$\Gamma_{\s}$};
    \end{tikzpicture}
  \caption{$m=0$ case.}
  \label{fig::largetosmallpicm=0}
\end{subfigure}
\caption{Example of ``cutting out'' a section of linking chain to obtain the small distance case.}
\label{fig::largetosmallpic}
\end{figure}
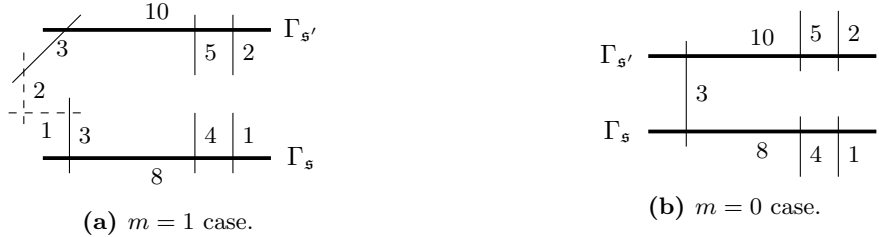
\end{eg}

Let us solidify this with a precise statement.

\begin{thm}
\label{thm::smalldistancelinking}
    Let $\mathcal{C} = \bigcup_{i=1}^{\lambda} E_i $ be a sloped chain of rational curves with parameters $(t_2,t_1,\mu)$, as in Definition \ref{def::linkingchainsections}. Suppose that $\Ccal$ has level section length $0$ and $[\mu t_2,\mu t_1]\subset (0,1)$. Suppose $E_i$ has multiplicity $\mu_i$ and the downhill section comprises of $E_i$ for $1 \leq i \leq l$, for some $l\in\Z$ with $1\leq l\leq \lambda$. Let $\mathcal{T}_j = \bigcup_{i=1}^{\lambda_j} F_i^{(j)}$ for $j=1,2$ be tails (with $\mathcal{T}_j$ possibly empty, in which case $\lambda_j=0$), where $\mathcal{T}_1$ has parameters $(t_1, \mu)$ and $\mathcal{T}_2$ has parameters $(\frac{1}{\mu} - t_2, \mu)$. Let $F_i^{(j)}$ have multiplicity $\mu_i^{(j)}$ (and write $\mu_0^{(j)} = \mathrm{denom}(\mu t_j)$), and let $l_j < \max(1,\lambda_j)$ be maximal such that $\mu_{l_1}^{(1)} = \mu_{l_2}^{(2)}$. Then $\mu_i = \mu_i^{(1)}$ for $1 \leq i \leq l_1$ and $\mu_{\lambda + 1 - i} = \mu_i^{(2)}$ for $1 \leq i \leq l_2$.
\end{thm}

\begin{remark}
    Let $\mathcal{C}$ be as in Theorem \ref{thm::smalldistancelinking}. Since the level section of $\mathcal{C}$ is empty, it must be the case that $(\mu t_2, \mu t_1) \cap \Z = \emptyset$. Therefore, after shifting $\mu t_2$ and $\mu t_1$ by an integer if necessary, we may insist that $[\mu t_2, \mu t_1] \subseteq [0,1]$. If $\mu t_2 \in \Z$ (hence $\mathcal{T}_2$ is empty) then it is immediate from Remark \ref{rem::timsmainrem} that $\lambda = \lambda_1 - 1$ and $\mu_i = \mu_i^{(1)}$ for $1 \leq i \leq \lambda$, since the multiplicities come from the same sequence of fractions. A similar conclusion applies if $\mu t_1 \in \Z$. So we are able to assume without loss of generality that $\mu t_2, \mu t_1 \not \in \Z$, hence our assumption in Theorem \ref{thm::smalldistancelinking} that $[\mu t_2,\mu t_1]\subset (0,1)$.
\end{remark}

Roughly, Theorem \ref{thm::smalldistancelinking}, states that when there is no level section, rather than seeing all of the multiplicities of the tails which the uphill and downhill sections correspond to, the two tails ``meet'' at the component of minimal shared multiplicity greater than $\mu$. Before we prove this theorem, let us prove a couple of lemmas.

\begin{lemma}
    \label{lem:intermediatefrac}
    Let $q_1,q_2\in \Q$ with $[q_1,q_2]\cap\Z=\emptyset$. Then there is a unique fraction with minimal denominator in the set $[q_1,q_2]\cap\Q$, when written with coprime numerator and denominator.
\end{lemma}
\begin{proof}
    Suppose not, and suppose $r_1,r_2 \in [q_1,q_2]\cap\Q$ can be written $r_i = \frac{m_i}{d}$ with $m_i,d$ coprime and $d$ the minimal denominator of elements in the set $[q_1,q_2]\cap\Q$. We will show that there exists a rational number $r$ lying between $r_1$ and $r_2$ of denominator $<d$.
    
    Write $r_i = \frac{m_i(d-1)}{d(d-1)}$, and consider the set $S = [m_1(d-1), m_2(d-1)] \cap \Z$. Since $m_2 > m_1$ and $m_1,m_2\in\Z$, $|S| \geq d$ and there must exist a multiple of $d$ in $S$. That is, there exists $m\in\Z$ such that $md\in S$. Since $m_i$ and $d$ are coprime, we have $m_1 < md < m_2$. Therefore,
    $$\frac{m_1(d-1)}{d(d-1)} < \frac{md}{d(d-1)} < \frac{m_2(d-1)}{d(d-1)}\Longrightarrow r_1 < \frac{m}{d-1} < r_2,$$
    which contradicts the minimality of $d$.
\end{proof}

\begin{lemma}
    With notation as in Theorem \ref{thm::smalldistancelinking}, there exists some $l_j<\lambda_j$, for $j=1,2$, such that $\mu^{(1)}_{l_1} = \mu_{l_2}^{(2)}$.
\end{lemma}
\begin{proof}
    Write $s_i = \mu t_i$. Recall that we assumed that, $[s_2, s_1]\subset(0,1)$, so $[s_2, s_1] \cap \Z = \emptyset$. Let $\frac{m}{d}$ be the unique fraction of minimal denominator in $[s_2, s_1]$, which exists by Lemma \ref{lem:intermediatefrac}. Then if 
    \begin{align*}
    s_1=\mu t_1=\frac{m_0}{d_0}>\frac{m_1}{d_1}>\dots>\frac{m_{\lambda}}{d_{\lambda}}>\frac{m_{\lambda+1}}{d_{\lambda+1}}= \mu t_2=s_2,
    \end{align*}
    is the reduced sequence giving rise to the linking chain $\mathcal{C}$, as in Remark \ref{rem::timsmainrem}, where $(m_i,d_i)=1$, $d_0>\dots>d_l$ and $d_l<\dots<d_{\lambda+1}$ for some $1\leq l\leq\lambda$, we must have that $d_l = d$.
    
    Consider the following two reduced sequences:
    \begin{align*}
        \mu t_1=\frac{m^{(1)}_0}{d^{(1)}_0}>\frac{m^{(1)}_1}{d^{(1)}_1}>\dots>\frac{m^{(1)}_{\lambda_1}}{d^{(1)}_{\lambda_1}}>\frac{m^{(1)}_{\lambda_1+1}}{d^{(1)}_{\lambda_1+1}}= -1,\\
         1 - \mu t_2=\frac{m_0^{(2)}}{d_0^{(2)}}>\frac{m_1^{(2)}}{d_1^{(2)}}>\dots>\frac{m^{(2)}_{\lambda_2}}{d^{(2)}_{\lambda_2}}>\frac{m^{(2)}_{\lambda_2+1}}{d^{(2)}_{\lambda_2+1}}= -1.
    \end{align*}
    These give rise to the multiplicites $\mu_i^{(j)}=\mu\cdot d^{(j)}_i$ for $1\leq i \leq \lambda_j$, $j=1,2$ of the tails $\mathcal{T}_j$. We will show that there exist $0\leq l_1 < \lambda_1+1$ and $0\leq l_2 < \lambda_2+1$ with $d^{(1)}_{l_1} = d=d^{(2)}_{l_2}.$ 
    
    We will first prove that $d^{(1)}_{l_1} = d$ for some $l_1\in\Z$. Since $[s_2,s_1]\subset (0,1)$, we have that $s_2 > \lfloor s_1 \rfloor=0$. So, some fraction of denominator $d$, say $\frac{m}{d}$, appears in the full sequence of fractions in $[\lfloor s_1 \rfloor, s_1] \cap \Q$ of denominator less than or equal to $\max\{d_0, d_{\lambda+1}\}$. To obtain a reduced sequence, we remove all terms of the form
    \begin{align*}
    \dots>\frac{a}{b}>\frac{a+c}{b+d}>\frac{c}{d}>\dots\mapsto\dots>\frac{a}{b}>\frac{c}{d}>\dots,
    \end{align*}
    as in Remark \ref{rem::timsmainrem}. We can only remove $\frac{m}{d}$ if there exists some $q \in \Q$ with $\mathrm{denom}(q) < d$ and $s_1 > q > \frac{m}{d}$. No such $q$ can exists since $d$ is the minimal denominator of any element of $[s_2,s_1]\cap\Q$. Therefore, $\frac{m}{d}$ cannot be removed in the reduction proccess and so must apppear in the reduced sequence. Therefore there exists $0\leq l_1 < \lambda_1+1$ such that $d^{(1)}_{l_1} = d.$ Proving that there exits $0\leq l_2 < \lambda_2+1$ such that $d^{(1)}_{l_1} = d=d^{(2)}_{l_2}$ is done similarly.
    \end{proof}

We can now prove Theorem \ref{thm::smalldistancelinking}.
    \begin{proof}[Proof of Theorem \ref{thm::smalldistancelinking}]
        The fractions $\frac{m_0}{d_0},\frac{m_1}{d_1},\ldots,\frac{m_l}{d_l}$ in the reduced sequence depend only on the elements of $[s_1, \frac{m_l}{d_l}]$ of denominator less than or equal to $\max(d_0, d_{\lambda+1})$, as do the fractions $\frac{m_0^{(1)}}{d_0^{(1)}},\ldots,\frac{m_{l_1}^{(1)}}{d_{l_1}^{(1)}} = \frac{m_l}{d_l}$. This proves that $d_i^{(1)} = d_i$ hence $\mu_i = \mu_i^{(1)}$ for $1 \leq i \leq l_1$. Similarly $d_i^{(2)} = d_{\lambda + 1 - i}$ hence $\mu_{\lambda + 1 - i} = \mu_i^{(2)}$ for $1 \leq i \leq l_2$. It remains to show maximality of $l_1$ and $l_2$. 
        
        Suppose there is some $r_1,r_2$ such that $\lambda_i > r_i > l_i$ and $\mu^{(1)}_{r_1} = \mu^{(2)}_{r_2} < \mu^{(1)}_{l_1}$. In addition to this, $d^{(1)}_{r_1} = d^{(2)}_{r_2} < d$ (recall $\frac{m}{d}$ is the unique fraction with least denominator in $[s_2,s_1]\cap\Q$). Therefore $q_2 = 1 - \frac{m^{(2)}_{r_2}}{d^{(2)}_{r_2}} \in (s_1, 1]$ and $q_1 = \frac{m^{(1)}_{r_1}}{d^{(1)}_{r_1}} \in [0,s_2)$. Let $q'$ be the unique rational with least denominator $d'$ in $\left[q_1, q_2\right]$. By uniqueness, $d' < d^{(1)}_{r_1} < d$. Therefore, $q' \in (s_1, q_2)$ or $(q_1, s_2)$. Suppose for now that $q' \in (s_1, q_2)$, and consider again the reduced sequence
         \begin{align*}
    1 - \mu t_2=\frac{m_0^{(2)}}{d_0^{(2)}}>\frac{m_1^{(2)}}{d_1^{(2)}}>\dots>\frac{m^{(2)}_{\lambda_2}}{d^{(2)}_{\lambda_2}}>\frac{m^{(2)}_{\lambda_2+1}}{d^{(2)}_{\lambda_2+1}}= -1.
    \end{align*}
    However $1 - q_2$ cannot appear in this reduced sequence since a fraction with smaller denominator, $1-q'$, appears to the left of it in the non-reduced sequence. So, at some step in the reduction process $1-q_2$ would have been removed. Therefore, $q' \not \in (s_1,q_2)$. Similarly, one can show that $q' \not \in (q_1,s_2)$. This is a contradiction. So no such $r_1$ and $r_2$ exist.
    \end{proof}

\section{Main Theorems}
\label{sec::mainthm}

The previous two sections looked at the minimal SNC models of specific cases of hyperelliptic curves. In this section, we state our main theorems in full generality. Theorem \ref{thm::structureofSNCmodel} gives the structure of the special fibre of the minimal SNC model 
, and Theorems \ref{thm::maincentral} and \ref{thm::chaindescription} give more explicit descriptions of multiplicities and genera of components appearing in the special fibre.

\subsection{Orbits}
\label{subsec::orbits}

Before we can state and prove the main results of this paper, we need to extend some of the definitions of Section \ref{sec::clusters}. Since the definitions in Section \ref{sec::clusters} come from \cite{DDMM18}, where the authors deal only with the semistable case, they do not deal with orbits of clusters. So, here we make some new definitions which extend the preexisting ones to orbits.

\begin{definition}
\label{def::principalorbit}
    Let $X$ be a Galois orbit of clusters. Then $X$ is \emph{\"ubereven} if for all $\s \in X$, $\s$ is \"ubereven. Define an orbit $X$ to be \emph{odd}, \emph{even}, and \emph{principal} similarly.
\end{definition}

\begin{definition}
\label{def::K_X}
    Let $X$ be an orbit of clusters. Define $K_X/K$ to be the field extension of $K$ of degree $|X|$. By Lemma \ref{lem::orbitssizeb}, $K_X/K$ is the minimal field extension over which for any $\s\in X$, $\sigma\in\Gal(\overline{K}/K_X)$ we have $\sigma(\s)=\s$. 
\end{definition}


\begin{definition}
\label{def::orbinvars}
    Let $X$ be a Galois orbit of clusters, and choose some $\s\in X$. Then we define
    $$ d_X = \frac{a_X}{b_X} = d_{\s},\quad \nu_X = \nu_{\s},\quad \lambda_X = \lambda_{\s},\quad \ssg{X} = \ssgs,\quad\textrm{and}\quad \epsilon_X = \epsilon_{\s}^{|X|}$$
\end{definition}

\begin{remark}
    Note that the invariants defined in Definition \ref{def::orbinvars} are well defined, i.e they do not depend on the choice of $\s\in X$. 
\end{remark}

\begin{definition}
\label{def::orbitchild}
    An orbit $X'$ is a \emph{child} of $X$, written $X' < X$, if for every $\s' \in X'$ there exists some $\s \in X$ such that $\s' < \s$. Define $\delta_{X'} = \delta_{\s'}$ for some $\s'\in X'$.
\end{definition}

\begin{definition}
    Let $X$ be a principal orbit of clusters with $\ssg{X} > 0$ and choose some $\s\in X$. Then $C_{\widetilde{X}}$ is defined to be the curve $C_{\tilde{\s}}$ over $K_X$. We denote the minimal SNC model of $C_{\widetilde{X}}/K_X$ by $\X_{\widetilde{X}}/\OO_{K_X}$, and the central component by $\Gamma_{\widetilde{X}}/k$.
\end{definition}

\begin{remark}
    The curve $C_{\widetilde{X}}$ depends on a choice of $\s \in X$, but the combinatorial description of the special fibre of the minimal SNC model will not. Since this is what we need $C_{\widetilde{X}}$ for, we do not need to worry about this.
\end{remark}

\begin{definition}
\label{def::e_X}
    Let $X$ be a principal orbit of clusters. Define $e_X$ to be the minimal integer such that $e_X|X|d_{\s} \in \Z$ and $e_X|X|\nu_{\s}\in 2\Z$ for all $\s\in X$. Define $g(X) = g(\s)$ for $\s \in X$ over $K_X$, where $g(\s)$ is as defined in Definition \ref{def::esgs}.
\end{definition}

\begin{remark}
    Analogously to Section \ref{subsec::cso}, the curve $C_{\widetilde{X}}/K_X$ is semistable over an extension of $K_X$ of degree $e_X$ and the quotient map $\Gamma_{\s,L} \rightarrow \Gamma_{\s,K_X}$ has degree $e_X$ for $\s \in X$.
\end{remark}

\subsection{The Special Fibre of the Minimal SNC Model}

We state here the first of our main theorems. Roughly this tells us that the cluster picture, the leading coefficient of $f$, and the action of $\GK$ on the cluster picture is enough to calculate the structure of the minimal SNC model, along with the multiplicities and genera of the components.

\begin{thm}
\label{thm::main1}
    Let $K$ be a complete discretely valued field with algebraically closed residue field of characteristic $p>2$. Let $C:y^2 = f(x)$ be a hyperelliptic curve over $K$ with tame potentially semistable reduction and cluster picture $\Sigma_{C/K}$. Then the dual graph, with genus  multiplicity, of the special fibre of the minimal SNC model of $C/K$ is completely determined by $\Sigma_{C/K}$ (with depths), the valuation of the leading coefficient $v_K(c_f)$ of $f$, and the action of $\GK$.
\end{thm}

\begin{remark}
    If $K$ does not have algebraically closed residue field, then the Frobenius action on the dual graph is determined by this data, as well as the values of $\epsilon_X{\Frob}$ for each orbits of clusters $X$. See Theorem \ref{thm::frobaction}.
\end{remark}

The proof of this will follow from the theorems proved in the rest of this section, and we make this more precise later. First we split Theorem \ref{thm::main1} into several smaller theorems. The first tells us which components appear in the special fibre of the minimal SNC model. Roughly, there is a central component for every orbit of principal, non \"ubereven clusters, one or two central components for every orbit of principal \"ubereven clusters, and a chain of rational curves associated to each orbit of twins. These central components are linked by chains of rational curves, and certain central components will also have tails intersecting them. The following theorem gives us the structure of the special fibre but is missing important details such as multiplicities, genera and lengths of these chains. These remaining details will be discussed in a later theorem.

\begin{thm}[Structure of SNC model]
    \label{thm::structureofSNCmodel}
    Let $K$ be a complete discretely valued field with algebraically closed residue field of characteristic $p>2$. Let $C/K$ be a hyperelliptic curve with tame potentially semistable reduction. Then the special fibre of its minimal SNC model is structured as follows. Every principal Galois orbit of clusters $X$ contributes one central component $\Gamma_X$, unless $X$ is \"ubereven with $\epsilon_X = 1$, in which case $X$ contributes two central components $\Gamma_X^+$ and $\Gamma_X^-$.
    
    These central components are linked by chains of rational curves, or are intersected transversely by a crossed tail in the following ways (where, for any orbit $Y$, we write $\Gamma_Y^+ = \Gamma_Y^- = \Gamma_Y$ if $Y$ is not \"ubereven):
    \begin{center}
    \begin{tabular}{|c|c|c|c|}
    \hline
         \small{Name} & \small{From} & \small{To} & \small{Condition} \\ \hline
         $L_{X,X'}$ & $\Gamma_X$ & $\Gamma_{X'}$  & $X' \leq X$ both principal, $X'$ odd \\ \hline
         $L_{X,X'}^+$ & $\Gamma_X^+$ & $\Gamma_{X'}^+$ & $X' \leq X$ both principal, $X'$ even with $\epsilon_{X'} = 1$ \\ \hline
         $L_{X,X'}^-$ & $\Gamma_X^-$ & $\Gamma_{X'}^-$ & $X' \leq X$ both principal, $X'$ even with $\epsilon_{X'} = 1$ \\ \hline
         $L_{X,X'}$ & $\Gamma_X$ & $\Gamma_{X'}$ & $X' \leq X$ both principal, $X'$ even with $\epsilon_{X'} = -1$ \\ \hline
         $L_{X'}$ & $\Gamma_X^-$ & $\Gamma_X^+$ & $X$ principal, $X' \leq X$ orbit of twins, $\epsilon_{X'} = 1$ \\ \hline
         $T_{X'}$ & $\Gamma_X$ & - & $X$ principal, $X' \leq X$ orbit of twins, $\epsilon_{X'} = -1$ \\ \hline
    \end{tabular}
    \end{center}
    Note that any chain where the ``To'' column has been left blank is a crossed tail. If $\mathcal{R}$ is not principal then we also get the following chains of rational curves:
    \begin{center}
        \begin{tabular}{|c|c|c|c|}
        \hline
         \small{Name} & \small{From} & \small{To} & \small{Condition}\\\hline
         $L_{\Rcal}$ & $\Gamma_{\s}^-$ & $\Gamma_{\s}^+$ & $\mathcal{R}$ a cotwin, $\s<\Rcal$ principal of size $2g$, $\epsilon_{\s} = 1$\\ \hline
         $T_{\Rcal}$ & $\Gamma_{\s}$ & - & $\mathcal{R}$ a cotwin, $\s<\Rcal$ principal of size $2g$, $\epsilon_{\s} = -1$\\ 
         \hline
         $L_{\s_1,\s_2}$ & $\Gamma_{\s_1}$ & $\Gamma_{\s_2}$ & $\Rcal = \s_1 \sqcup \s_2$, with $\s_i$ both principal and odd, $e_{\mathcal{R}} = 1$\\ \hline
         $T_X$ & $\Gamma_X$ & - &  $\Rcal = \s_1\sqcup \s_2$, with $X=\{s_1,s_2\}$ a principal, odd orbit\\ \hline
         $L_{\s_1,\s_2}^+$ & $\Gamma_{\s_1}^+$ & $\Gamma_{\s_2}^+$ & $\Rcal = \s_1\sqcup \s_2$, $\s_i$ both principal and even, $e_{\mathcal{R}} = 1$, $\epsilon_{\s_i} = 1$ \\ \hline
         $L_{\s_1,\s_2}^-$ & $\Gamma_{\s_1}^-$ & $\Gamma_{\s_2}^-$ & $\Rcal=\s_1\sqcup\s_2$, $\s_i$ both principal and even, $e_{\mathcal{R}} = 1$, $\epsilon_{\s_i} = 1$ \\ \hline
         $L_{\s_1,\s_2}$ & $\Gamma_{\s_1}$ & $\Gamma_{\s_2}$ & $\Rcal=\s_1\sqcup\s_2$, $\s_i$ both principal and even, $e_{\mathcal{R}} = 1$, $\epsilon_{\s_i} = -1$ \\ \hline
         $T_X^+$ & $\Gamma_X^+$ & - & $\Rcal = \s_1\sqcup \s_2$, with $X=\{s_1,s_2\}$ a principal, even orbit, 
         $\epsilon_{\s_i} = 1$ \\ \hline
         $T_X^-$ & $\Gamma_X^-$ & - & $\Rcal = \s_1\sqcup \s_2$, with $X=\{s_1,s_2\}$ a principal, even orbit, 
         $\epsilon_{\s_i} = 1$ \\ \hline
         $T_X$ & $\Gamma_X$ & - & $\Rcal = \s_1\sqcup \s_2$, with $X=\{s_1,s_2\}$ a principal, even orbit, 
         $\epsilon_{\s_i} = -1$ \\ \hline
         $L_{\tfrak}$ & $\Gamma_{\s}^-$ & $\Gamma_{\s}^+$ & $\mathcal{R} = \s \sqcup \tfrak$, $\s$ principal and even, $\tfrak$ a twin, $\epsilon_{\tfrak} = 1$ \\ \hline 
         $T_{\tfrak}$ & $\Gamma_{\s}$ & - & $\mathcal{R} = \s \sqcup \tfrak$, $\s$ principal and even, $\tfrak$ a twin, $\epsilon_{\tfrak} = -1$\\
         \hline
         \end{tabular}
    \end{center}
    Finally, a central component $\Gamma_X$ is intersected transversally by some tails if and only if $e_X > 1$. These are explicitly described in Theorem \ref{thm::chaindescription}.
\end{thm}

\begin{remark}
    \label{rem::explicitequations} At no point do we give explicit equations for the central components $\Gamma_X^{\pm}$. However, these can be calculated using the method laid out in this paper. In particular, one can take the explicit equations given in \cite[Theorem~8.5]{DDMM18} for the components $\Gamma_{\s,L}^{\pm}$ in the semistable model of $C/L$ and the Galois action on these components, and apply \cite[Theorem~1.1]{DD18}.
\end{remark}

Before we prove this, let us prove a couple of lemmas. Recall that $L$ is a field over which $C$ has semistable reduction and that $\Gamma_{\s,L}$ is the component associated to a cluster $\s$ in the special fibre of the minimal semistable model $\Y$ of $C$ over $L$.

\begin{lemma}
\label{lem::numberofchildren}
    Let $\s$ be a principal cluster with $\ssgs = 0$. \begin{enumerate}
        \item If $\s = \Rcal$ and $\s$ is not \"ubereven (resp. \"ubereven) then $\Gamma_{\s,L}$ (resp. each of $\Gamma_{\s,L}^+$ and $\Gamma_{\s,L}^-$) intersects at least two other components.
        \item If $\s \neq \Rcal$ and $\s$ is not \"ubereven (resp. \"ubereven) then $\Gamma_{\s,L}$ (resp. each of $\Gamma_{\s,L}^+$ and $\Gamma_{\s,L}^-$) intersects at least three other components.
    \end{enumerate}
\end{lemma}
\begin{proof}
(i) Let $\s=\Rcal$ and suppose $\s$ is not \"ubereven. Since $\ssgs = 0$, $\s$ can have at most two odd children and in particular at most two singletons. Since, $g(C) \geq 2$, we have $|\s| \geq 5$. If $|\s|$ is odd then $\s$ must have an even child $\s'$ and, by \cite[Theorem~8.5]{DDMM18}, $\Gamma_{\s,L}$ is intersects by the two linking chains to $\Gamma_{\s',L}$. Note that, since $\s$ is principal, $\s$ cannot be the union of two odd clusters. So, if $|\s|$ is even then $\s$ has an even child and we are done by \cite[Theorem~8.5]{DDMM18}.
        
If $\s=\Rcal$ is \"ubereven then every child of $\s$ is even. In particular, there are at least two even children $\s_1$ and $\s_2$. So, each of $\Gamma_{\s,L}^{\pm}$ intersects $L_{\s_1}^{\pm}$ and $L_{\s_2}^{\pm}$ (the linking chains to the children).
        
(ii) Let $\s\neq\Rcal$ and suppose $\s$ is not \"ubereven. Since $\s$ is principal, we know $|\s|\geq 3$. Therefore, $\s$ must have at least one proper child $\s'$. Suppose that $P(\s)$ is principal. If $\s'<\s$ is even then $\Gamma_{\s,L}$ intersects the linking chain to $\Gamma_{P(\s),L}$ and the two linking chains to $\Gamma_{\s',L}$. Otherwise $\s$ must be the union of two odd clusters, hence $\s$ is even. In this case there are two linking chains to $\Gamma_{P(\s),L}$ and one to $\Gamma_{\s,L}$. A similar argument works if $\s$ is \"ubereven. If $P(\s) = \Rcal = \s \sqcup \s_2$ is not principal, the argument is similar, but linking chains to $\Gamma_{P(\s),L}$ are replaced by linking chains to $\Gamma_{\s_2,L}$.
\end{proof}

\begin{prop}
    \label{prop::centralnotexceptional}
    Let $\Y$ be the semistable model of $C/L$ and $\ZZ$ the imagine under the quotient map. Let $\X$ be the SNC model obtained by resolving the singularities of $\ZZ$ such that all rational chains are minimal. Let $X$ be a principal orbit of clusters. Let $\Gamma_{X,K} \in \X_k$ be the image of $\Gamma_{\s,L}$ for some $\s\in X$ under the quotient by $\Gal(L/K)$. Then if $g(\Gamma_{X,K}) = 0$ and $(\Gamma_{X,K}\cdot \Gamma_{X,K}) = -1$, $\Gamma_{X,K}$ intersects at least three other components of the special fibre (i.e. blowing down $\Gamma_{X,K}$ would not result in an SNC model).
\end{prop}
\begin{proof}
     If $|X| > 1$, there is a non trivial field extension of $K$ to $K_X$. Over $K_X$, each $\s' \in X$ is fixed by $\Gal(\overline{K}/K_X)$. The Galois group $\Gal(K_X/K)$ then induces an \'etale morphism $\bigsqcup_{\s'\in X} \Gamma_{\s',K_X} \rightarrow \Gamma_{X,K}$. Therefore, $g(\Gamma_{X,K}) = g(\Gamma_{\s',K_X})$, $(\Gamma_{X,K}\cdot\Gamma_{X,K}) = (\Gamma_{\s',K_X}\cdot\Gamma_{\s',K_X})$, and $\Gamma_{X,K}$ and $\Gamma_{\s',K}$ intersect the same number of other components. So, it is enough to prove this proposition when $|X|=1$, and from now on let $X = \{\s\}$. When $g(\Gamma_{\s,L} > 0$, Riemann-Hurwitz implies that $$\sum_{P \in \Gamma_{\s,K}} \left( \frac{e_{\s}}{|q^{-1}(P)|} - 1 \right) \geq 2e_{\s}.$$ where $q : \Gamma_{\s,L} \rightarrow \Gamma_{\s,K}$ is the quotient by $\Gal(L/K)$. So, if $g(\Gamma_{\s,L})>0$, there must be at least three points $P\in\Gamma_{\s,K}$ with $|q^{-1}(P)|<e_{\s}$. These ramification points are singular points by Proposition \ref{prop::tpgrSings}. After blowing up these singular points, we see that $\Gamma_{\s,K}$ intersects at least three other components of $\X_k$.
     
    
    It remains to deal with the case when $g(\Gamma_{\s,L})=0$. If $e_{\s} = 1$, Lemma \ref{lem::numberofchildren} implies that $\Gamma_{\s,K}$ intersects two or more other components. In this case $\Gamma_{\s,K}$ will have multiplicity $e_{\s}=1$. This tells us that $(\Gamma_{\s,K} \cdot \Gamma_{\s,K})<-1$, so $\Gamma_{\s,K}$ is not exceptional. 
    
    Suppose instead that $e_{\s} > 1$. We will show that the component $\Gamma_{\s,K}$ intersects at least three components. There are two ramification points $P_0$ and $P_{\infty}$ of the morphism $q:\Gamma_{\s,L}\rightarrow\Gamma_{\s,K}$, the images of $0$ and $\infty$ respectively in $\Gamma_{\s,L}$. Both $P_0$ and $P_{\infty}$ are singularities. If $q^{-1}(P_0)$ is an intersection point of $\Gamma_{\s,L}$ with another component $\Gamma$ then $P_0$ will be the intersection point of $\Gamma_{\s,K}$ and $q(\Gamma)$\footnote{We may have to blow down $q(\Gamma)$ but even then $P_0$ will remain an intersection point, since the eventual linking chain will be minimal. This follows from Lemmas \ref{lem:linkingchaininvariantdetermines} and \ref{lem:cnew} below.}. Otherwise, blowing up $P_0$ introduces a component intersecting $\Gamma_{\s,K}$. Similarly for $P_{\infty}$. If $\s = \Rcal$ then $q^{-1}(P_{\infty})$ will never be an intersection point by \cite[Proposition~5.20]{DDMM18}. Since $\Gamma_{\s,L}$ has two intersection points with other components $Q_1$ and $Q_2$, either $q(Q_1) \neq q(Q_2)$, or $q(Q_1) = q(Q_2) \neq P_0$ (since $|q^{-1}(P_0)| = 1$). If $q(Q_1) \neq q(Q_2)$ then these are both intersection points with other components, hence $\Gamma_{\s,K}$ intersects at least 3 components at $P_{\infty},q(Q_1)$ and $q(Q_2)$ which are all distinct. If $q(Q_1) = q(Q_2) \neq P_0$ then $P_{\infty}, q(Q_1)$ and $P_0$ are distinct intersection points with other components. A similar argument works if $\s\neq\Rcal$.
\end{proof}

We are now able to prove our structure theorem (Theorem \ref{thm::structureofSNCmodel}).

\begin{proof}[Proof of Theorem \ref{thm::structureofSNCmodel}]
    First let us find which central components appear. Over $L$, by \cite[Theorem~8.5]{DDMM18}, we know there is a component for every principal, non-\"ubereven cluster, and we know the action of $\Gal(L/K)$ on these central components is the same as the action on the clusters. After taking the quotient by $\Gal(L/K)$, we get a component for every orbit of principal, non \"ubereven clusters. Similarly over $L$, by \cite[Theorem~8.5]{DDMM18}, we know there are two components for every \"ubereven cluster $\s$. These are swapped by Galois if and only if $\epsilon_{\s}=-1$. After taking the quotient this gives us two components for an \"ubereven orbit $X$ if $\epsilon_X=1$ and a single component if $\epsilon_X=-1.$ We call these components the central components. Showing which linking chains which appear is done similarly, using the information given in \cite[Theorem~8.5]{DDMM18}.
    
    To ensure these central components do in fact appear in the minimal SNC model, we must check that they cannot be blown down. Any central component $\Gamma_{X,K}\in\X_k$ is the image of $\Gamma_{\s,L}\in\Y_k$ for some $\s\in X$. A central component $\Gamma_{X,K}$ can only be blown down if $g(\Gamma_{X,K})=0$, and $(\Gamma_{X,K}\cdot \Gamma_{X,K})=-1$. However, by Proposition \ref{prop::centralnotexceptional}, any central component $\Gamma_{X,K}$ with $g(\Gamma_{X,K}) = 0$ and $(\Gamma_{X,K} \cdot \Gamma_{X,K}) = -1$ intersects at least three other components of the special fibre. Therefore, if $\Gamma_{X,K}$ were to be blown down, $\X_k$ would no longer be an SNC divisor. So $\Gamma_{X,K}$ must appear in the special fibre of the minimal SNC model. 
    \end{proof}

\begin{remark}
    Note that a linking chain can have length $0$ - this indicates an intersection between central components (in the case $X' < X$ both principal) or a singular central component (in the case where $X$ is principal and $X' < X$ is an orbit of twins).
\end{remark}

\subsection{A More Explicit Description}

Theorem \ref{thm::structureofSNCmodel} describes the structure of the special fibre, but says nothing about the multiplicity or genera of the components, or the action of Frobenius. The following theorems fill in these details. The first focuses on the central components, the second describes the chains of rational curves present in the special fibre, and the last gives the Frobenius action.

\begin{thm}[Central Components]
\label{thm::maincentral}
    Let $K$ and $C/K$ be as in Theorem \ref{thm::structureofSNCmodel}. Let $X$ be a principal orbit of clusters in $\Sigma_{C/K}$. If $X$ is not \"ubereven then $\Gamma_X$ has multiplicity $|X|e_X$ and genus $g(X)$. If $X$ is \"ubereven with $\epsilon_X = 1$ then $\Gamma_X^+$ and $\Gamma_X^-$ have multiplicity $|X|e_X$ and genus 0, and if $\epsilon_X = -1$ then $\Gamma_X$ has multiplicity $2|X|e_X$ and genus 0.
\end{thm}

\begin{proof}
    Let $X$ be a principal, non-\"ubereven orbit, and choose some $\s \in X$. Recall that $K_X$ is the minimal field extension of $K$ such that the clusters of $X$ are fixed by $\Gal(\overline{K}/K_X)$, and $L$ is the minimal field extension of $K$ such that $C$ is semistable over $L$. The image $\Gamma_{\s,K_X}$ of $\Gamma_{\s,L}$ after taking the quotient by $\Gal(L/K_X)$ has multiplicity $e_X$, since the action on $\Gamma_{\s,L}$ has multiplicity $e_X$ (by Lemma \ref{lem::degreeofaction}). There are $|X|$ such components, which are permuted by $\Gal(K_X/K)$ in the minimal SNC model of $C/K_X$. So, $\Gamma_X$ has multiplicity $|X|e_X$ by \cite[Fact~IV]{Lor90}. The multiplicities of components corresponding to \"ubereven clusters follows similarly, being careful to account for whether $\Gamma_{\s,L}^+$ and $\Gamma_{\s,L}^-$ are swapped by $\Gal(L/K)$ in the semistable model (which happens precisely when $\epsilon_{\s} = -1$).
    
    To find the genus of the central components, note that if $g(\Gamma_{\s,L})=0$ then $g(\Gamma_{X,K})=0$. So let us assume that $g(\Gamma_{\s,L})>0$. In this case, as mentioned in Remark \ref{rem::Cstilde}, $\Gamma_{\s,L}$ is isomorphic to the special fibre of the smooth model of $C_{\widetilde{\s}}$ over $L$. Furthermore, the action on $\Gamma_{\s,L}$ is the same as the action on $\Gamma_{\widetilde{\s},L}$. Hence, the genus of $\Gamma_{\s,K_X}$ is $g(X)$, and also the genus of $\Gamma_{X,K}$.
\end{proof}

\begin{thm}[Description of Chains]
    \label{thm::chaindescription}
    Let $K$ and $C/K$ be as in Theorem \ref{thm::structureofSNCmodel}. Let $X$ be a principal orbit of clusters with $e_X > 1$. Choose some $\s \in X$ of depth $d_{\s}$ with denominator $b_{\s}$. Then the central component(s) associated to $X$ are intersected transversely by the following sloped tails with parameters $(t_1,\mu)$ (writing $\Gamma_X = \Gamma_X^+ = \Gamma_X^-$ if $X$ is not \"ubereven):
    \begin{center}
    \begin{tabular}{|c|c|c|c|c|p{6.07cm}|}
    \hline
         \small{Name} & \small{From} & \small{Number} & $t_1$ & $\mu$ & \small{Condition}  \\ \hline
         $T_{\infty}$ & $\Gamma_X$ & $1$ & \small{$(g + 1)d_{\Rcal} - \lambda_{\Rcal}$} & $1$ & $X=\{\Rcal\}$, $\Rcal$ odd \\ \hline
         $T_{\infty}^{\pm}$ & $\Gamma_X^{\pm}$ & $2$ & $-d_{\Rcal}$ &$1$ & $X = \{\Rcal\}$, $\Rcal$ even, $\epsilon_{\Rcal} = 1$ \\ \hline
         $T_{\infty}$ & $\Gamma_X$ & $1$ & $-d_{\Rcal}$ & $2$ & $X = \{\Rcal\}$, $\Rcal$ even, $e_{\Rcal} > 2$, $\epsilon_{\Rcal}=-1$ \\ \hline
         $T_{y=0}$ & $\Gamma_X$ & $\frac{|\singletonsofs||X|}{b_X}$ & $-\lambda_X$ & $|X|b_X$ &  $|\s_{\mathrm{sing}}|\geq2$, and $e_X>b_X/|X|$  \\\hline
         $T_{x=0}$ & $\Gamma_X$ & $1$ & $-d_X$ & $2|X|$ & $X$ has no stable child, $\lambda_X \not \in \Z$, $e_X>2$ and either $g(X) > 0$ or $X$ is \"ubereven \\ \hline 
         $T_{x=0}^{\pm}$ & $\Gamma_X^{\pm}$ & $2$ & $-d_X$ & $|X|$ & $X$ has no stable child, $\lambda_X \in \Z$, and either $g(X) > 0$ or $X$ is \"ubereven \\ \hline
         $T_{(0,0)}$ & $\Gamma_X$ & $1$ & $-\lambda_X$ & $|X|$ & $X$ has a stable singleton or $g(X) = 0$, $X$ is not \"ubereven and $X$ has no proper stable odd child\\\hline
    \end{tabular}
    \end{center} 
    The  central  components  are  intersected  by  the  following  sloped  chains  of  rational  curves  with parameters $(t_2,t_2+\delta,\mu)$:
    \begin{center}
    \begin{tabular}{|c|c|c|c|c|}
    \hline
         \small{Name} & $t_2$ & $\delta$ & $\mu$ & \small{Condition} \\ \hline
         $L_{X,X'}$ & $-\lambda_X$ & $\delta_{X'}/2$& $|X|$ & $X' \leq X$ both principal, $X'$ odd \\ \hline
         $L_{X,X'}^+$ & $-d_X$ & $\delta_{X'}$ & $|X|$ & $X' \leq X$ both principal, $X'$ even with $\epsilon_{X'} = 1$ \\ \hline
         $L_{X,X'}^-$ & $-d_X$ & $\delta_{X'}$ & $|X|$ & $X' \leq X$ both principal, $X'$ even with $\epsilon_{X'} = 1$ \\ \hline
         $L_{X,X'}$ & $-d_X$ & $\delta_{X'}$ & $2|X|$ & $X' \leq X$ both principal, $X'$ even with $\epsilon_{X'} = -1$ \\ \hline
         $L_{X'}$ & $-d_X$ & $2\delta_{X'}$ & $|X|$ & $X$ principal, $X' \leq X$ orbit of twins, $\epsilon_{X'} = 1$ \\ \hline
         $T_{X'}$ & $-d_X$ & $\delta_{X'}+\frac{1}{\mu}$ & $2|X|$ & $X$ principal, $X' \leq X$ orbit of twins, $\epsilon_{X'} = -1$ \\ \hline
    \end{tabular}
    \end{center}
    If $\Rcal$ is not principal we get additional sloped chains with parameters $(t_2,t_2+\delta,\mu)$ as follows:
    \begin{center}
        \begin{tabular}{|c|c|c|c|p{7.3cm}|}
        \hline
         \small{Name} & $t_2$ & $\delta$ & $\mu$ & \small{Condition} \\
         \hline
         $L_{\Rcal}$ & $-d_{\Rcal}$ & $2\delta_{\s}$ & $1$ & $\Rcal$ a cotwin, $\s<\Rcal$ child of size $2g$, $v_K(c_f) \in 2\Z$ \\\hline
         $T_{\Rcal}$ & $-d_{\Rcal}$ & $\delta_{\s}+\frac{1}{\mu}$ & $2$ & $\Rcal$ a cotwin, $\s<\Rcal$ child of size $2g$, $v_K(c_f) \not \in 2\Z$ \\\hline
         $L_{\s_1,\s_2}$ & \small{$(g(\s_1) + 1)d_{\s_1}- \lambda_{\s_1}$} & $\frac{1}{2}\delta(\s_1,\s_2)$ &  $1$ & $\Rcal = \s_1 \sqcup \s_2$, $\s_i$ principal, odd, $e_{\mathcal{R}} = 1$ \\ \hline
         $L_X$ &  \small{$(g(\s_1) + 1)d_{\s_1}- \lambda_{\s_1}$} & $\frac{1}{2}\delta(\s_1,\s_2)$ & $2$ & $\Rcal = \s_1 \sqcup \s_2$, $X = \{\s_1,\s_2\}$ principal, odd orbit\\ \hline
         $L_{\s_1,\s_2}^+$ & $d_{\s_1}$ & $\delta(\s_1,\s_2)$ & $1$ & $\Rcal = \s_1 \sqcup \s_2$, $\s_i$ principal, even, $e_{\mathcal{R}} = 1$, $\epsilon_{\s_i} = 1$ \\ \hline
         $L_{\s_1,\s_2}^-$ & $d_{\s_1}$ & $\delta(\s_1,\s_2)$ & $1$ & $\Rcal = \s_1 \sqcup \s_2$, $\s_i$ principal, even, $e_{\mathcal{R}} = 1$, $\epsilon_{\s_i} = 1$ \\ \hline
         $L_{\s_1,\s_2}$ & $d_{\s_1}$ & $\delta(\s_1,\s_2)$ & $2$ & $\Rcal = \s_1 \sqcup \s_2$, $\s_i$ principal, even, $e_{\mathcal{R}} = 1$, $\epsilon_{\s_i} = -1$ \\ \hline
         $L_X^+$ & $d_{\s_1}$ & $\delta(\s_1,\s_2)$ & $2$ & $\Rcal = \s_1 \sqcup \s_2$, $X = \{\s_1,\s_2\}$ principal, even orbit, and $\epsilon_{\s_i} = 1$ \\ \hline
         $L_X^-$ & $d_{\s_1}$ & $\delta(\s_1,\s_2)$ & $2$ & $\Rcal = \s_1 \sqcup \s_2$, $X = \{\s_1,\s_2\}$ principal, even orbit, and $\epsilon_{\s_i} = 1$ \\ \hline
         $T_X$ & $d_{\s_1}$ & $\delta(\s_1,\s_2)+\frac{1}{\mu}$ & $4$ & $\Rcal = \s_1 \sqcup \s_2$, $X = \{\s_1,\s_2\}$ principal, even orbit, and $\epsilon_{\s_i} =-1$\\ \hline
         $L_{\tfrak}$ & $d_{\s}$ & $2\delta(\s,\tfrak)$ & $1$ & $\mathcal{R} = \s \sqcup \tfrak$, $\s$ principal even, $\tfrak$ twin, $\epsilon_{\tfrak} = 1$ \\ \hline 
         $T_{\tfrak}$ & $d_{\s}$ & $\delta(\s,\tfrak)+\frac{1}{\mu}$ & $2$ & $\mathcal{R} = \s \sqcup \tfrak$, $\s$ principal even, $\tfrak$ twin, $\epsilon_{\tfrak} = -1$\\
         \hline
         \end{tabular}
    \end{center}
    
    Finally, the crosses of any crossed tail have multiplicity $\frac{\mu}{2}$.
\end{thm}
\begin{proof}
    Postponed to Section \ref{subsec::proofofmain}.
\end{proof}

\begin{remark}
    If there is any confusion over which central components linking chains or tails intersect, the reader is urged to refer back to the tables in Theorem \ref{thm::structureofSNCmodel}. We have omitted this information from these tables due to spatial concerns.
\end{remark}

\begin{remark}
    Let $X$ be a principal orbit of clusters in $\Sigma_{C/K}$. As in Remark \ref{rem::upccorrespondence}, we make a comparison between the rational chains intersecting a central component, $\Gamma_X\in\X_k$ to the tails in the special fibre of the minimal SNC model $\X_{\widetilde{X}}$. This comparison makes sense when $g(\Gamma_{\s,L}) > 0$ for some $\s\in X$. The central component $\Gamma_X\in \X_k$ will have the same genus as the central component $\Gamma_{\widetilde{X}}\in\X_{\widetilde{X},k}$ and multiplicity multiplied by $|X|$. It will have the same tails (with all multiplicities multiplied by $|X|$) except these tails will make up part of the linking chains intersecting $\Gamma_X$ in the following cases:
    \begin{itemize}[leftmargin=*]
        \item If $X\neq\mathcal{R}$ and $P(X)$ is principal, an $\infty$-tail in $\X_{\widetilde{X},k}$ will form the uphill section of one of the linking chains $L_{P(X),X}^{\pm}$,
        \item If $X < \Rcal$ and $\Rcal$ is not principal, then any $\infty$-tail in $\X_{\widetilde{X},k}$ will form the uphill section of a chain: the linking chain between $\Gamma_{\s_1}$ and $\Gamma_{\s_2}$ if $\Rcal = \s_1 \sqcup \s_2$ and $X = \{\s_1\}$; the crossed tail if $\Rcal = \s_1 \sqcup \s_2$ and $X = \{\s_1,\s_2\}$; and the loop or crossed tail arising from $\Rcal$ if $\Rcal$ is a cotwin,
        \item a $(y=0)$-tail will form the downhill section of a linking chain $L_{X,X'}$ if there exists some $X'<X$, a non-trivial orbit of odd, principal children,
        \item a $(x=0)$-tail will form the downhill section of a linking chain $L_{X,X'}^{\pm}$ if there exists some $\{\s'\} = X'<X$, a stable even child,
        \item a $(0,0)$-tail will form the downhill section of a linking chain $L_{X,X'}$ if there exists some $\{\s'\} =X'<X$, a stable odd child.
    \end{itemize}
where again, all multiplcities are multiplied by $|X|$. 
\end{remark}

We finish with a description of the Frobenius action on the components of the minimal SNC model (or equivalently, on the dual graph).

\begin{thm}[Frobenius Action]
    \label{thm::frobaction}
    Let $K$ be a field, not necessarily with algebraically closed residue field, and let $C/K$ be a curve 
    with tame potentially semistable reduction and minimal SNC model $\X$ over $K^{\textrm{ur}}$. 
    Then the Frobenius automorphism, $\Frob$, acts on the components of $\X$ as:
    
    \begin{enumerate}
        \item $\Frob(\Gamma_X^{\pm}) = \Gamma_{\Frob(X)}^{\pm\epsilon_X(\Frob)}$,
        \item $\Frob(L_{X,\; X'}^{\pm}) = L_{\Frob(X),\; \Frob(X')}^{\pm\epsilon_{X'}(\Frob)}$,
        \item a loop $L_X$ is sent to $\epsilon_X(\Frob)L_{\Frob(X)}$, a crossed tail $T_X$ to $\epsilon_X(\Frob)T_{\Frob(X)}$,\footnote{$-L_X$ is same loop but with reversed orientation. $-T_X$ is the same crossed tail but with crosses swapped.}
        \item  and tails are permuted as $\Frob(T_{\infty}^{\pm}) = T_{\infty}^{\pm\epsilon_X(\Frob)}$, $\Frob(T_{x=0}^{\pm}) = T_{x=0}^{\pm1^{v(c_X)}}$, and $(y=0)$-tails are permuted as the corresponding roots of the cluster pictures are.
    \end{enumerate}
\end{thm}
\begin{proof}
    Let $C$ have semistable reduction over a Galois extension $L$ of $K$, and let $\Y$ be the minimal semistable model of $C$ over $L^{\textrm{ur}}$. Then $\Frob$ acts on the components of $\Y_k$ as required by \cite[Theorem~8.5]{DDMM18}. Let $\ZZ$ be the quotient of $\Y$ be $\Gal(L^{\textrm{ur}}/K^{\textrm{ur}})$. By considering $G$-invariant open affines, we see that the following square commutes:
    
    \begin{center}
    \begin{tikzcd}
        \Y \arrow[d, "q"] \arrow[r, "\Frob"] & \Y \arrow[d, "q"]\\
        \ZZ \arrow[r, "\Frob"] & \ZZ
    \end{tikzcd}
    \end{center}
    
    So $\Frob$ permutes the components of $\ZZ$ as required. Since all central components are components of $\ZZ$, this proves (i).
    
    It remains to show that, after resolving the singularities on $\ZZ$, Frobenius acts on the components as desired. Consider a single blow up of an ideal sheaf $\mathfrak{I}$ corresponding to an orbit of points under Frobenius. Denote the resulting scheme $\ZZ'$. The Frobenius automorphism on $\ZZ$ extends to an automorphism on $\ZZ'$, which must also be induced by Frobenius. Note that the exceptional components of $\ZZ'$ are permuted by Frobenius in the same way as the corresponding singularities of $\ZZ$ are. So it is sufficient to show that Frobenius acts on the singularities of $\ZZ$ as expected.
    
    The action on singularities on linking chains is determined by the action on the rest of the linking chain. The action on the linking chain is entirely determined by the action on the central components they link, except in the case that there are two linking chains between central components. In this case, they are swapped if and only if $\epsilon_X(\Frob)=-1$. This follows from \cite[Theorem~8.15]{DDMM18} and the commutative square above. This proves (ii). Loops and crossed tails can be dealt with similarly to prove (iii).
    
    If there are two infinity tails, the singularities they arise from are the images of two points at infinity of a component of $\Y_k$ (see the proof of Theorem \ref{thm::structureofSNCmodel}). Points at infinity of a component $\Gamma_{\s}$ of $\Y_k$ arising from a cluster $\s$ are swapped by Frobenius if and only if $\epsilon_{\s}(\Frob)=-1$. This proves the first condition of (iv). The singularities giving rise to $(y=0)$-tails are images of roots of $f(x)$, and those giving rise to $(x=0)$-tails are images of the points $(0, \pm \sqrt{c_X})$, hence (iv).
\end{proof}

\subsection{Proof of the Theorem \ref{thm::chaindescription}} 
\label{subsec::proofofmain}

To prove Theorem \ref{thm::chaindescription}, we will proceed by induction on two things: the number of proper clusters in $\Sigma_{C/K}$, and the degree $e=[L:K]$ of the minimal extension $L/K$ such that $C/L$ is semistable. The base cases for these are when $\Sigma_{C/K}$ consists of a single proper cluster (which is covered in Section \ref{sec::tpgr}, in particular Theorem \ref{thm::tpgr} and Proposition \ref{prop::singPointsAretsgs}), and when $C$ has semistable reduction over $K$ i.e. $e=1$ (which is covered in Section \ref{subsec::SemistableModels}). For our inductive hypothesis, suppose that for any hyperelliptic curve where the number of proper clusters in its cluster picture is strictly less than that of $C/K$, or the degree of an extension needed such that it is semistable is strictly less than that of $C$, we can completely determine the special fiber of its minimal SNC model. 

\subsubsection{$\Rcal$ Principal}
\label{subsubsec::proofprincipal}

We start by assuming that the top cluster $\Rcal$ is principal, and that it has a Galois invariant proper child $\s$. We will calculate the tails of $\Gamma_{\Rcal,K}^{\pm}$ and, if $\s$ is principal, $\Gamma_{\s,K}^{\pm}$. We will also calculate the linking chain(s) (or the chain arising from $\s$ if $\s$ is a twin) between them. This will be done by comparing the linking chain(s) to those in the special fibre of the minimal SNC model of another hyperelliptic curve over $K$, which we will call $\Cnew$. We will write $\Cnew:y^2=f^{\textrm{new}}(x)$, and denote the set of roots of $f^{\textrm{new}}$ over $\overline{K}$ by $\Rcalnew$. The curve $\Cnew/K$ is chosen so that $\Sigma_{\Cnew/K}$ has a unique proper cluster $\snew\neq\Rcalnew$, enabling us to apply the results of Section \ref{sec::upc}. We will then use induction to deduce the components of the model arising from the subclusters of $\s$. Finally, we will remove the assumption that $\s$ is Galois invariant. 

\begin{lemma}
    Let $\Rcal$ be principal and suppose that $e_{\Rcal} > 1$. The tails of the central component(s) associated to $\Rcal$ are as described in Theorem \ref{thm::chaindescription}.
\end{lemma}
\begin{proof}
    First suppose that $\Rcal$ is not \"ubereven. Let $\Y$ be the semistable model of $C/L$ and consider $\Gamma_{\Rcal,L}\subseteq \Y$. The stabiliser of $\Rcal$ has order $e_{\Rcal}$. Under the quotient map, a Galois orbit $T$ of points of $\Gamma_{\Rcal,L}$ gives rise to a singularity on $\Gamma_{\Rcal,K}$ lying on precisely one component of $\X_K$ if and only if $|T|<e_{\Rcal}$ and the points of $T$ lie on $\Gamma_{\Rcal,L}$ and no other components of $\Y_k$.
    
    Suppose that $g(\Gamma_{\Rcal,L}) = 0$. There are only two orbits with size less than $e_{\Rcal}$, which after an appropriate shift we can assume are at $0$ and $\infty$. The point at $\infty$ certainly lies on no other component of $\Y_k$ by \cite[Proposition~5.20]{DDMM18}, so $\Gamma_{\Rcal,K}$ will always have $\infty$-tails. By \cite[Proposition~5.20]{DDMM18}, the point $0$ lies on no other component of $\Y_k$ if and only if $\Rcal$ has no stable proper odd child. This is because if $\s<\Rcal$ is a stable odd child then $L_{\Rcal,\s}$ intersects $\Gamma_{\Rcal,L}$ at $0$, however no other linking chain to a child will ever intersect $\Gamma_{\Rcal,L}$ at $0$. Therefore $\Gamma_{\Rcal,K}$ will have a $(0,0)$-tail if and only if it has no stable proper odd child. The description of the tails follows.
    
    Suppose instead that $g(\Gamma_{\Rcal,L}) > 0$. The orbits of points on $\Gamma_{\Rcal,L}$ of size less than $e_{\Rcal}$ are the same as the small orbits $\Gamma_{\widetilde{\Rcal},L}$, which are described in Lemmas \ref{lem::inftyorbits} - \ref{lem::y0orbits}. To complete the description, we must calculate when these small orbits are intersection points with other components. We do this using the explicit description of the components of $\Y_k$ given in \cite[Proposition~5.20]{DDMM18}. From this, we can deduce that the points at $\infty$ never lie on a component other than $\Gamma_{\Rcal,L}$, $(y=0)$-orbits are intersection points if and only if $\s$ has a non-trivial orbit of proper odd children, $(x=0)$-orbits are intersection points if and only if $\s$ has a stable even child, and the $(0,0)$-orbit is an intersection point if and only if $\Rcal$ has a proper stable odd child.
    
    Now suppose $\Rcal$ is \"ubereven. Then each $\Gamma_{\Rcal,L}^{\pm}$ has two orbits of size less than $e_{\Rcal}$, which lie at their respective points at $0$ and $\infty$. The points at $\infty$ do not lie on any other components of $\Y_k$. The points at $0$ lie on no other component of $\Y_k$ if and only if $\Rcal$ has no stable child. So, $\Gamma_{\Rcal,K}^{\pm}$ has a $(x=0)$-tail if and only if $\Rcal$ does not have a stable child. The description of the tails follows. 
\end{proof}

\begin{lemma}
    Let $\s < \Rcal$ be a principal, Galois invariant cluster with $e_{\s}>1$. Then the tails intersecting the central component(s) assosciated to $\s$ are as described in Theorem \ref{thm::chaindescription}.
\end{lemma}
\begin{proof}
    The proof is similar to that of the previous lemma, noting that all of the orbits at infinity are the intersection points of $\Gamma^{\pm}_{\s,L}$ and the linking chain between $\Gamma^{\pm}_{\Rcal,L}$ and $\Gamma^{\pm}_{\s,L}$.
\end{proof}

Following is a technical lemma allowing us to compare the chain(s) appearing between $\Gamma_{\R,K}$ and $\Gamma_{\s,K}$ to those of a simpler curve $C^{\textrm{new}}$.

\begin{lemma}
\label{lem:linkingchaininvariantdetermines}
    Let $\s_1,\s_2$ be two Galois invariant principal clusters (resp. a principal cluster and a twin) such that either $\s_2 < \s_1$, or $\Rcal = \s_1 \sqcup \s_2$ is not principal. Then any linking chain between $\Gamma^{\pm}_{\s_1,K}$ and $\Gamma^{\pm}_{\s_2,K}$ (resp. the chain of rational curves arising from $\s_2$ intersecting $\Gamma_{\s_1,K}^{\pm}$) is determined entirely by $\lambda_{\s_i}\mod\Z$, the parity of $|\s_2|$, $d_{\s_i}$,  and when $\Rcal$ is not principal $d_{\Rcal}$.
\end{lemma}
\begin{proof}
    Assume that both $\s_i$ are principal, Galois invariant clusters. From Section \ref{subsec::QuotientsofModels}, we know that a linking chain between $\Gamma_{\s_1,K}^{\pm}$ and $\Gamma_{\s_2,K}^{\pm}$ is completely determined by the length and number of linking chains between $\Gamma_{\s_1,L}^{\pm}$ and $\Gamma_{\s_2,L}^{\pm}$, the order of the action of $\Gal(L/K)$ on any individual component of a linking chain between $\Gamma_{\s_1,L}^{\pm}$ and $\Gamma_{\s_2,L}^{\pm}$, and the nature of the singularities at the intersection points of components after taking the quotient. Recall from \cite[Theorem~8.5]{DDMM18} that there is one linking chain, say $\mathcal{C}$, between $\Gamma^{\pm}_{\s_1,L}$ and $\Gamma^{\pm}_{\s_2,L}$ if $\s_2$ is odd and two linking chains, say $\mathcal{C}^+$ and $\mathcal{C}^-$, if $\s_2$ is even. We will write $\mathcal{C} = \mathcal{C}^+ = \mathcal{C}^-$ if $\s_2$ is odd. The theorem \cite[Theorem~8.5]{DDMM18} tell us that the length of $\mathcal{C}^{\pm}$ is determined by $\delta(\s_1,\s_2)$, which is given in terms of $d_{\s_1}$ and $d_{\s_2}$ (and $d_{\Rcal}$ in the case where $\Rcal = \s_1\sqcup\s_2$ is not principal). 
    
    Let $P$ be an intersection point of components $E_1,E_2 \in \{\Gamma_{\s_1,L},\Gamma_{\s_2,L},\mathcal{C}^{\pm}\}$, and $\sigma_{E_i}$ the induced $\GK$ action on ${E_i}$ for a generator $\sigma\in\Gal(L/K)$. Suppose $\sigma_{E_1}^{a}$, and $\sigma_{E_2}^{b}$, generate the stabilisers of $P$ in $E_1$ and $E_2$ respectively. Then $q(P)$ is a tame cyclic quotient singularity with parameters 
    \begin{align*}
        n = \gcd(o(\sigma_{E_1}^a), o(\sigma_{E_2}^b)), \quad & m_1 = o(\sigma_{E_1}^a)/n, \quad m_2 = o(\sigma_{E_2}^b)/n, \textrm{ and }& r = \begin{cases}\frac{d_{E_1}^{-a} d_{E_2}^b}{n^2} & \s_2 \textrm{ even,} \\ \frac{\lambda_{E_1}^{-a} \lambda_{E_2}^b}{n^2} & \s_2 \textrm{ odd,} \end{cases}
    \end{align*}
    where $o(\tau)$ is the order of $\tau \in \Gal(L/K)$. In other words, the tame cyclic quotient singularity is determined entirely by the automorphisms on the $E_i$ and the parity of $\s_2$. Therefore, since the automorphisms on $E_i$ are determined entirely by the invariants in the statement of the theorem (by \cite[Theorem~6.2]{DDMM18}), we are done. The case where $\s_2$ is a twin follows similarly.
\end{proof}

For the following lemma we first need some notation. Recall that a child of $\s\in\Sigma_{C/K}$ is \textit{stable} if has the same stabiliser as $\s$. Let $\childrennotfixedofs{\s}$ denote the set of stable children of $\s$, and $\childrenfixedofs{\s}$ denote the set of unstable children of $\s$.

\begin{lemma}
\label{lem:cnew}
    Let $C/K$ be a hyperelliptic curve with $\Rcal$ principal, and let $\s< \Rcal$ be a Galois invariant proper child. We can construct a hyperelliptic curve, $\Cnew$, such that the cluster picture $\Sigma_{\Cnew}$ of $\Cnew$ consists of two proper clusters $\snew<\Rcalnew$, where $|\s|\equiv |\snew| \mod 2, d_{\Rcal} = d_{\Rcalnew}, d_{\s}=d_{\snew}$ and $\lambda_{\Rcal} - \lambda_{\Rcalnew}, \lambda_{\s} - \lambda_{\snew} \in \Z$.
\end{lemma}
\begin{proof}
     Let $\Cnew$ be the hyperelliptic curve over $K$ defined by $\Cnew:y^2=c_ff_{\Rcal}f_{\s},$ where
    \begin{align*}
        f_{\Rcal} & = \begin{cases} \displaystyle\prod_{\s\neq\oo\in\widetilde{\Rcal}} (x-z_{\oo}) & |\widetilde{\Rcal} \setminus \s| \geq 2,\\
        \pi_K^{|\widehat{\Rcal} \setminus \widetilde{\Rcal}| d_{\Rcal}}  \displaystyle\prod_{\s\neq\s'<\Rcal} (x-z_{\s'}) & \mathrm{otherwise},
    \end{cases}\\
        f_{\s} &= \begin{cases} \displaystyle\prod_{\oo\in\widetilde{\Rcal}} (x-{z_\oo}) & |\widetilde{\s}| \geq 2, \\
        \displaystyle\prod_{\oo\in \widetilde{\s}^{\mathrm{f}}}(x-z_{\oo})\displaystyle\prod_{\s'\in \childrennotfixedofs{\s}}(x-z_{\s'}) & |\widetilde{\s}| \leq 1 \textrm{ and } |\childrennotfixedofs{\s}| \textrm{ even},\\
        \displaystyle\prod_{\oo\in \widetilde{\s}^{\mathrm{f}}}(x-z_{\oo})\displaystyle\prod_{\s' \in \childrennotfixedofs{\s}}(x-z_{\s'})(x+z_{\s'}) & |\widetilde{\s}| \leq 1 \textrm{ and } |\childrennotfixedofs{\s}| \textrm{ odd.} \end{cases}
    \end{align*}
    It is clear that $\Sigma_{C_{\mathrm{new}}/K}$ consists of proper two clusters which we will call $\Rcalnew$ and $\snew$, where $\Rcal_{\mathrm{new}}$ consists of the roots of $f_{\Rcal}\cdot f_{\s}$, and $\snew$ consists of the roots of $f_{\s}$. It follows that $\snew<\Rcalnew$. It remains to check how the cluster invariants of $\Rcalnew$ and $\snew$ compare to those of $\Rcal$ and $\s$. Since any root in a cluster can be taken as its center, it is immediate that $d_{\Rcal} = d_{\Rcalnew}$ and $d_{\s} = d_{\snew}$. By comparing $\deg(f_{\s})$ to $|\s|$ we see that $|\s| \equiv |\snew| \mod 2$.
    
    It remains to check that $\lambda_{\Rcal}-\lambda_{\Rcalnew}, \lambda_{\s}-\lambda_{\snew} \in \Z$. Let us begin with the first. By construction, $\snew$ is odd if and only if $\s$ is. Therefore, if $|\widetilde{\Rcal}\setminus\s| \geq 2$ it follows that $\lambda_{\Rcalnew} = \lambda_{\Rcal}$. Else,
    \begin{align*}
        2(\lambda_{\Rcalnew} - \lambda_{\Rcal}) & =  v_K(c_f) + |\widehat{\Rcal}|d_{\Rcal} + |\widehat{\Rcal} \setminus \widetilde{\Rcal}|d_{\Rcal} - v_K(c_f) - |\widetilde{\Rcal}|d_{\Rcal} = 2|\widehat{\Rcal} \setminus \widetilde{\Rcal}|d_{\Rcal}.
    \end{align*}
    If $d_{\Rcal} \in \Z$, then clearly $\lambda_{\Rcalnew} - \lambda_{\Rcal} \in \Z$. Otherwise, $d_{\Rcal} \not \in \Z$. By Lemma \ref{lem::orbitssizeb}, the children of $\Rcal$ must lie in orbits of size $b_{\Rcal} > 1$. Therefore, any such orbit must be an orbit of even children of $\Rcal$, since $\s$ is fixed and there is at most one child not equal to $\s$. Hence, $|\widehat{\Rcal} \setminus \widetilde{\Rcal}|d_{\Rcal} \in \Z$, and so $\lambda_{\Rcalnew} - \lambda_{\Rcal} \in \Z$. It can be checked similarly that $\lambda_{\snew} - \lambda_{\s} \in \Z$. 
\end{proof}

By the above lemmas and Theorem \ref{thm:uniquepropchild}, we have proved the statements in Theorem \ref{thm::structureofSNCmodel} about the linking chain(s) between $\Gamma_{\s,K}^{\pm}$ and $\Gamma_{\Rcal,K}^{\pm}$ where $\s<\Rcal$ is a Galois invariant proper child. 

We now turn our focus to the components of $\X_k$ which arise from $\s$ and its subclusters. In order to do this, we construct another new hyperelliptic curve, which we shall call $C'$, given by 
    \begin{align}
    \label{eqn::cprime}
    C' : y^2 = c_f' \prod_{r \in \s} (x-r),\textrm{ where }
    c_f' = c_f \prod_{r \not\in \s} (z_{\s} - r).
    \end{align}
Note that $C'$ is also semistable over $L$, and let $\Y'$ be the semistable model of $C'$ over $L$. Comparing the cluster pictures of $C'$ and $C$, we see that the cluster picture $\Sigma_{C'}$ appears within the cluster picture $\Sigma_{C}$ of $C$. This is illustrated in Figure \ref{fig::C'comparedtoC}. In particular, $\s$ and all of its subclusters in $\Sigma_{C}$ are drawn in solid black in Figure \ref{fig::C}. These are exactly the clusters that make up $\Sigma_{C'}$, also shown in solid black. 
\begin{figure}[h]
\centering
\begin{subfigure}{0.5\textwidth}
  \centering
    \begin{tikzpicture}
    \draw (0,0) ellipse (0.4cm and 0.25cm);
    \path (0.1,0) -- node[auto=false]{\ldots} (1.5,0);
    \draw (1.6,0) ellipse (0.4cm and 0.25cm);
    
    \draw (0.8,0) ellipse (1.5cm and 0.7cm) node[below, yshift = -0.3cm, xshift = 1.4cm,font=\small]{$\s$};
    
    \draw (3,0)[dashed,gray] ellipse (0.4cm and 0.25cm);
    \path (3.1,0)[gray] -- node[auto=false]{\ldots} (4.5,0);
    \draw (4.6,0)[dashed,gray] ellipse (0.4cm and 0.25cm);

    \draw (2.1,0)[dashed,gray] ellipse (3.2cm and 1.1cm)node[below, yshift = -0.5cm, xshift = 3cm,font=\small]{$\Rcal$}; 
    \end{tikzpicture}
  \caption{Cluster picture $\Sigma_{C}$.}
  \label{fig::C}
\end{subfigure}%
\begin{subfigure}{.5\textwidth}
  \centering
  \begin{tikzpicture}
    \draw (0,0) ellipse (0.4cm and 0.25cm);
    \path (0.1,0) -- node[auto=false]{\ldots} (1.5,0);
    \draw (1.6,0) ellipse (0.4cm and 0.25cm);
    
    \draw (0.8,0) ellipse (1.5cm and 0.7cm);
    \end{tikzpicture}
  \caption{Cluster picture $\Sigma_{C'}$.}
  \label{fig::C'}
\end{subfigure}
\caption{Comparison of the cluster pictures of $C$ and $C'$}
\label{fig::C'comparedtoC}
\end{figure}
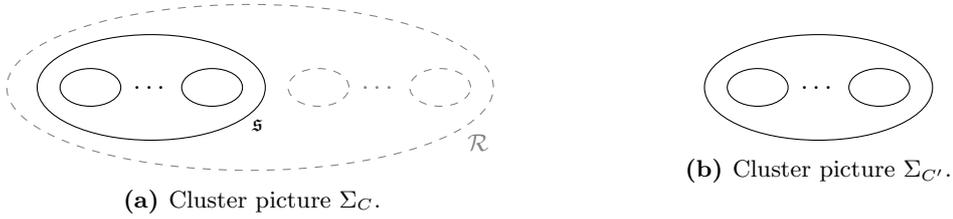

The leading coefficient of $C'$ has been chosen so that the corresponding clusters in $\Sigma_{C}$ and $\Sigma_{C'}$ have the same cluster invariants. Therefore, there is a closed immersion $\Y'_k \rightarrow\Y_k$ which commutes with the action of $\GK$. The existence of this immersion is illustrated in Figure \ref{fig::C'comparedtoCmodel}. We can see this by calculating the explicit equations of the components of $\Y'$ and using the explicit Galois action on these components given in \cite[Theorem~8.5]{DDMM18}. Therefore, this immersion also commutes with the quotient by $\Gal(L/K)$.

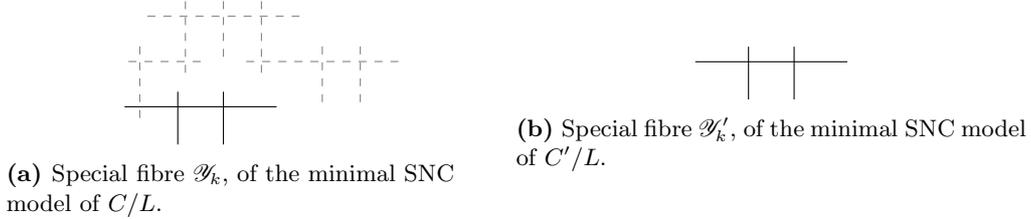
\begin{figure}[h]
\centering
\begin{subfigure}{0.45\textwidth}
  \centering
    \begin{tikzpicture}
    \draw(0,0)[dashed,gray] -- ++(2,0);
    \draw(0.5,0.2)[dashed,gray] -- ++(0,-1);
    \draw(0.7,-0.6)[dashed,gray] -- ++(-1,0);
    \draw(-0.1,-0.4)[dashed,gray] -- ++(0,-1);
    
    \draw(1,0.2)[dashed,gray] -- ++(0,-0.8);
    
    \draw(1.5,0.2)[dashed,gray] -- ++(0,-1);
    
    \draw(1.3,-0.6)[dashed,gray] -- ++(2,0);
    \draw(2.3,-0.4)[dashed,gray] -- ++(0,-0.8);
    \draw(2.8,-0.4)[dashed,gray] -- ++(0,-0.8);
    
    \draw(-0.3,-1.2) -- ++(2,0);
    \draw(0.4,-1) -- ++(0,-0.7);
    \draw(1,-1) -- ++(0,-0.7);
    
    \end{tikzpicture}
  \caption{Special fibre $\Y_k$, of the minimal SNC\\ model of $C/L$.}
  \label{fig::Cmodel}
\end{subfigure}%
\begin{subfigure}{0.45\textwidth}
  \centering
  \begin{tikzpicture}
    \draw(0.1,-1.2) -- ++(-2,0);
    \draw(-0.6,-1) -- ++(0,-0.7);
    \draw(-1.2,-1) -- ++(0,-0.7);
    \end{tikzpicture}
  \caption{Special fibre $\Y'_k$, of the minimal SNC model of $C'/L$.}
  \label{fig::C'model}
\end{subfigure}
\caption{Comparison of the special fibres of the minimal SNC models of $C$ and $C'$}
\label{fig::C'comparedtoCmodel}
\end{figure}

After taking this quotient by $\Gal(L/K)$, and performing any appropriate blow ups and blow downs, we obtain a closed immersion $\overline{\X'_k \setminus T_{\infty}} \rightarrow \X_k$, where $\X'$ is the minimal SNC model of $C'/K$ and $T_{\infty}$ is the set of infinity tails of $\X'_k$. We remove the infinity tails since in the small distance case (see Section \ref{subsec:upcsmalldistance}) the whole tails do not appear in $\X_k$. By our inductive hypothesis (since the number of proper clusters in $\Sigma_{C'}$ is strictly less than that in $\Sigma_C$), we can calculate $\X'_k$. This gives us a full description of the components of $\X_k$ which arise from the subclusters of $\s$. 

Finally let us remove the assumption that $\s$ is $\GK$ invariant. Let $X < \Rcal$ be a non-trivial orbit of children. Extend $K$ by degree $|X|$ to the field $K_X$, the minimal extension such that each cluster in $X$ is fixed by $\Gal(\overline{K}/K_X)$. By our inductive hypothesis (since $C/K_X$ needs an extension of degree strictly less than $C/K$ does in order to have semistable reduction), we can calculate the minimal SNC model of $C$ over $K_X$, which we denote $\X_X$. Since each cluster of $X$ is fixed by $\Gal(L/K_X)$, there is a divisor $D_{\s}$ corresponding to every cluster $\s \in X$ and all of the subclusters of $\s$. Let $D_X=\bigcup_{\s\in X}D_{\s}$ be the union of these divisors. Since $\Gal(K_X/K)$ simply permutes these divisors, the quotient by $\Gal(K_X/K)$ is an \'etale morphism, and the image of $D_X$ consists of precisely the same components as $D_{\s}$ for some $\s \in X$, but with all the multiplcities multiplied by $|X|$. See Figure \ref{fig::orbitsofdivisors} for an illustration. This concludes the proof when $\Rcal$ is principal.
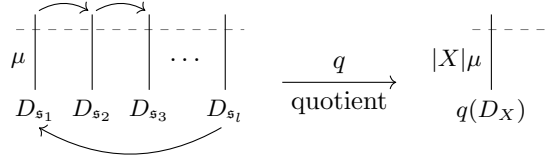
\begin{figure}[h]
\centering
  \centering
    \begin{tikzpicture}
    \draw(0,0)[dashed,gray] --  ++(3,0);
    \draw(0.25,0.2) --node[left,font=\small,yshift=-0.1cm] {$\mu$} node[below,font=\small,yshift=-0.5cm] {$D_{\s_1}$}++(0,-1);
    \draw(1,0.2) -- node[below,font=\small,yshift=-0.5cm] {$D_{\s_2}$} ++(0,-1);
    \draw(1.75,0.2) -- node[below,font=\small,yshift=-0.5cm] {$D_{\s_3}$} ++(0,-1);
    \path (1.75,-0.4) -- node[auto=false]{\ldots} (2.75,-0.4);
    \draw(2.75,0.2) -- node[below,font=\small,yshift=-0.5cm] {$D_{\s_l}$} ++(0,-1);
    \draw[->](0.3,0.25) to [bend left] (0.95,0.25);
    \draw[->](1.05,0.25) to [bend left] (1.7,0.25);
    \draw[->](2.7,-1.3) to [bend left] (0.3,-1.3);
    
    \draw[->](3.5,-0.7) to node[above]{$q$} node[below]{quotient} (5,-0.7);
    
    \draw(6,0)[dashed,gray] --  ++(1,0);
    \draw(6.25,0.2) --node[left,font=\small,yshift=-0.1cm] {$|X|\mu$} node[below,font=\small,yshift=-0.5cm] {$q(D_{X})$}++(0,-1);
    \end{tikzpicture}
  \caption{Divisors $D_{\s_i}$, where $X=\{\s_1\dots,\s_l\}$, are permuted by $\Gal(K_X/K)$. After taking the quotient the image of $D_X=\bigcup_{i=1}^lD_{\s_i}$ consists of the components of $D_{\s_i}$ but where a component of multiplicity $\mu$ in $D_{\s_i}$ now has multiplicity $|X|\mu$.}
  \label{fig::orbitsofdivisors}
\end{figure}

\vspace{-10px}
\subsubsection{$\Rcal$ not principal}
\label{subsubsec::proofnotprincipal}

Now suppose that $\mathcal{R}$ is not principal. If $\Rcal$ is a cotwin, then the contribution to the special fibre of the minimal SNC model from $\Rcal$ can be deduced using Remark \ref{rem::upccorrespondence} and Lemmas \ref{lem:linkingchaininvariantdetermines} and \ref{lem:cnew}. The contribution of $\s<\Rcal$, the child of size $2g$, can be calculated by induction using a curve $C'$ as in (\ref{eqn::cprime}) above. 

If $\Rcal$ is not principal and not a cotwin then $\Rcal$ is even and the union of two proper children. In this case, we will write $\Rcal=\s_1\sqcup\s_2$. Here the $s_i$ are either fixed or swapped by $\GK$. We will deal with the case when $s_i$ are swapped at the end of this section, so for now suppose that both $s_i$ are fixed by $\GK$. The first of these lemmas shows that there is a M\"obius transform taking a certain class of curves with $\Rcal$ not principal to the curves we studied in Section \ref{sec::upc}.

\begin{lemma}
\label{lem::nonprincmobiustransform}
Let $C/K$ be a hyperelliptic curve with cluster picture $\Sigma_{C/K}$, and set of roots $\Rcal$.
    \begin{enumerate}
    \item Let $\s\in\Sigma_{C/K}$ be a cluster with centre $z_{\s}$. Write every root $r \in \s$ as $r = z_{\s} + r_h$, where $v_K(r_h) \geq d_{\s}$. Then there exists at most one $r\in\s$ such that $v_K(r_h) > d_{\s}$.
    \item If $\mathcal{R} = \s_1\sqcup \s_2$ with $d_{\Rcal} \geq 0$, where $\s_1$ and $\s_2$ are both fixed by $Gal(L/K)$, have no proper children, and $z_{\s_1} = 0$. Then the M\"obius transform $\psi : r \mapsto \frac{1}{r}$ takes $C$ to a new curve $C_M$ which has cluster picture $\Sigma_M = \{\Rcal_M=\s_{1,M}, \s_{2,M}\}$, with $\s_{1,M} = \{\frac{1}{r} : 0 \neq r \in \s_1\}$, $\s_{2,M} = \{\frac{1}{r} : r \in \s_2\}$, $d_{\s_{1,M}} = - d_{\s_1}$ and $d_{\s_{2,M}} = d_{\s_2} - 2d_{\Rcal}$.  
    \end{enumerate}
\end{lemma}
\begin{proof}
(i) Suppose there are two roots $r$ and $r'$ such that $v_K(r_h),v_K(r'_h)>d_{\s}$. Then $d_{\s} = v_K(r - r') = v_K(r_h - r'_h) \geq \min(v_K(r_h),v_K(r'_h)) > d_{\s}$.

(ii) Since $z_{\s_1} = 0$, we have that $v_K(r) = d_{\s_1}$ for any $0 \neq r \in \s_1$. Note also that, $v_K(z_{\s_2}) = d_{\Rcal}$, hence $v_K(r) = d_{\Rcal}$ for any $r \in \s_2$. The statement then follows from the fact that $v_K\left(\frac{1}{x} - \frac{1}{y}\right) = v_K(x-y) - v_K(x) - v_K(y)$.
\end{proof}

\begin{remark}
    Note that $\delta_{\s_{1,M}} = \delta_{\s_1} + \delta_{\s_2}$, $\lambda_{\s_{1,M}} = \lambda_{\s_1} - (g(\s) + 1)d_{\s}$ and $\lambda_{\s_2} - \lambda_{\s_{2,M}} = (|s_1|-|s_2|)d_{\Rcal} \in 2\Z$.
\end{remark}

The next lemma is analogous to Lemma \ref{lem:cnew}, it gives us the existence of some new curve, which we will again call $\Cnew$, to which we can apply Lemma \ref{lem::nonprincmobiustransform}. This will allow us to calculate the linking chain(s) between $\Gamma_{\s_1}^{\pm}$ and $\Gamma_{\s_2}^{\pm}$, by using Lemma \ref{lem:linkingchaininvariantdetermines}.

\begin{lemma}
\label{lem::unionoftwoclustersnewinvars}
    Let $\Rcal = \s_1\sqcup\s_2$ with $\s_i$ both fixed by Galois. Then there exists a hyperelliptic curve $\Cnew:y^2=f^{\mathrm{new}}(x)$ whose set of roots of $f^{\mathrm{new}}$ we denote by $\Rcalnew$, such that $\Rcalnew=\snew_1\sqcup \snew_2$, where $\snew_i$ has no proper children, $|\s_i| - |\snew_i|\in 2\Z$, $d_{\s_i} = d_{\snew_i}$ and $\lambda_{\s_i} - \lambda_{\snew_i} \in \Z$ for $i=1,2$.
\end{lemma}
\begin{proof}
For $i=1,2$ define
    \begin{align*}
        f_{\s_i} = \begin{cases} \displaystyle\prod_{\oo\in \widetilde{\s_i}} (x-{z_\oo}) & g(\Gamma_{\s_i,L})>0,\\
        \displaystyle\prod_{\oo \in \childrenfixedofs{\s_i}}(x-z_{\oo})\displaystyle\prod_{\s'\in \childrennotfixedofs{\s_i}}(x-z_{\s'}) & g(\Gamma_{\s_i,L}) = 0 \textrm{ and } |\childrennotfixedofs{\s_i}| \textrm{ even},\\
        \displaystyle\prod_{\oo \in \childrenfixedofs{\s_i}}(x-z_{\oo})\displaystyle\prod_{\s' \in \childrennotfixedofs{\s_i}}(x-z_{\s'})(x+z_{\s'}) & g(\Gamma_{\s_i,L}) = 0 \textrm{ and } |\childrennotfixedofs{\s_i}| \textrm{ odd}. \end{cases}
    \end{align*}
    Let $f^{\mathrm{new}}=c_f f_{\s_1}f_{\s_2}$, so $\Cnew: y^2= c_f f_{\s_1} f_{\s_2}$. Proving this satisfies the conditions in the statement of this lemma is similar to the proof of Lemma \ref{lem:cnew}.
\end{proof}

So, if $\Rcal$ is not principal and a union of two clusters $s_i$ which are fixed by $\GK$ then, by Lemma \ref{lem::unionoftwoclustersnewinvars}, Lemma \ref{lem:linkingchaininvariantdetermines}, and Lemma \ref{lem::nonprincmobiustransform}, we know now the linking chain(s) between $\Gamma_{\s_1}^{\pm}$ and $\Gamma_{\s_2}^{\pm}$. We can calculate the components associated to $\s_i$ and its subclusters by induction, constructing a curve as in (\ref{eqn::cprime}). Therefore this gives us the full special fibre of minimal SNC model of $C/K$ when $\Rcal = \s_1 \sqcup \s_2$ is not principal and $\s_i$ are fixed by Galois. 

It remains to consider the case when $\Rcal=s_1\sqcup \s_2$ is not principal and $\s_i$ are swapped by Galois. This is solved by extending the field $K$ to $K_X$, an extension of degree two. Here, $C/K_X$ has a non principal top cluster $\Rcal'=\s_1'\sqcup\s_2'$, where $s_i'$ are both proper clusters, and are fixed by $\Gal(\overline{K}/K_X)$. So we can apply the above lemmas to find the special fibre of the minimal SNC model of $C/K_X$. Taking the quotient by $\Gal(K_X/K)$, which we know how to do by Section \ref{subsec::QuotientsofModels}, gives the special fibre of the minimal SNC model of $C/K$. This completes the cases when $\Rcal$ is not principal.

\begin{proof}[Proof of Theorem \ref{thm::main1}]
Combining the results proved in the rest of this section proves this.   
\end{proof}

Recall that, in Section \ref{sec::intro} we assumed that $\Rcal$ was principal, and gave some examples. We conclude with a couple of additional examples of when $\Rcal$ is not principal. Let $K=\Qur$

\begin{eg}
Let $C/K$ be the hyperelliptic curve given by $C:y^2=\left((x^2-p)^2+p^{4}\right)\left((x-1)^2-p^3\right)$. Note that $\tfrak_1$ and $\tfrak_2$ are swapped by $\GK$ and denote their orbit by $X$. This is a hyperelliptic curve of Namikawa-Ueno type $\mathrm{II}_{2-4}$  as in \cite[p.~183]{NU73}. 
Note $\s$ is \"ubereven and $\epsilon_{\s}=1$, hence $\s$ gives rise to two components; $X$ is an orbit of twins with $\epsilon_X = 1$, so gives rise to a linking chain, and $\Rcal$ is a cotwin (Definition \ref{def:cotwin}) so gives rise to a linking chain. Also $e_{\s} = 2$ so $\Gamma_{\s}^{\pm}$ are both intersected by tails. 
\begin{figure}[ht]
\centering
\begin{subfigure}{0.5\textwidth}
  \centering
    \begin{tikzpicture}
    \fill (0,0) circle (1.5pt);
    \fill (0.25,0) circle (1.5pt);
    \fill (1,0) circle (1.5pt);
    \fill (1.25,0) circle (1.5pt);
    \fill (2.5,0) circle (1.5pt);
    \fill (2.75,0) circle (1.5pt);

    \draw (0.125,0)
    ellipse (0.4cm and 0.25cm) node[below, yshift = -0.1cm, xshift = -0.3cm, font=\small]{$\tfrak_1$}node[above, yshift = 0cm, xshift = 0.5cm, font=\small]{$\frac{3}{2}$};
    
    \draw (1.125,0)
    ellipse (0.4cm and 0.25cm) node[below, yshift = -0.1cm, xshift = -0.3cm, font=\small]{$\tfrak_2$};
    
    \draw (0.625,0)
    ellipse (1.2cm and 0.7cm) node[below, yshift = -0.2cm, xshift = 1.2cm, font=\small]{$\s$}node[above, yshift = 0cm, xshift = 1.3cm, font=\small]{$\frac{1}{2}$};
    
    \draw (2.625,0)
    ellipse (0.4cm and 0.25cm) node[below, yshift = -0.1cm, xshift = 0.4cm, font=\small]{$\tfrak_3$}node[above, yshift = 0cm, xshift = 0.5cm, font=\small]{$\frac{3}{2}$};

    \draw (1.275,0) ellipse (2.2cm and 1.1cm) node[below, yshift = -0.5cm, xshift = 2.1cm,font=\small]{$\Rcal$} node[above, yshift = 0.5cm, xshift = 2.1cm,font=\small]{$0$};
    \end{tikzpicture}
  \caption{Cluster picture $\Sigma_{C/K}$.}
  \label{fig::eg1clusterpic}
\end{subfigure}%
\begin{subfigure}{.5\textwidth}
  \centering
  \begin{tikzpicture}
    \draw (-0.45,0)[line width = 0.5mm
    ] -- node[above, xshift=-0.6cm, font=\small]{$2$} node[left, font=\small, xshift=-0.8cm] {$\Gamma_{\s}^+$} ++ (1.75,0);
    \draw (1.7,0)[line width = 0.5mm
    ] -- node[above, xshift=0.6cm, font=\small] {$2$}  node[right, font=\small, xshift=0.8cm] {$\Gamma_{\s}^-$} ++ (1.75,0);
    \draw (0.25,0.2)
    -- node[left, font=\small] {$1$} ++ (0,-1);
    \draw (2.75,0.2)
    -- node[left, font=\small] {$1$} ++ (0,-1);
    
    \draw (1.15,0.2)
    -- node[left, font=\small] {$2$} ++ (0,-1);
    \draw (0.75,-0.2)
    -- node[left, font=\small] {$1$} ++ (0,1);
    
    \draw (1.85,0.2)
    -- node[right, font=\small] {$2$} ++ (0,-1);
    \draw (2.25,-0.2)
    -- node[right, font=\small] {$1$} ++ (0,1);
    
    \draw(0.95,-0.6)
    -- node[below, font=\small] {$2$} ++ (1.1,0);
    
    \node at (1.5,-1.2)[
    font=\small]{$L_{X}$};
    
    \draw(0.6,0.5)
    -- node[above, font=\small] {$1$} ++ (1,0.8);
    \draw(1.4,1.3)
    -- node[above, font=\small] {$1$} ++ (1,-0.8);
    
    \node at (1.5,1.5)[
    font=\small]{$L_{\tfrak_3}$};
    \end{tikzpicture}
  \caption{Special fibre of the minimal SNC model of $C/K$.}
  \label{fig::eg1model}
\end{subfigure}
\caption{$C:y^2=((x^2-p)^2+p^{4})((x-1)^2-p^3)$ over $K=\Qur$.}
\label{fig::eg1}
\end{figure}
\end{eg}
\begin{eg}
Let $C/K$ be the hyperelliptic curve given by $C:y^2=x(x^2-p)\left((x-1)^3-p^2\right)$  This is a curve of Namikawa-Ueno type $\mathrm{IV-III}-{0}$ as in \cite[p.~167]{NU73}. 
Observe that $\Rcal$ is not principal so gives rise to a linking chain between $\Gamma_{\s_1}$ and $\Gamma_{\s_2}$. Note that the special fibre here is the same as in Example \ref{eg::compareegs}, and there is in fact a M\"obius transform between the two curves.
\begin{figure}[ht]
\centering
\begin{subfigure}{0.5\textwidth}
  \centering
    \begin{tikzpicture}
    \fill (0,0) circle (1.5pt);
    \fill (0.25,0) circle (1.5pt);
    \fill (0.5,0) circle (1.5pt);
    \fill (1.5,0) circle (1.5pt);
    \fill (1.75,0) circle (1.5pt);
    \fill (2,0) circle (1.5pt);

    \draw (0.25,0)
    ellipse (0.6cm and 0.25cm) node[below, yshift = -0.1cm, xshift = 0.65cm, font=\small]{$\s_1$}node[above, yshift = 0cm, xshift = 0.65cm, font=\small]{$\frac{1}{2}$};
    
    \draw (1.75,0)
    ellipse (0.6cm and 0.25cm) node[below, yshift = -0.1cm, xshift = 0.65cm, font=\small]{$\s_2$}node[above, yshift = 0cm, xshift = 0.65cm, font=\small]{$\frac{2}{3}$};

    \draw (1.1,0) ellipse (1.8cm and 0.9cm) node[below, yshift = -0.4cm, xshift = 1.75cm,font=\small]{$\Rcal$} node[above, yshift = 0.4cm, xshift = 1.75cm,font=\small]{$0$};
    \end{tikzpicture}
  \caption{Cluster picture $\Sigma_{C/K}$.}
  \label{fig::egfrobclusterpic}
\end{subfigure}%
\begin{subfigure}{.5\textwidth}
  \centering
  \begin{tikzpicture}
    \draw (0,0)[line width = 0.5mm] -- node[above, font=\small] {$4$} node[right, font=\small, xshift=1.2cm] {$\Gamma_{\s_1}$}  ++ (2.5,0);
    \draw (0.25,-0.2) -- node[left, font=\small, xshift=0.05cm, yshift=0.1cm] {$1$} ++ (0,0.8);
    \draw (0.75,-0.2) -- node[left, font=\small, xshift=0.05cm, yshift=0.1cm] {$2$} ++ (0,0.8);
    \draw (2.25,0.2) -- node[left, font=\small, xshift=0.05cm] {$1$} ++ (0,-1);
    \draw (0,-0.6)[line width = 0.5mm] -- node[below, font=\small] {$3$}  node[right, font=\small, xshift=1.2cm] {$\Gamma_{\s_2}$}  ++ (2.5,0);
    \draw (0.25,-0.4) -- node[left, font=\small, xshift=0.05cm, yshift=-0.1cm] {$1$} ++ (0,-0.8);
    \draw (0.75,-0.4) -- node[left, font=\small, xshift=0.05cm, yshift=-0.1cm] {$1$} ++ (0,-0.8);
    \end{tikzpicture}
  \caption{Special fibre of the minimal SNC model of $C/K$.}
  \label{fig::egfrobmodel}
\end{subfigure}
\caption{$C:y^2=x(x^2-p)((x-1)^3-p^2)$ over $K=\Qur$.}
\label{fig::egfrob}
\end{figure}

\end{eg}

\end{document}